\documentclass[aap]{imsart}
\pdfoutput=1
\usepackage[]{amsmath,amssymb,amsfonts,latexsym,amsthm,enumerate, comment}
\usepackage{dsfont,bbm,mathrsfs,tikz,stmaryrd}
\usepackage{enumerate,graphicx,paralist}
\usepackage{booktabs,longtable,multirow}
\usepackage[noadjust]{cite}
\usepackage[footnotesize]{caption}
\usetikzlibrary{arrows,calc}

\numberwithin{equation}{section}
\RequirePackage[colorlinks,citecolor=blue,urlcolor=blue]{hyperref}
\newtheorem{maintheorem}{Theorem}

\newtheorem{theorem}{Theorem}[section]
\newtheorem*{theorem*}{Theorem}
\newtheorem{lemma}[theorem]{Lemma}
\newtheorem{claim}[theorem]{Claim}
\newtheorem{proposition}[theorem]{Proposition}

\newtheorem{corollary}[theorem]{Corollary}

\theoremstyle{definition}{

\newtheorem*{definition*}{Definition}

\newtheorem*{question*}{Question}
\newtheorem*{example*}{Example}
\newtheorem*{examples*}{Examples}
\newtheorem{remark}[theorem]{Remark}
\newtheorem*{remark*}{Remark}

}
\newcommand{\cA}{{\mathcal{A}}}
\newcommand{\cB}{{\mathcal{B}}}
\newcommand{\cC}{{\mathcal{C}}}

\newcommand{\cF}{{\mathcal{F}}}
\newcommand{\cG}{{\mathcal{G}}}

\newcommand{\cI}{{\mathcal{I}}}
\newcommand{\cL}{{\mathcal{L}}}
\newcommand{\cN}{{\mathcal{N}}}
\newcommand{\cP}{{\mathcal{P}}}
\newcommand{\cS}{{\mathcal{S}}}

\newcommand{\cU}{{\mathcal{U}}}

\newcommand{\fM}{{\mathfrak{M}}}
\newcommand{\Z}{\mathbb{Z}}
\newcommand{\R}{\mathbb{R}}

\newcommand{\N}{\mathbb{N}}
\renewcommand{\P}{\mathbb{P}}
\newcommand{\Q}{\mathbb{Q}}
\newcommand{\E}{\mathbb E}

\newcommand{\one}{\mathbbm{1}}
\renewcommand{\d}{\,\mathrm{d}}
\newcommand{\given}{\;\big|\;}

\newcommand\abs[1]{\left|#1\right|}
\newcommand\pth[1]{\left ( #1 \right )}
\newcommand\brak[1]{\left [ #1 \right ]}

\newcommand{\floor}[1]{\left\lfloor#1\right\rfloor}

\newcommand{\usim}{\mathop{\sim}_{\scriptscriptstyle(\mathrm{u})}}
\newcommand{\uasymp}{\mathop{\asymp}_{\scriptscriptstyle(\mathrm{u})}}
\newcommand{\ulesssim}{\mathop{\lesssim}_{\scriptscriptstyle(\mathrm{u})}}

\newcommand{\tl}{\tilde{t}}
\newcommand{\tll}{\tilde{t}-\ell}
\newcommand{\zint}{[L^{1/6}, L^{2/3}]}
\newcommand\UB[2]{\overline{\cB}_{#1}^{#2}}
\newcommand\LB[2]{\underline{\cB}_{#1}^{#2}}
\newcommand{\B}{\mathcal{B}^{\Bumpeq}}
\newcommand{\bB}{\overline{\mathcal{B}}^{\Bumpeq}}
\newcommand{\x}{\mathbf{x}}
\newcommand{\y}{\mathbf{y}}

\newcommand{\fc}{\mathfrak c}
\newcommand{\fT}{\mathfrak T}
\newcommand{\win}{{\normalfont\texttt{win}}}
\newcommand{\bT}{{\bar{\Theta}^*}}

\newcommand{\rx}{r_{1,\x}}
\newcommand{\ry}{r_{2,\y}}
\newcommand{\fo}{\mathfrak{o}}
\newcommand{\ep}{\mathfrak{b}}
\newcommand{\bep}{\bar{\mathfrak{b}}}
\newcommand{\fm}{\mathfrak{m}}

\hypersetup{
    colorlinks=true, 
    linkcolor=blue,  
    urlcolor=magenta, 
}

\begin{document}

\begin{frontmatter}
\title{The maximum of branching Brownian motion in~$\R^{\lowercase{d}}$}
\runtitle{The maximum of BBM in~{$\R^d$}}

\begin{aug}
  \author{\fnms{Yujin H.} \snm{Kim}\ead[label=e1]{yujin.kim@cims.nyu.edu}},\; 
  \author{\fnms{Eyal}  \snm{Lubetzky}\corref{}\ead[label=e2]{eyal@courant.nyu.edu}}
  \and
  \author{\fnms{Ofer} \snm{Zeitouni}\ead[label=e3]{ofer.zeitouni@weizmann.ac.il}}

  \runauthor{Y.H. Kim, E. Lubetzky and O. Zeitouni}

  \affiliation{New York University and Weizmann Institute}

 \address{Yujin H. Kim\\
Courant Institute of Mathematical Sciences
\\ New York University\\
New York, NY 10012, USA.\\
\printead{e1}}

  \address{
Eyal Lubetzky\\
Courant Institute of Mathematical Sciences
\\ New York University\\
New York, NY 10012, USA.\\
\printead{e2}}

  \address{Ofer Zeitouni\\
Department of Mathematics\\
Weizmann Institute of Science\\
Rehovot 76100, Israel.\\
\printead{e3}}
 \end{aug}

\begin{abstract}
We show that in branching Brownian motion (BBM) in $\R^d$, $d\geq 2$, the law of $R_t^*$, the maximum distance of a particle from the origin at time $t$, converges as $t\to\infty$ to the law of a randomly shifted Gumbel random variable. 
\end{abstract}



\end{frontmatter}

\section{Introduction}

Let $\{X_t^{(v)}\}_{t\geq 0,v\in \cN_t}$ 
denote $d$-dimensional branching Brownian motion (BBM), where $d \geq 1$; here, $\cN_t$ denotes the particles existing at time $t$ (a formal definition 
of the BBM model appears in Section \ref{sec-subsecdef} below). 
For $v\in \cN_t$, let  $R_t^{(v)} := \| X_t^{(v)} \|$ denote the $\ell^2$-modulus of
the location of 
the particle $v$ at time $t$, and set $R_t^* := \sup_{v \in \cN_t} R_t^{(v)} $.
Further, define 
\begin{equation}
  \label{def:m_t-c_d}
   \alpha_d := (d-1)/2, \quad m_t(d) := \sqrt{2} t + \fc_d \log t, 
  \quad \mbox{\rm where}\quad  
  \fc_d := \frac{d-4}{2\sqrt{2}}\,.
\end{equation}
(When the dimension $d$ is clear from the context, 
we omit it from the notation, writing e.g.\ $m_t$ for $m_t(d)$, etc.) 

When $d=1$, Bramson~\cite{Bramson83} proved the convergence in distribution 
of $\max_{v\in \cN_t} X_t^{(v)}-m_t(1)$, and the limit was identified by Lalley and Selke~\cite{LS87} to be the limit of a certain \textit{derivative martingale}. It is not hard to deduce from their results and methods
(see, e.g.,~\cite[Thm.~1.1]{SBM20})
that, when $d=1$,
\begin{equation}
  \label{eq-limd=1}
    \lim_{t\to \infty} \P(R_t^* - m_t(1) \leq y) = \E\big[ \exp ( - y Z)\big]\,,
\end{equation}
where $Z$ is an appropriate random variable (determined in terms of the limit of two 
a-priori dependent derivative martingales.)

We are interested in the case $d\geq 2$, where far less is known. See Figure~\ref{fig:my_label} for a simulation of BBM for $d=2$. Mallein~\cite{Mallein15} proved that for~$d\geq 2$, the collection $\{R_t^*-m_t(d)\}_{t>0}$ is tight, and that
there exists some $C>0$ such that 
for any $t \geq 1$ and  $y \in [1,t^{1/2}]$,
\begin{align}
    \frac{ye^{-\sqrt{2}y}}{C}
    \leq 
    \P( R_t^* \geq m_t(d) + y) 
    \leq 
    C y e^{-\sqrt{2}y}\,.
    \label{eq.mallein}
\end{align}
Mallein's result contrasts with  the classical G\"{a}rtner
propagation estimate for multidimensional KPP~\cite{Gartner82},
which gives a different constant in front of the
logarithmic correction term in $m_t$.

\begin{figure}
    \centering
    \includegraphics[width=0.6\textwidth]{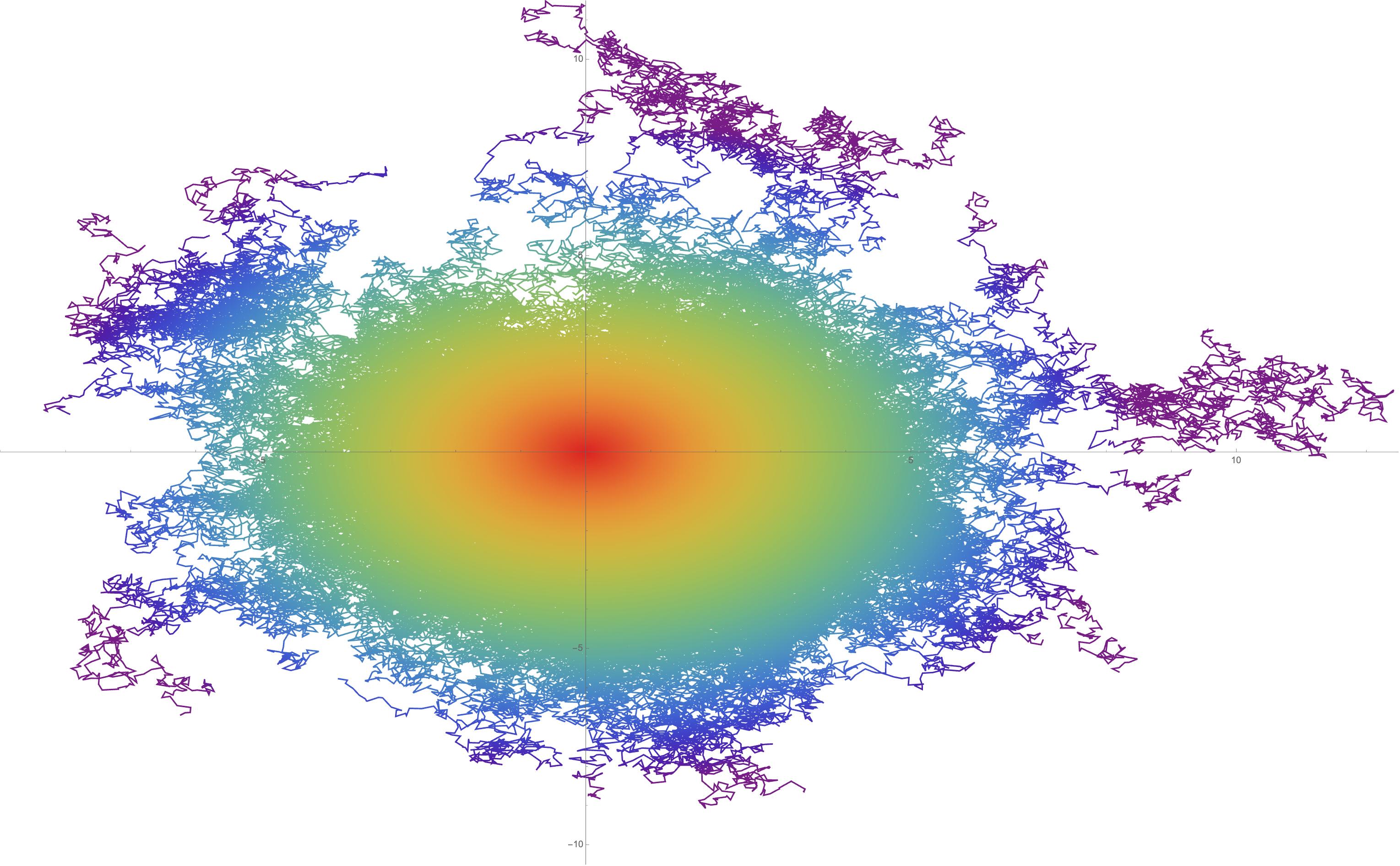}
    \raisebox{0.4in}{\includegraphics[width=0.39\textwidth]{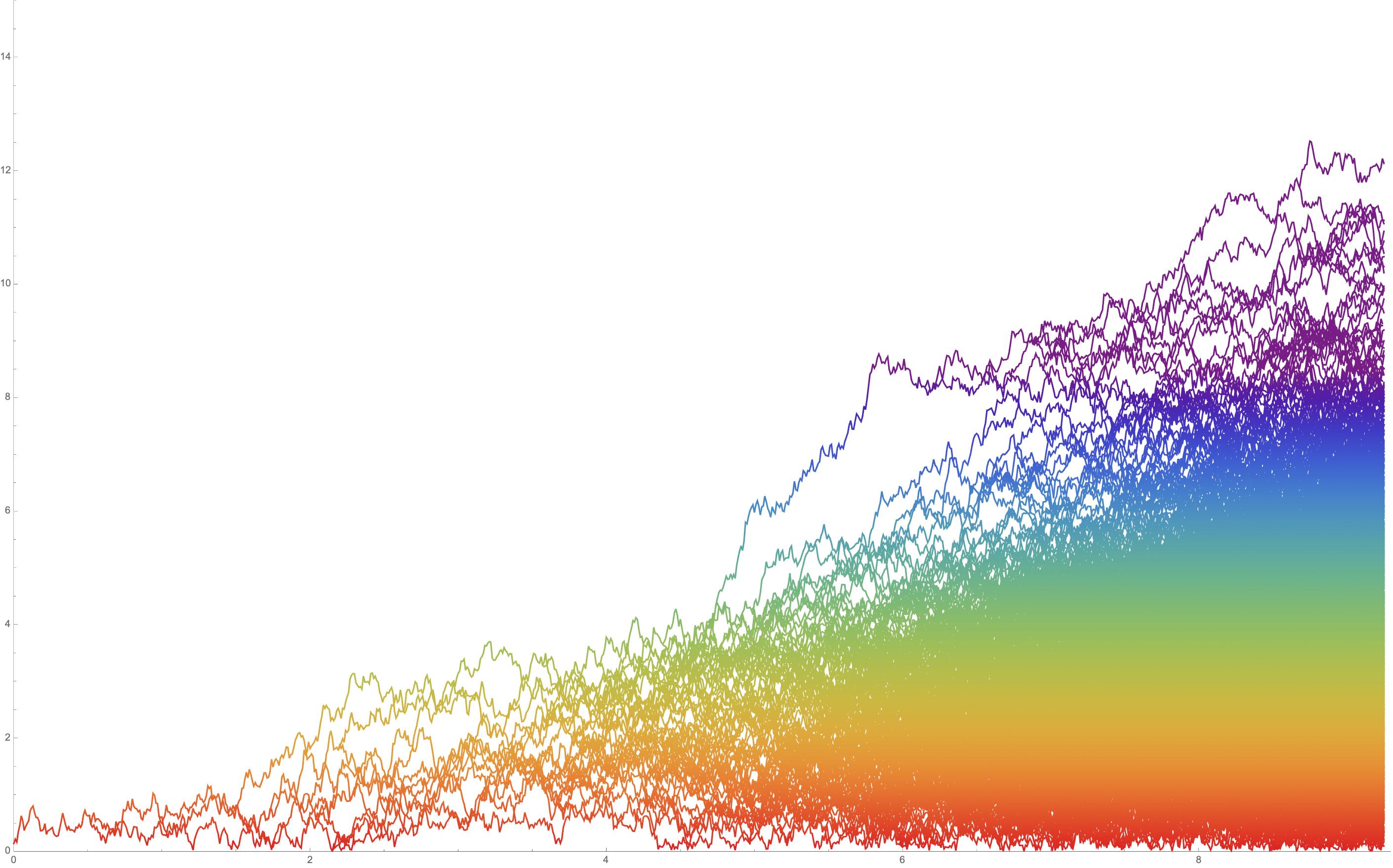}}
    \caption{Simulation of 2-dimensional BBM (left) and its modulus as a function of time (right) for BBM until the moment the population is $10,000$ particles. Color denotes the modulus.}
    \label{fig:my_label}
\end{figure}

Our goal in this paper is to  complement~\eqref{eq.mallein} and 
prove an analogue of~\eqref{eq-limd=1} for $d\geq 2$, thereby establishing convergence in distribution of $R_t^* -m_t$. Our main result reads as follows.
\begin{maintheorem}\label{thm:limiting-law}
Fix $d\geq 2$. Let $R_t^*$ denote the maximum modulus of the location 
at time $t$ of particles in a binary BBM in $\R^d$.
Let
 $m_t$ be as in~\eqref{def:m_t-c_d}. Then there exists a
 non-degenerate positive random variable $Z_\infty$ and a constant $\gamma^*>0$ so that 
$R_t^*-m_t$ converges in distribution as $t \to \infty$, to a Gumbel law shifted
by $-\log (\gamma^* Z_\infty/\sqrt{2})$, namely
\begin{align}
  \lim_{t\to \infty} \P(R_t^* - m_t \leq y) = \E\Big[ \exp \Big( - \gamma^* Z_{\infty}
  e^{-y\sqrt{2}} \Big)\Big]\,.
\end{align}
\end{maintheorem} 
\noindent Theorem~\ref{thm:limiting-law} resolves some of the open questions
in~\cite[Page 5]{SBM20}.

We now make some amplifying remarks on Theorem~\ref{thm:limiting-law}.
\begin{remark}   
The structure of $Z_\infty$ is reminiscent of the construction
of the limit of the derivative martingale in the theory of one 
dimensional Branching Brownian Motion. Namely, 
for $L>0$, let $\cF_L=\sigma\{ X_t^{(v)}, t\leq L, v\in \cN_L\}$, 
set 
\begin{align}
    I_L^{\win} := [
\sqrt{2} L - L^{2/3}, \sqrt{2}L - L^{1/6}]\,,\qquad
\cN_L^{\win}:= \{ v \in \cN_L : R_L^{(v)} \in I_L^{\win} \}\,.
\label{eq:def-window}
\end{align}
Introduce then
\begin{align}
    Z_L := \sum_{v \in \cN_L^{\win}} (R_L^{(v)})^{-\alpha_d} (\sqrt{2}L - R_L^{(v)}) e^{-(\sqrt{2}L - R_L^{(v)}) \sqrt{2}} \,.
    \label{def:Z_L}
\end{align}
The proof of Theorem~\ref{thm:limiting-law} will show that $Z_L$ (which is \textit{not} a martingale)
converges in distribution to a limiting random variable denoted $Z_\infty$, and that
conditionally on $\cF_L$, $R_t^*-m_t-2^{-1/2}\log(\gamma^* Z_L)$ converges in distribution,
as first $t\to\infty$ and then $L\to\infty$, to a Gumbel random variable. 
We emphasize that we do not prove the convergence in probability of
$Z_L$ to $Z_\infty$.

It is worthwhile to note that the restriction to $v \in \cN_L^{\win}$
in the definition of $Z_L$ is an artifact of our proof. In fact,
we show that particles $v\in\cN_L$ with $R_L^{(v)}\not\in I_L^{\win}$
are unlikely to produce a descendent  at time $t$ that has modulus larger than $m_t+y$, if $L$ is large enough (depending on $y$ but not on $T$).
Further, our calculations actually imply that the exponents $1/6$ and $2/3$ in~\eqref{eq:def-window} need only be smaller than $1/4$ and larger than $1/2$ respectively,
though we expect that one should be able to take any exponents smaller and larger than $1/2$, respectively.

We also note that in the definition
\eqref{def:Z_L}, one obviously can replace the factor $(R_L^{(v)})^{-\alpha_d}$ by $(\sqrt{2}L)^{-\alpha_d}$. 
\end{remark}
\begin{remark} 
  In a recent very interesting preprint, Stasi\'{n}ski, Berestycki, and Mallein
 ~\cite{SBM20} discuss the functional  derivative martingale
  \[ \widetilde Z_t(\theta)=\sum_{v\in \cN_t}(\sqrt{2}t-X_t^{(v)}\cdot \theta)
  e^{-(\sqrt{2}t-X_t^{(v)}\cdot\theta)}\,, \quad \theta\in S^{d-1}\,,\]
  where 
  $X_t^{(v)}\cdot \theta$ denotes the projection of $X_t^{(v)}$ in the direction $\theta$.
  They show (see their Theorem 1.3) that for almost every $\theta$, $\widetilde
  Z_t(\theta)
  \to \widetilde 
  Z_\infty(\theta)>0$ almost surely, and that $\langle \widetilde Z_t, f\rangle$
  converges almost surely, for bounded measurable $f$ on $S^{d-1}$, to
  $\langle \widetilde
  Z_\infty, f\rangle$, 
  where $\langle f,g\rangle$ is the standard inner 
  product with respect to surface measure on $S^{d-1}$. 
  They conjecture then (as a consequence of 
  their more general Conjecture 1.4) that $R_t^*-m_t-\log \gamma^*-
  \log \langle \widetilde Z_\infty,1\rangle$ converges in distribution to
  a Gumbel law. This conjecture is equivalent to stating that
  $Z_\infty\stackrel{d}{=} \langle\widetilde Z_\infty,1\rangle$.
  Of course, for $t$ large, $\langle \widetilde Z_t,1\rangle$ is formally
  dominated by a neighborhood of those $\bar \theta$
  which are local maxima of 
  $\widetilde Z_t(\cdot)$ so that $\widetilde Z_t(\bar \theta)$ has
  near maximal  value. In fact, a \textit{formal}
  Laplace asymptotic (expanding $Z_t$ quadratically in a neighborhood of $\theta^*$) yields
  that $\langle \widetilde Z_t, 1\rangle \sim C \widetilde Z_t(\theta^*)/
  t^{-\alpha_d}$, where $\theta^*$ is the global maximizer. In particular, Theorem~\ref{thm:limiting-law} seems compatible
  with the conjecture of~\cite{SBM20}. We note that the latter was recently proved in ~\cite{BKLMZ21}, building on the results of this paper and on \cite{SBM20}.
We refer to  \cite{BKLMZ21} for further details.

\end{remark}
\begin{remark} Theorem \ref{thm:limiting-law} can be understood as a statement on branching Bessel processes, and in fact our proof proceeds through that  prism. As such, it makes sense to ask for the analogue for arbitrary positive real $d$. An inspection of our proof reveals that in the context of
branching Bessel processes, it continues to
hold for $d\geq 2$, even if $d$ is not an integer. The situation for $d\in (0,2)$ is slightly different, because one has to properly define what happens to the Bessel process after it hits $0$,
and the result may depend on that definition: different definitions would result in different 
laws  of $Z_\infty$.
\end{remark}

\subsection{Model definition}
\label{sec-subsecdef}
The BBM model (with binary branching and branching rate $1$) is defined as follows. Start from a particle at the origin of $\R^d$. The particle performs a standard Brownian motion, and after an exponentially distributed time $\tau$ (independent of the motion of the particle), gives birth to two particles, and dies (we refer to this event as \textit{branching}). The process now repeats itself: all particles alive at time $t$ perform independent Brownian motion,
with their own (independent) exponential clocks determining 
their branching. The notions of ancestors and descendants of a particle
are defined in a self-evident way.

Let $\cN_t$ denote the collection of particles at time $t$. For a particle
$v\in\cN_t$ we let $X_t^{(v)}\in \R^d$ denote its location, and let
$R_t^{(v)} =\|X_t^{(v)}\|$ denote its (Euclidean) norm. For $v\in \cN_t$ we let $X_s^{(v)}$, $s\in[0,t]$, denote the continuous function obtained by
concatenating the trajectories of all ancestors of $v$. Note that
$X_\cdot^{(v)}$ is a Brownian motion in $\R^d$. We define similarly
$R_s^{(v)}$.

\subsection{Structure of the proof}
\label{sec-proofstructure}
Previous approaches to the analysis of multi-dimensional BBM involved looking at projections on given directions. Specifically,
Mallein~\cite{Mallein15} considers a discretization of angles (with mesh size increasing in $t$) en route to the proof of tightness; it seems hard to improve directly this approach for studying the convergence in distribution of the centered maximal modulus. Similarly,
Stasi\'{n}ski, Berestycki and Mallein~\cite{SBM20} consider, for fixed $t$,
the whole projection process as a function of the angle, and then prove convergence as $t\to\infty$,
but at a topology that is not strong enough for deducing results on the maximal modulus $R_t^*$. 

The crucial observation in the approach discussed in this paper 
is that, since $X_s^{(v)}$, $0\leq s\leq t$, is a Brownian motion  for any $v\in \cN_t$, the process $R_s^{(v)}$ is a $d$-dimensional Bessel process, and hence Markovian. In particular, one
can run through the proof of convergence of the maximum of branching random walks (see e.g.~\cite{Aidekon13}) in the version discussed in~\cite{BDZ16}. This involves a modified second moment method, coupled with
appropriate conditions on $\{R_L^{(v)}\}_{v\in\cN_L}$, for large $L$ that
goes to infinity only after $t$ does. That approach needs to be adapted to accommodate the fact that one is dealing with Bessel processes, and therefore increments are not independent; to handle that,
one rewrites probabilities in terms of (one-dimensional) Brownian motion, taking into account a Girsanov factor (see~\eqref{eq.girsanov}). 
To control the latter, we need to use slightly different barriers
than those used in~\cite{BDZ16}, and this requires developing appropriate barrier estimates. 

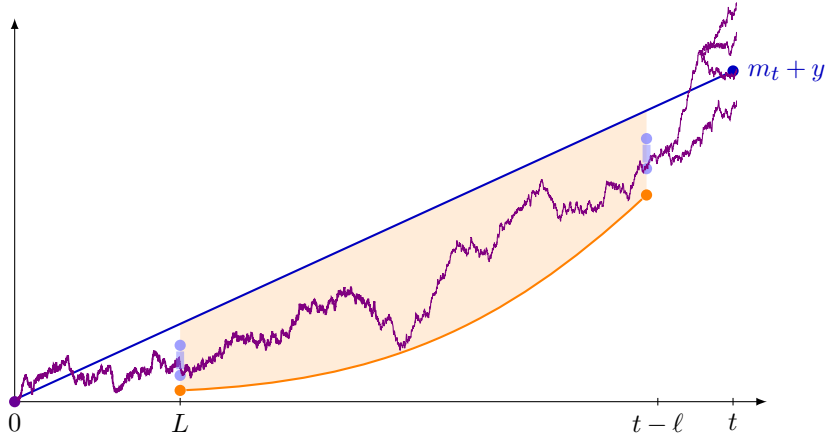
\begin{figure}
     \begin{tikzpicture}[>=latex,font=\small]
 \draw[->] (0, 0) -- (10, 0); 
  \draw[->] (0, 0) -- (0, 5.1);
   \coordinate  (b0) at (9.55,0);
   \coordinate  (L) at (2.2,0);
  \fill[color=orange!15] (2.2,0.15) to[bend right=20] (8.4,2.75) -- (8.4,3.88) -- (2.2,1.05);
   \node[circle,fill=blue!45!purple,inner sep=1.5pt] (o) at (0, 0) {};
   \node[circle,fill=blue!75!black,label={[blue!75!black]right:$m_t+y$},inner sep=1.5pt] (b) at ($(b0)+(0,4.4)$) {};
   \draw[blue!75!black, thick] ($(o)+(0.05,0.05)$) -- (b);
   \node[circle, fill=blue!50, inner sep=1.5pt, opacity=0.75] (y1) at ( $(b)-(1.15,0.9)$) {};
   \node[circle, fill=blue!50, inner sep=1.5pt, opacity=0.75] (y2) at ( $(y1)-(0,0.4)$) {};
   \node[circle, fill=blue!50, inner sep=1.5pt, opacity=0.75] (x1) at ( $(L)+(0,0.75)$) {};
   \node[circle, fill=blue!50, inner sep=1.5pt, opacity=0.75] (x2) at ( $(x1)-(0,0.4)$) {};
  \node[circle, fill=orange, inner sep=1.5pt] (ba1) at ( $(x2)-(0,0.2)$ ) {};
  \node[circle, fill=orange, inner sep=1.5pt] (ba2) at ( $(y2)-(0,0.35)$ ) {};
  \draw[orange, thick] (ba1) to[bend right=20] (ba2);
  \node (fig1) at (4.82,2.35) {
  \includegraphics[width=0.67\textwidth]{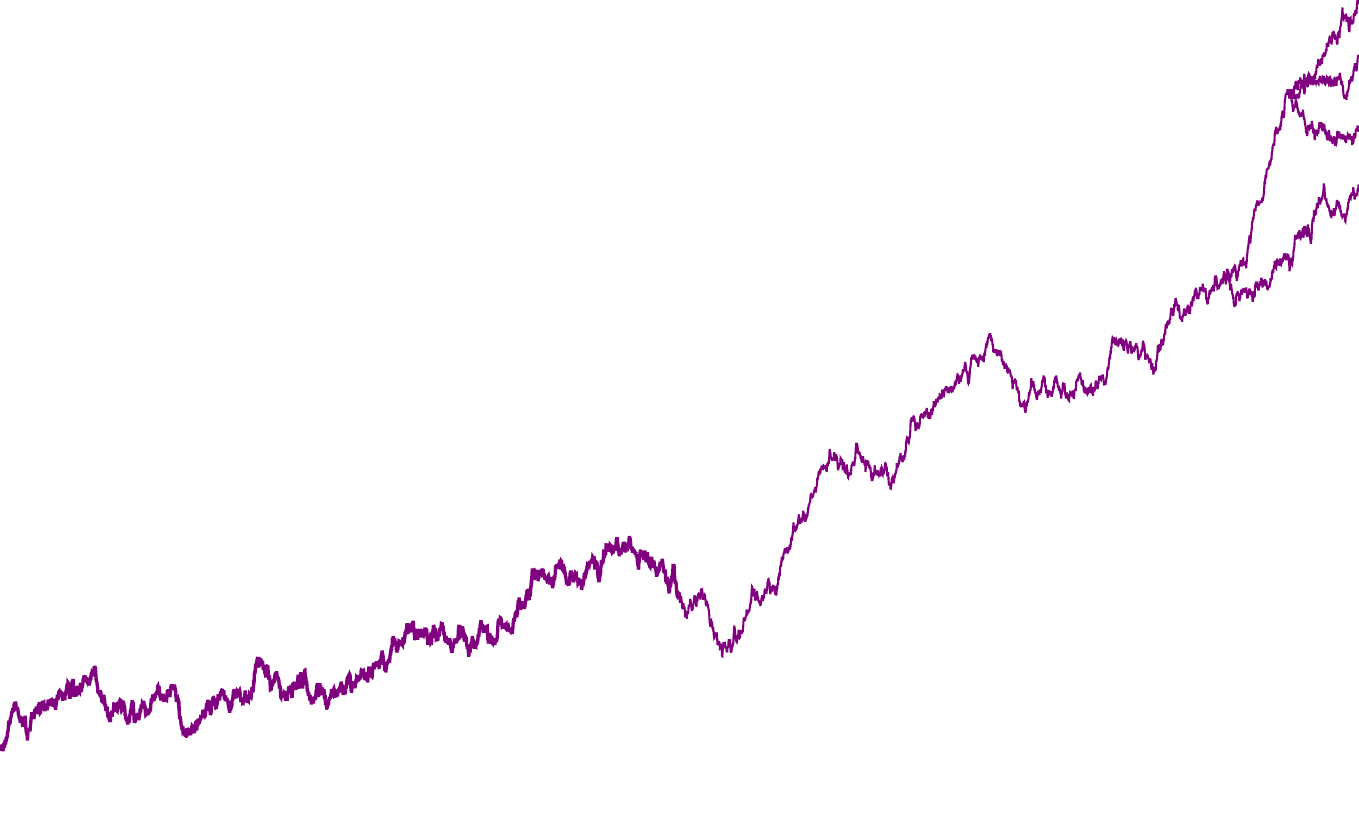}};
  \draw[blue!50, line width=3pt, opacity=0.5] (y1) -- (y2);
    \draw[blue!50, line width=3pt, opacity=0.5] (x1) -- (x2); 
  \node[below] at (0,-0.05) {$0$};
  \draw ($(L)+(0,0.05)$) -- ($(L)+(0,-0.05)$) node[below] {$L$};
  \draw ($(b0)+(-1,0.05)$) -- ($(b0)+(-1,-0.05)$) node[below] {$ t - \ell$};
  \draw ($(b0)+(0,0.05)$) -- ($(b0)+(0,-0.05)$) node[below] {$ t$};
  \end{tikzpicture}
     \caption{Particles that contribute to the maximizing event: at time $L$ are at height in $I_L^\win$, stay in the shaded (orange) region up to time $t-\ell$, at time $t-\ell$ are located in another ``window", and then produce a descendant that reaches $m_t+y$ at time $t$.}
     \label{fig:intro-lb}
 \end{figure}

In more detail: fix $L$ large enough (possibly dependent on $y$, which throughout this sketch is considered a fixed parameter). 
In the first step of the proof, see Theorem~\ref{thm:outside_window},
we 
show that the only particles in $\cN_L$ that will produce a descendant $v \in \cN_t$ with
$R_t^{(v)}\geq m_t+y$ are those in $I_L^{\win}$; the proof uses
the Bessel density, together with a-priori barrier estimates for the Bessel process, which are developed in Section  \ref{sec:window}. 
As a result of Theorem~\ref{thm:outside_window} and the Markov property, it will suffice to consider particles that start (at time $0$) in $I_L^{\win}$. This reduction allows us to control the rational factor in the Girsanov transform from Bessel to one-dimensional Brownian motion (i.e., $W_0^{-\alpha_d}$  in~\eqref{eq.girsanov}), which may now be uniformly approximated by $\sqrt{2}L^{-\alpha_d}$. We note that the one dimensional case has no need for this technical step, for several reasons. 
First, there is no Girsanov factor to consider,
and the (spatial) shift invariance of Brownian motion makes the ensuing barrier estimates in the following steps simpler to perform without the need to localize the starting point. Further, in dimension $d= 1$,  a simple union bound shows that with high probability, for $L$ large and for \textit{all}
$v\in\cN_L $ we have $R_L^{(v)}\ll \sqrt{2}L$, and therefore all particles stay below the linear barrier  at time $L$.

The most important step of the proof is Theorem~\ref{thm:right_tail_asymp},
which gives the precise tail asymptotics for the maximum of branching Bessel
processes started at time $L$ within $I_L^{\win}$, uniformly over $I_L^{\win}$. The proof of the latter, which is 
given in Section~\ref{sec:mod2ndmom} with key estimates proved in subsequent sections, is based on a  modified second moment approach. 
We employ first moment estimates (the many-to-one lemma, Lemma~\ref{lem:many-to-one})
together with a Girsanov transformation to show (within the proof of Proposition~\ref{prop:right_tail_equiv}) that particles $R_s^{(v)}$,
$s\in [L,t-\ell]$ (where $\ell \in [1, L^{1/6}]$ is going to infinity with $L$, see~\eqref{eqn:def-ell})
that do not stay within a domain determined by a lower barrier and 
 the linear line 
$m_t s/t$ are unlikely to contribute to the event
$R_t^*\geq m_t+y$; see Figure~\ref{fig:intro-lb} for an illustration. 
Proposition~\ref{prop:right_tail_equiv} further shows (by a second moment computation) that for a given particle
in $I_L^\win$, the expected number of its descendants at time $t-\ell$
that stay within the barriered domain up to that time and produce a descendant $w\in \cN_{t}$ with $R_t^{(w)}\geq m_t+y$, henceforth referred to as ``good particles,''
is
a good proxy for the probability of creating a descendant there
(this is carried out in Lemma~\ref{lem:2ndmom}, employing
a many-to-two lemma, see Lemma~\ref{lem:many-to-two},
and a truncation at level $t-\ell$ whose purpose is to ensure decorrelation.
The details are given in Section~\ref{sec:2ndmom}.).
Once these barriers are in place, the Girsanov factor can be controlled; and for a fixed $v \in \cN_L^\win$ with $R_L^{(v)}$ given, 
a precise estimate on the expected 
number of good particles that  descend from $v$,
can be obtained. This is the content of Proposition 
\ref{prop:1stmom-asymp}, whose proof is based in turn on the barrier
computations of Lemma~\ref{lem:barrier-equiv}; the proof of the latter 
for $d \geq 3$
takes up Section \ref{sec:1stmom}. Here, $d>2$ is useful because a certain term in the 
Girsanov exponent has a definitive sign and can be omitted from the computation. The case of $d=2$ requires a small modification, in order to control 
in the Girsanov transformation for the 
first moment an exponential term which is now positive instead of negative.
This requires an a-priori step where an extra barrier is introduced, which
gives an a-priori control of that term. Once this is carried out, the rest of the proof is as for $d>2$. The details are spelled out in Section \ref{sec-2d}. 

Once Theorem~\ref{thm:right_tail_asymp} holds, the proof of Theorem 
\ref{thm:limiting-law}, obtained by conditioning on $\mathcal{F}_L=\sigma(X_L^{(v)}, v\in \cN_L)$, is standard, and carried out in Section \ref{sec-3}.

In Section \ref{sec:prelim} below, we provide some a-priori material, including a description of the classical modified second moment method in dimension $1$, the many-to-few lemmas, the Girsanov transform, and barrier estimates for Brownian motion that will be used extensively
in the rest of the paper.
\\[0.8em]

\noindent
{\bf Acknowledgements} We thank the referees for a careful reading of the manuscript and useful comments and suggestions.
Y.K. and E.L.\ were supported by NSF grants DMS-1812095 and DMS-2054833.
O.Z.\ was partially supported by the European Research Council (ERC) under the European Union's Horizon 2020 research and innovation programme (grant agreement No.~692452).
This research was further supported in part by BSF grant 2018088.

\section{Preliminaries}
\label{sec:prelim}
We collect in this section preliminary material concerning the modified second moment method for standard (one-dimensional) BBM, the many-to-few lemmas, the Bessel process, and barrier estimates for Brownian motion. Section~\ref{subsec-BDZ-method} recalls the one dimensional  version of the modified second moment method; it informs the exposition in Section~\ref{sec:mod2ndmom}, where our version of the modified second moment method is described in full detail. The results of Sections~\ref{subsec:many-to-few}---\ref{subsec:hittingP} will be key technical tools used repeatedly throughout the rest of the article.

\subsection{The classical one-dimensional modified second moment method}
\label{subsec-BDZ-method}
The second moment method has played an oversized role in the study of
the maximum of BBM since its introduction in~\cite{Bramson78}. In the present article, we mostly follow the version described 
in~\cite{BDZ16} (done in the context of one-dimensional branching random walks).
As an introduction to our proof, we briefly summarize
their method in the language of one-dimensional BBM, which we denote by $\{W_s^{(v)}\}_{s\geq 0, v\in \cN_s}$, setting
$W_s^*:= \max_{v \in \cN_s} W_s^{(v)}$ 

By a standard reduction, see e.g.
\cite[Section~4]{BDZ16}, the proof of convergence in distribution of $W_t^*-m_t(1)$ follows from the following key step:
    there exists some constant $\xi^* >0 $ such that
    \begin{align}
        \lim_{x \to \infty} \liminf_{t \to \infty} \frac{\P( W_t^* > m_t(1) + x) }{ xe^{-x\sqrt{2}} }
        =
        \lim_{x \to \infty} \limsup_{t \to \infty} \frac{\P( W_t^* > m_t(1) + x) }{ xe^{-x\sqrt{2}} } = \xi^* \,.
        \label{eqn:1dcase-goal}
    \end{align}
 The analogue of~\eqref{eqn:1dcase-goal} in our case is Theorem~\ref{thm:right_tail_asymp}.
We now explain how~\eqref{eqn:1dcase-goal} is classically shown.
For any $r, s\geq 0$ and some $v\in \cN_s$, let
$\cN_r^v$ denote the particles in $\cN_{r+s}$ that are descendants of~$v$. 
Fix a parameter $\ell := \ell(x)$ such that 
\begin{equation*}
    1\leq \ell(x) \leq x 
    \qquad\mbox{and}\qquad 
    \lim_{x\to\infty}\ell(x)=\infty \,.
\end{equation*}
The exact choice of $\ell$ does not matter as long as the above conditions are satisfied.
For each $v\in \cN_{t -\ell}$,
define the $\cF_{t}$-measurable event 
\begin{align}
    \cA_{v,t}(x) := 
    \Big \{
    \Big \{
        W_s^{(v)} \leq 
        \frac{m_t(1)}{t}s+x, ~\forall s \in [ 0, t -\ell] 
    \Big \}
    \cap 
    \Big \{
        \max_{v' \in \cN_{\ell}^v}
        W_{t}^{(v')} > m_t(1) +x
    \Big \}
    \Big \} \, .
    \label{eqn:def-A_v,t,ell}
\end{align}
The central idea of~\cite{BDZ16} is that the random variable
$
    \Xi_{t,\ell}(x) = \sum_{v \in \cN_{t- \ell}} \one_{\cA_{v,t}(x)}
$
satisfies two important properties: first,
\begin{equation}
    \lim_{x\to \infty} \liminf_{t\to \infty} 
    \E [\Xi_{t,\ell}(x)] / \E [\Xi_{t,\ell}(x)^2] =
    \lim_{x\to \infty} \limsup_{t\to \infty} 
    \E [\Xi_{t,\ell}(x)] / \E [\Xi_{t,\ell}(x)^2]
    =1 \, ,
    \label{eqn:bdz-1st=2ndmom}
\end{equation}
from which the Paley--Zygmund inequality 
(and some additional technical estimates) yields
\begin{align*}
    &\lim_{x\to \infty} \liminf_{t\to \infty}
    \P ( W_{t}^* > m_t(1) +x) / \E[\Xi_{t,\ell}(x)]\\
    &=
    \lim_{x\to \infty} \limsup_{t\to \infty}
    \P ( W_{t}^* > m_t(1) +x) / \E[\Xi_{t,\ell}(x)] =  1 \, ;
\end{align*}
and second, the first moment $\E[\Xi_{t,\ell}(x) ]$ is amenable to precise computations due to the many-to-one lemma (Lemma~\ref{lem:many-to-one}), the Brownian ballot theorem (Lemma~\ref{lem:brownian-ballot-explicit}), and known estimates on $W_s^*$. 
The key condition that dictates the choice of $\cA_{v,t}(x)$ is~\eqref{eqn:bdz-1st=2ndmom}, which
 will be satisfied if the $\cA_{v,t}
$ are sufficiently decorrelated. This is the role of $\ell$: 
``cutting" the tree at a time that is order $1$ from the ending time $t$ provides crucial decorrelation. Indeed,
note that the events
\[
    \fT_{\ell}^{(1)}(v) := \Big \{
        \max_{v' \in \cN_{\ell}^v}
        W_{t}^{(v')} > m_t(1) +y
    \Big \} \, , ~v\in \cN_{t - \ell}
\]
are all independent conditional on $\cF_{t-\ell}$. 
The linear barrier up to time $t-\ell$ in ~\eqref{eqn:def-A_v,t,ell} provides further  decorrelation by disregarding particles that at any time rise above $\frac{m_t(1)}{t}(\cdot) + x$. 

\subsection{Many-to-few lemmas}
\label{subsec:many-to-few}
We describe in this section the many-to-few lemmas, based on~\cite{HaRo17}; we do not describe the historical source of the term and the important role
that these lemmas played in the study of branching random walks and Brownian motion. The interested reader is referred to~\cite{HaRo17} and~\cite{Maillard12} for background.
For our purposes, we are interested in a binary branching process $\{X_s^{(v)}\}_{s\geq 0, v \in \cN_s}$ with branching rate $1$, where the
$X_.^{(v)}$ are diffusion processes. The following two lemmas appear
(in a more general form\footnote{
Lemmas~\ref{lem:many-to-one} and~\ref{lem:many-to-two} come from p.~$230$ of Sec.~$4.1$ and p.~$231$ of Sec.~$4.2$ of~\cite{HaRo17}, resp., as follows. 
We always take their $\zeta(v,t)$ to be $1$, whence their measure $\P_x^k$ is identical to their $\Q_x^k$, and we denote its expectation by~$\E_x$. Their $\alpha_n(y)$ is  $2^n-1$ in our situation, since the $n$-th moment of the offspring distribution (their $m_n(y)$) is equal to $2^n$ for us and the branching rate (their $R(y)$) is $1$. Fixing a time $T\geq 0$, we have taken their random variable~$Y$, defined to be measurable w.r.t.\ the entire process up to time $T$, to be a product of measurable functions of paths of single particles in both lemmas. Lastly, in Lemma~\ref{lem:many-to-two}, we have denoted  their $T(1,2)$ by $\tau$.
}) in~\cite{HaRo17}.
\begin{lemma}[Many-to-one lemma]
For any $T\geq 0$, $x \in \R$, and a measurable function $f:C[0,T] \to \R$,  we have
\label{lem:many-to-one}
\begin{align}
    \E_x\bigg[ \sum_{v \in \cN_T} f\big((X_s^{(v)})_{s \leq T}\big)\bigg]
    =
    e^T \E_x \brak{f\big((X_s^{(v)})_{s \leq T}\big)} \,.
    \label{eqn:many-to-one}
\end{align}
\end{lemma}
Lemma~\ref{lem:many-to-one} will be used repeatedly, oftentimes in the following situation:
suppose we wish to bound from above $\P_x(\exists v \in \cN_T: (X_s^{(v)})_{s \leq T} \in A)$, for some event $A \in \cF_T$. A union bound gives the upper bound $\E_x[ \sum_{v \in \cN_T} \mathds{1}_{(X_s^{(v)})_{s \leq T} \in A } ]$, which by an application of the
many-to-one lemma reduces to the study of the path of a single particle, which is a Bessel process. The many-to-two lemma below will only be used at the start of Section~\ref{sec:2ndmom} for a second moment computation.
\begin{lemma}[Many-to-two lemma]
\label{lem:many-to-two}
Fix any $T \geq 0$, and let $f$ and $g$ be measurable real functions on $C[0,T]$. Distinguish two particles $v, v' \in \cN_T$, and let $\tau$ denote the time at which $v$ and $v'$ branched from each other. Then for any $x \in \R$,  
\begin{eqnarray}
    &&\!\!\!\!\!\!\!\!\!\!\!\!\!\!\!\!\!\!\E_x\Big[
    \sum_{v_1, v_2 \in \cN_T} f\big((X_s^{(v_1)})_{s \leq T}\big) g\big((X_s^{(v_2)})_{s\leq T}\big)
    \Big] =
    e^T \E_x \Big[f\big((X_s^{(v)})_{s \leq T}\big) g\big((X_s^{(v)})_{s \leq T}\big) \Big] \nonumber \\
    &&+2
\int_{0}^{T} e^{T+\sigma}\mathbb
E_{x} \Big[f \big((X_{s}^{(v)}
)_{s \leq T} \big) g \big((X_{s}^{(v')}
)_{s
\leq T} \big) | \tau = T-\sigma \Big] \,\mathrm{d}\sigma,.
    \label{eqn:many-to-two}
\end{eqnarray}
\end{lemma}

As an example, taking $f$ and $g$ to be identically $1$, the many-to-one lemma tells us that
$\E[\cN_T] =e^T$, and the many-to-two lemma gives $\E[\cN_T^2] = 2e^{2T} - e^T$.

\subsection{Bessel processes}
\label{subsec:bessel}
We will frequently use $R_t$ to denote the process given by the norm of standard $d$-dimensional Brownian motion at time $t$. 
It is well known that $R$ is a \textit{Bessel process} of dimension $d$; therefore, when $R_0 >0$, it satisfies the following SDE  (see~\cite[Chapter~XI]{ReYo99} for a treatment of Bessel processes):
\begin{align}
    \d R_t = \frac{\alpha_d}{R_t} \d t + \d W_t, 
    \label{eqn:bessel_SDE}
\end{align}
where we recall that $\alpha_d = (d-1)/2$ (as defined in~\eqref{def:m_t-c_d}) and that $W_t$ denotes a standard Brownian motion\footnote{
For general $d \in \R$, the SDE~\eqref{eqn:bessel_SDE} is only satisfied up to time $\tau := \inf \{t >0 : R_t = 0 \}$. However, since we take $d\geq 2$ an integer and $R_0 >0$, $\tau = \infty$ almost surely. 
}. In  particular, the Girsanov transform gives us 
\begin{align}
    \d P_x^R \big |_{\cF_t} = \pth{\frac{W_t}{W_0}}^{\alpha_d} 
    \exp \pth{\frac{\alpha_d - \alpha_d^2}{2} \int_0^t \frac{1}{W_u^2} \d u} \one_{\{ W_u > 0, ~u \in [0,t]\}} \d P_x^W \big |_{\cF_t} \, ,
    \label{eq.girsanov}
\end{align}
where $P_x^R$ and $P_x^W$ are the laws of a $d$-dimensional Bessel process and a one-dimensional Wiener process respectively, each started from $x>0$. Henceforth, $W_t$ will always denote a one-dimensional Wiener process.

\subsection{Hitting probabilities of Brownian motion}
\label{subsec:hittingP}
In this subsection, we record several important results pertaining to hitting probabilities of Brownian motion.
We begin by defining some notation that will be used throughout the article. 

For functions $f,g: [0,\infty) \to \R$, a set  $I \subset [0,\infty)$,
and a real-valued process $X_\cdot$
we call events of the following form \emph{barrier events}:
\begin{align}
    \UB{I}{f}(X_\cdot) &:= \{ X_u \leq f(u), ~\forall u \in I \} \quad \text{ and } 
    \quad 
    \LB{I}{f} (X_\cdot) := \{ X_u \geq f(u), ~\forall u \in I \}
    \, . \label{eqn:notation-barrier}
\end{align}
In each instance, we will take $I$ to be a union of intervals and $X_\cdot$ to be one of $W_\cdot$, $R_\cdot$, $W_.^{(v)}$ or~$R_.^{(v)}$. An important barrier function will be the linear function $f_a^b(s;T): [0,T] \to \R$ whose graph is the line segment connecting $(0,a)$ to $(T,b)$; that is,
\begin{equation}\label{eq:def-f_a^b}
f_a^b(s;T) := a + (b-a) \frac{s}T \qquad (0\leq s \leq T)\,.
\end{equation}
When clear, we will write $\P_x(\cdot)$ to denote the law of a process started from $x$ at time $0$. Further, for fixed $T> 0$ and  $x, y \in \R$, we will write
$\P_{x,T}^y(\cdot)$ to denote the law of a process started from $x$ at time $0$ and ending at $y$ at time $T$.
Our first result of this subsection is the classical ballot theorem for the  Brownian bridge.

\begin{lemma}[Brownian ballot Theorem]
\label{lem:brownian-ballot-explicit}
Let $a \geq x$ and $b \geq y$. For any $T > 0$, 
\begin{align}
    \P_{x,T}^y \Big( \UB{[0,T]}{f_a^b(\cdot;T)}(W_\cdot)  \Big) = 1 - \exp \Big(\! -2 \frac{(a-x)(b-y)}{T} \Big) \, .
    \label{eqn:brownian-ballot-explicit}
\end{align}
Consequently, the following holds uniformly over $(a-x)(b-y) \leq g(T)$, for any $g(T) = o(1)$:
\begin{align}
    \P_{x,T}^y \Big( \UB{[0,T]}{f_a^b(\cdot;T)}(W.)  \Big)
    \sim
    2 \frac{(a-x)(b-y)}{T} \, .
    \label{eqn:brownian-ballot-sim}
\end{align}
\end{lemma}

Lemma~\ref{lem:brownian-ballot-explicit} computes the hitting probability of a Brownian bridge w.r.t.\ a straight line. 
Lemma~$2.6$, Lemma~$2.7$, and Proposition~$6.1$ of~\cite{Bramson83} estimate hitting probabilities for a much more general family of barriers.
These results are stated in~\cite{Bramson83} for a Brownian bridge starting and ending at $0$ on the interval $[0,T]$.
The latter  may easily be generalized to 
general Brownian bridges using the process-level equivalence
\begin{equation}
\P_{x,T}^y \big (\big( W_s \big)_{s \in [0,T]} \in \cdot \big) = \P_{0,T}^0 \big( \big ( W_s+ f_x^y(s;T) \big)_{s\in[0,T]} \in \cdot \big) \,.
    \label{eqn:bridge-stoch-equiv}
\end{equation}
In Lemmas~\ref{lem:Bram2.6}---\ref{lem:Bram6.1}, we record these generalized results and demonstrate how they follow from~\cite{Bramson83} by proving Lemma~\ref{lem:Bram2.6}; Lemmas~\ref{lem:Bram2.7} and~\ref{lem:Bram6.1} are proved in the same way. Figures~\ref{fig:lemma-Bram2.7} and~\ref{fig:lemma-Bram6.1} complement these results.

\begin{lemma}[From {\cite[Lemma~2.6]{Bramson83}}]
\label{lem:Bram2.6}
Fix $x,y \in \R$ and $T >0$.
Consider the (possibly infinite-valued) functions $\cL(s)$, $\cU_1(s)$, and $\cU_2(s)$ satisfying $\cL(s) < \cU_1(s) \leq \cU_2(s)$  for $s \in [0,T]$ and $\P_x^y (W_s \leq \cU_1(s)) > 0$.  Then we have
\begin{align}
    \P_{x,T}^y \Big( \LB{[0,T]}{\cL}(W_\cdot) \given
    \UB{[0,T]}{\cU_2}(W_\cdot) \Big )
    \geq 
    \P_{x,T}^y \pth{ \LB{[0,T]}{\cL}(W_\cdot) \given
    \UB{[0,T]}{\cU_1}(W_\cdot)}\,,
    \label{eqn:Bram2.6-upper}
\end{align}
\begin{align}
    \P_{x,T}^y \pth{ \UB{[0,T]}{\cL}(W_\cdot) \given
    \UB{[0,T]}{\cU_1}(W_\cdot)}
    \geq 
    \P_{x,T}^y \pth{ \UB{[0,T]}{\cL}(W_\cdot) \given
    \UB{[0,T]}{\cU_2}(W_\cdot)}  \,.
    \label{eqn:Bram2.6-lower}
\end{align}

\begin{proof}
\cite[Equation~2.17]{Bramson83} and the reflection principle give
\begin{align}
    \P^0_{0,T} \pth{ \LB{[0,T]}{\cL}(W_\cdot) \given 
    \UB{[0,T]}{\cU_2}(W_\cdot) }
    \geq 
    \P^0_{0,T} \pth{ \LB{[0,T]}{\cL}(W_\cdot) \given 
    \UB{[0,T]}{\cU_1}(W_\cdot)
    }.
\end{align}
Adding $f_x^y(\cdot; T)$ to both sides, we may rewrite our barrier events as
\[
\LB{[0,T]}{\cL}(W_\cdot)
= 
\LB{[0,T]}{\cL+f_x^y(\cdot; T)}(W_\cdot+f_x^y(\cdot; T)) \, ,
\]
and similarly for $\UB{[0,T]}{\cU_1}(W_\cdot)$ and $\UB{[0,T]}{\cU_2}(W_\cdot)$. It then follows from~\eqref{eqn:bridge-stoch-equiv} that
\begin{align}
    \P^0_{0,T} \pth{ \LB{[0,T]}{\cL}(W_\cdot)
    \given 
    \UB{[0,T]}{\cU_2}(W_\cdot) }
    =
    \P^x_{y,T} \pth{ 
    \LB{[0,T]}{\cL+f_x^y(\cdot; T)}(W_\cdot)  \given 
    \UB{[0,T]}{\cU_2+f_x^y(\cdot; T)}(W_\cdot) 
    }\,,
\end{align}
and 
\begin{align}
    \P_{0,T}^0 \Big( \LB{[0,T]}{\cL}(W_\cdot)
    \given 
    \UB{[0,T]}{\cU_1}(W_\cdot) \Big)
    =
    \P_{x,T}^y \Big(
    \LB{[0,T]}{\cL+f_x^y(\cdot; T)}(W_\cdot)  \given 
    \UB{[0,T]}{\cU_1+f_x^y(\cdot; T)}(W_\cdot) 
    \Big)\,.
\end{align}
Renaming $\cL := \cL+f_x^y(\cdot; T)$, $\cU_1 := \cU_1+f_x^y(\cdot; T)$, and $\cU_2 := \cU_2+f_x^y(\cdot; T)$, we find that 
\begin{align}
    \P_{x,T}^y \Big(
    \LB{[0,T]}{\cL}(W_\cdot)  \given 
    \UB{[0,T]}{\cU_2}(W_\cdot) 
    \Big)
    \geq 
    \P_{x,T}^y\Big(
    \LB{[0,T]}{\cL}(W_\cdot)  \given 
    \UB{[0,T]}{\cU_1}(W_\cdot) 
    \Big)  \,.
\end{align}
This gives~\eqref{eqn:Bram2.6-upper}. Equation~\eqref{eqn:Bram2.6-lower} follows in the same way.
\end{proof}
\end{lemma}

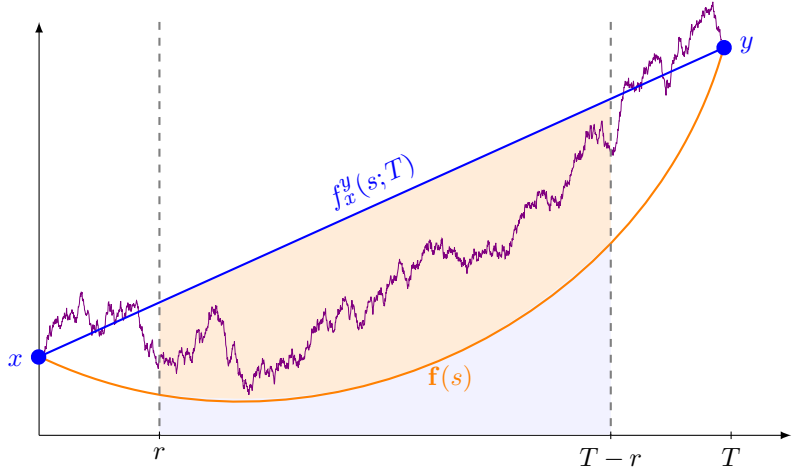
\begin{figure}
 \begin{tikzpicture}[>=latex,font=\small]
 \draw[->] (0, 0) -- (10, 0); 
  \draw[->] (0, 0) -- (0, 5.5);
   \coordinate (x) at (-0.1,1);
   \coordinate (y) at (9.2,5.2);
  \fill[color=orange!15] (1.6,1.75) -- (7.6,4.47) to (7.6,2.57) to[bend left=25] (1.6, 0.55);
  \fill[color=blue!5] (7.6,0.01) -- (7.6,2.57) to[bend left=25] (1.6, 0.55) -- (1.6,0.01);
  \draw[thick,orange] ($(x)+(0.1,0.05)$) to [bend right=50] node[midway,below] {$\mathbf{f}(s)$} ($(y)+(-0.1,-0.05)$);
  \node (fig1) at (4.55,3.15) {
  \includegraphics[width=.66\textwidth]{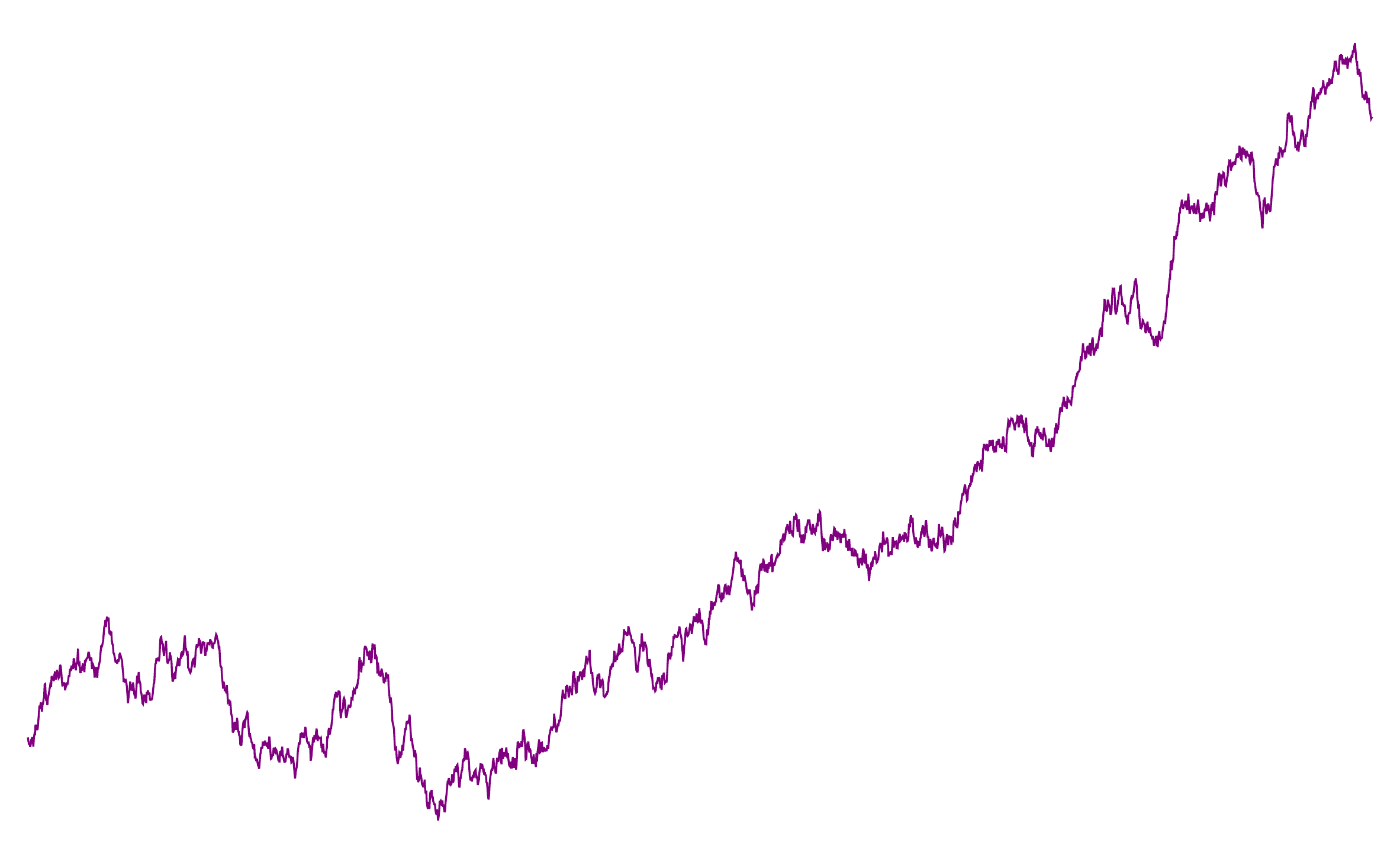}};
  \draw[blue, thick,*-*] (x) -- node[above,midway,sloped]{$f_x^y(s;T)$} (y);
  \draw[dashed, thick,gray] (1.6,5.5) -- (1.6,0);
  \draw[dashed, thick,gray] (7.6,5.5) -- (7.6,0);
  \node[blue,left] at (x) {$x$};
  \node[blue,right] at (y) {$y$};
  \draw (1.6,0.05) -- (1.6,-0.05) node[below] {$r$};
  \draw (7.6,0.05) -- (7.6,-0.05) node[below] {$T-r$};
  \draw (9.2,0.05) -- (9.2,-0.05) node[below] {$T$};
  \end{tikzpicture}
     \caption{The event from Lemma~\ref{lem:Bram2.7}: for $T\geq 2r$, the probability that Brownian motion is above $\mathbf{f}(s)$ (orange) on $[r,T-r]$, given that it is below $f_x^y(s;T)$ (blue) on $[r,T-r]$, goes to $1$ as $r \to \infty$.}
     \label{fig:lemma-Bram2.7}
 \end{figure}

\begin{lemma}[From {\cite[Lemma~2.7]{Bramson83}}]
\label{lem:Bram2.7}
Fix $C > 0$, $\epsilon > 1/2$, and $x,y \in \R$. Define
\[ \mathbf{f}(s) 
:= \mathbf{f}(s;x,y,T) 
= f_x^y(s;T)-C(s\wedge(T-s))^\epsilon\, .\]
Then, uniformly over $T \geq 2r$, we have
\begin{align}
    \P_{x,T}^y \pth{ \LB{[r, T-r]}{\mathbf{f}}( W_\cdot ) 
    \given 
    \UB{[r,T-r]}{f_x^y(\cdot;T)} ( W_\cdot )
    }
    \to 1 \text{ as } r \to \infty \, .
    \label{eqn:lem-bram2.7}
\end{align}
Furthermore, the left-hand side of~\eqref{eqn:lem-bram2.7} is constant in $x$ and $y$.
\end{lemma}

In the following result, we will consider a fixed family of (not necessarily finite-valued) functions $\{\cL_T\}_{ T > 0}$, each defined at least on $[0,T]$. Mimicking the notation of~\cite[Section~6]{Bramson83}, we will write  $\overline{\cL}_T(s) := \cL_T(s) + C\pth{ s \wedge (T-s) }^{\delta}$ for fixed constants $C >0$ and $\delta \in (0,1/2)$.

\begin{lemma}[From  {\cite[Lemma~6.1]{Bramson83}}]
\label{lem:Bram6.1}
Fix constants $a,b \in \R$, $C>0$, $\delta \in (0,1/2)$, and fix a family of (not necessarily finite-valued) functions $\{\cL_T\}$ as in the previous paragraph such that there exists a fixed $r_0 >0$ for which
\begin{align}
    \overline{\cL}_T(s) \geq f_a^b(s;T)~\text{for all } s \in [r_0, T- r_0],
    \label{eqn:bram6.1-cond}
\end{align}
holds 
for all $ T \geq 2 r_0$. 
Then, uniformly over $\cL_T$, $x \leq a$, $y \leq b$, and $T \geq 2r$, we have
\begin{align}
    \frac{
        \P_{x,T}^y \Big(\UB{[r,T-r]}{\overline{\cL}_T} (W_\cdot) \Big)
    }
    {
        \P_{x,T}^y 
        \Big( \UB{[r,T-r]}{\cL_T}(W_\cdot) \Big)
    } 
    \to 1, \text{ as } r\to\infty \, .
    \label{eqn:bram6.1-eqn}
\end{align}
\end{lemma}

 \begin{figure}
     \begin{tikzpicture}[>=latex,font=\small]
 \draw[->] (0, 0) -- (10, 0); 
  \draw[->] (0, 0) -- (0, 5.5);
   \coordinate (a) at (-0.1,2.38);
   \coordinate (b) at (9.2,5.2);
   \node[circle,fill=orange,inner sep=1.75pt] (a0) at (0.0,1.65) {};
   \node[circle,fill=orange,inner sep=1.75pt] (b0) at (9.14,4.48) {};
   \node[circle,fill=blue!45!purple,label={[blue!45!purple]left:$x$},inner sep=1.75pt] (x) at (0.0,2) {};
   \node[circle,fill=blue!45!purple,label={[blue!45!purple]right:$y$},inner sep=1.75pt] (y) at (9.13,4.75) {};
   \fill[color=orange!15] (1.6,1.86) to[bend right=7] (7.5,3.65) -- (7.5,0.01) -- (1.6,0.01);
   \fill[color=blue!5] (7.5,3.65) to[bend left=7] (1.6, 1.86) -- (1.6,3.1) to[bend left=20] (7.5,4.76);
  \node (fig1) at (4.55,3.2) {
  \includegraphics[width=.63\textwidth]{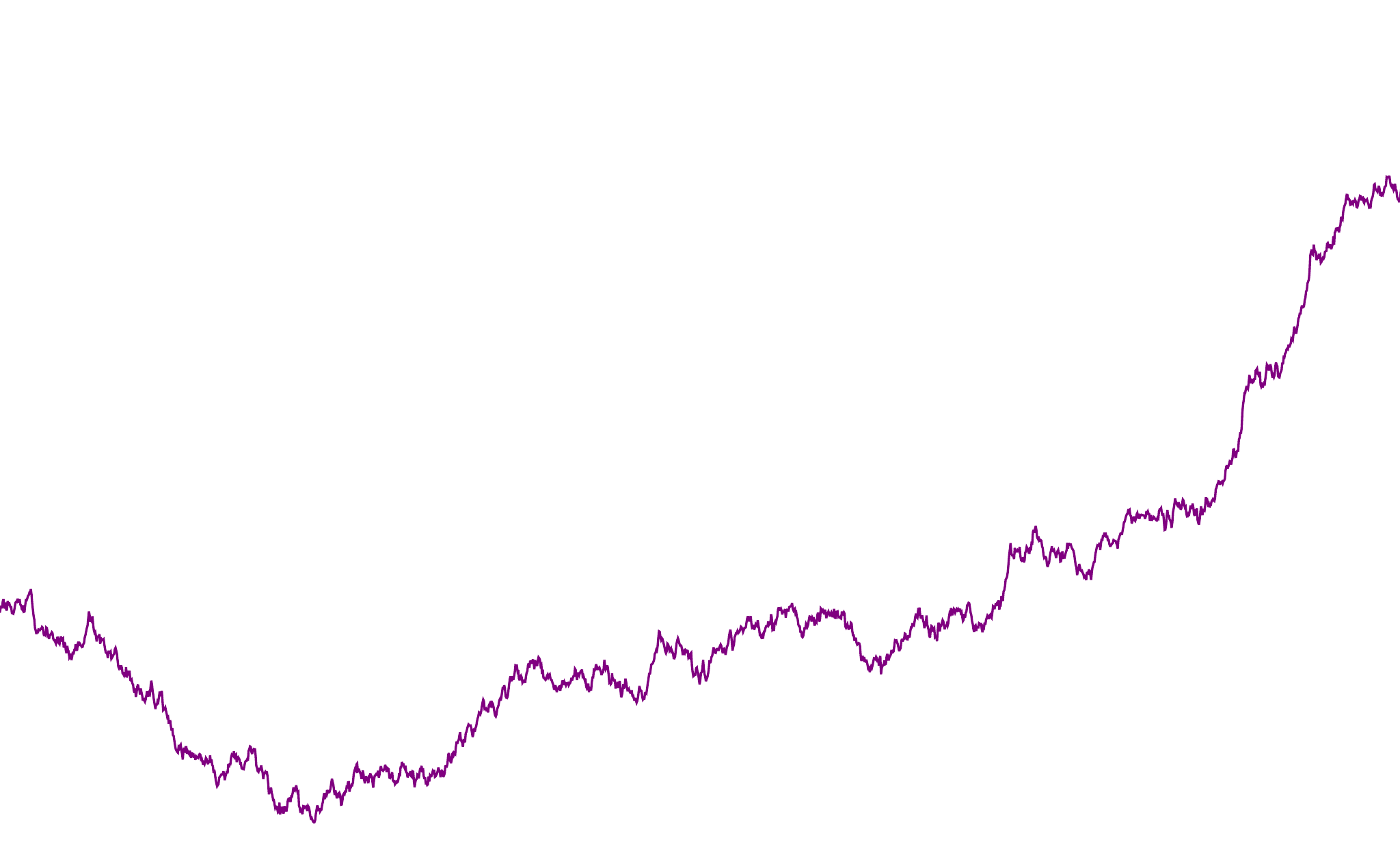}};
  \path[blue!75!black,thick] (a0) edge [bend left=30]
  node[midway,above] {$\overline{\mathcal{L}}_T(s)$}
   (b0);
   \draw[thick,orange] (a0) edge [bend right=10] node[midway,above] {$\mathcal{L}_T(s)$} (b0);
  \draw[blue!75, thin, dashed, *-*] ($(a)+(0.02,0.02)$) -- node[above,midway,sloped,font=\tiny,inner sep=1pt]{$f_a^b(s;T)$} (b);
  \draw[dashed, thick,gray] (1.6,5.5) -- (1.6,0);
  \draw[dashed, thick,gray] (7.5,5.5) -- (7.5,0);
  \node[blue,left] at (a) {$a$};
  \node[blue,right] at (b) {$b$};
  \draw (1.6,0.05) -- (1.6,-0.05) node[below] {$r_0$};
  \draw (7.5,0.05) -- (7.5,-0.05) node[below] {$T-r_0$};
  \draw (9.2,0.05) -- (9.2,-0.05) node[below] {$T$};
  \end{tikzpicture}
     \caption{The event from Lemma~\ref{lem:Bram6.1}: a Brownian bridge on $[0,T]$ from $x\leq a$ to $y\leq b$ (purple). The curve  $\overline{\cL}_T$ (blue) satisfies~\eqref{eqn:bram6.1-cond}. The probability of staying below the blue curve on $[r,T-r]$ is asymptotically equivalent to the probability of staying below the orange curve on $[r,T-r]$.}
     \label{fig:lemma-Bram6.1}
 \end{figure}
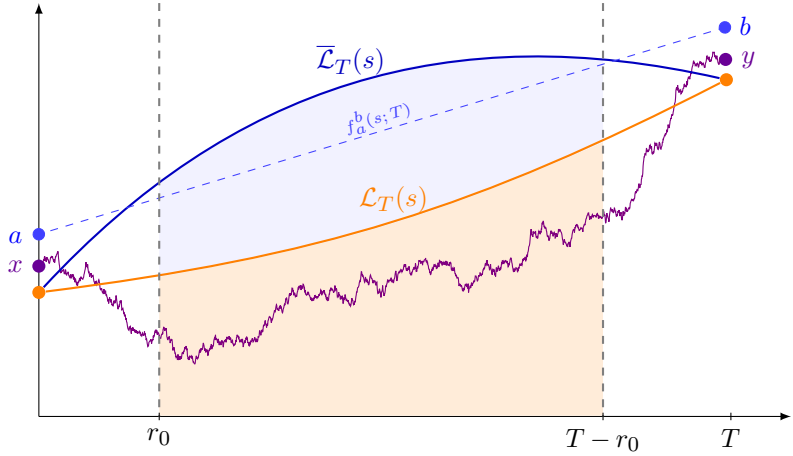

Two consequences of Lemmas \ref{lem:Bram2.6}---\ref{lem:Bram6.1} are~\cite[Lemma~2.1]{BRZ19}, reproduced next as Lemma~\ref{lem:brz-lem2.1-bridge}, and Lemma~\ref{lem:brz-modified}, which is a slight modification of Lemma~\ref{lem:brz-lem2.1-bridge}. These results may be thought of as extensions of the ballot theorem (Lemma~\ref{lem:brownian-ballot-explicit}) to more general barriers. See Figure~\ref{fig:lemma-BRZ-2.1} for a visual description of the events in question in Lemma~\ref{lem:brz-lem2.1-bridge} 

\begin{lemma}[{\cite[Lemma~2.1]{BRZ19}}]
\label{lem:brz-lem2.1-bridge}
Fix any $\epsilon \in (0,1/2)$, $\eta > 1$, and $C>0$.
Then the following holds uniformly in $x \leq a$, $y\leq b$, $(a-x)(b-y) \leq \eta T$, and $T$ large enough:
\begin{align}
    &\P_{x,T}^y \pth{
    \LB{[1,T-1]}{f_x^y(s;T) -  
    C(s\wedge (T-s))^{\frac{1}{2}+ \epsilon}}(W_{\cdot})
    \cap
    \UB{[1,T-1]}{f_a^b(s;T) - C(s\wedge (T-s))^{\frac{1}{2}-\epsilon}}(W_{\cdot})
     }
   \nonumber \\
  & \asymp
    \P_{x,T}^y \pth{\UB{[1,T-1]}{
    f_a^b(s;T) + C(s\wedge (T-s))^{\frac{1}{2}-\epsilon}
    } (W_{\cdot})}
       \asymp 
    \frac{(1+a-x)(1+b-y)}{T} \, .
    \label{eqn:brz2.1-bridge}
\end{align}
\end{lemma}

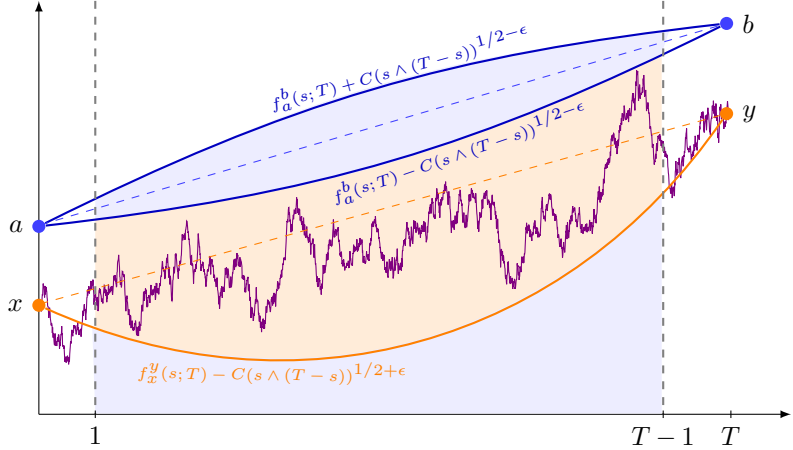
\begin{figure}
         \begin{tikzpicture}[>=latex,font=\small]
 \draw[->] (0, 0) -- (10, 0); 
  \draw[->] (0, 0) -- (0, 5.5);
   \fill[color=blue!7]  (0.75,2.85) to[bend left=10] (8.38,5.1) -- (8.28,4.8) to[bend left=10] (0.75,2.6);
   \fill[color=orange!15] (8.28,4.79) to[bend left=9.5] (0.75,2.6) -- (0.75,1.15) to[bend right=45] (8.28,2.9);
   \fill[color=blue!7] (0.71,1.15) to[bend right=30] (8.28,2.9) to (8.28,0.01) to (0.75,0.01);
  \node (fig1) at (4.6,2.65) {
  \includegraphics[width=.64\textwidth]{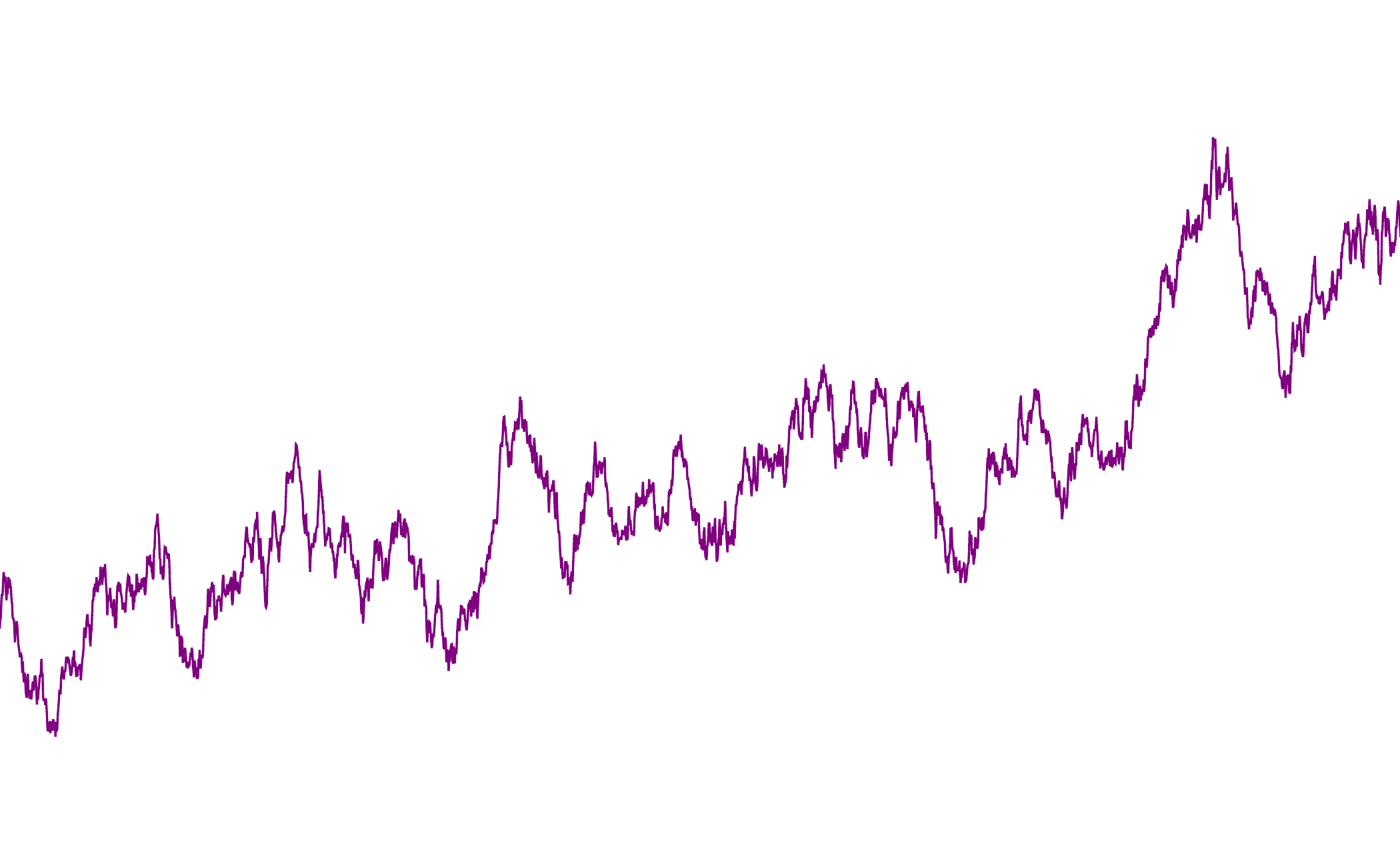}};
   \node[circle,fill=blue!75,label=left:$a$,inner sep=1.75pt] (a) at (0.0,2.5) {};
   \node[circle,fill=blue!75,label=right:$b$,inner sep=1.75pt] (b) at (9.143,5.2) {};
   \node[circle,fill=orange,label=left:$x$,inner sep=1.75pt] (x) at (0.0,1.45) {};
   \node[circle,fill=orange,label=right:$y$,inner sep=1.75pt] (y) at (9.143,4) {};
  \draw[blue!75!black,thick] (a) edge [bend left=10] node[pos=0.55,above,font=\tiny,sloped,inner sep=0pt] {$f_a^b(s;T)+C(s\wedge(T-s))^{1/2-\epsilon}$} (b);
  \path[blue!75!black,thick] (a) edge [bend right=10] 
  node[pos=0.6,below,font=\tiny,sloped,inner sep=0pt] {$f_a^b(s;T)-C(s\wedge(T-s))^{1/2-\epsilon}$} 
  (b);
   \draw[thick,orange] (x) to [bend right=40] 
   node[below,sloped,font=\tiny,pos=0.3,inner sep=0pt] {$f_x^y(s;T)-C(s\wedge(T-s))^{1/2+\epsilon}$} (y);
  \draw[blue!75, thin, dashed] ($(a)+(0.02,0.02)$) -- (b);
  \draw[orange, thin, dashed] ($(x)+(0.02,0.02)$) -- (y);
  \draw[dashed, thick,gray] (0.75,5.5) -- (0.75,0);
  \draw[dashed, thick,gray] (8.3,5.5) -- (8.3,0);
  \draw (.75,0.05) -- (.75,-0.05) node[below] {$1$};
  \draw (8.3,0.05) -- (8.3,-0.05) node[below] {$T-1$};
  \draw (9.2,0.05) -- (9.2,-0.05) node[below] {$T$};
 \end{tikzpicture}
 
     \caption{Lemma~\ref{lem:brz-lem2.1-bridge}: the probability that a Brownian bridge on $[0,T]$ from $x$ to $y$ (purple) stays below the blue curve on $[1,T-1]$ is asymptotically equivalent (in the sense of $\uasymp$) to the probability of staying in the orange region on $[1,T-1]$. In Lemma~\ref{lem:brz-modified}, we show that if $a-x$ and $b-y$ are bounded away from $0$, the interval can be extended to $[0,T]$.}
     \label{fig:lemma-BRZ-2.1}
 \end{figure}
 
Lemma~\ref{lem:brz-lem2.1-bridge} gives a ballot theorem-type result for the truncated time interval $[1,T-1]$. To derive a result for the entire time interval $[0,T]$, we of course must assume that $a-x$ and $b-y$ are bounded away from $0$, and that the lower barrier is bounded away from $x$ and $y$ at time $0$. These necessary assumptions are sufficient to obtain a ballot-type result on $[0,T]$:
\begin{lemma}
\label{lem:brz-modified}
Fix any $\epsilon \in (0,1/2)$, $\eta > 1$, $C>0$, $\delta_1>0$, $\delta_2 >0$, and $\delta_3 >0$.
Then the following holds uniformly in $x+ \delta_1  \leq a$, $y+\delta_2 \leq b$, $(a-x)(b-y) \leq \eta T$, and $T$ large enough:
\begin{align}
    &\P_{x,T}^y \pth{
    \LB{[0,T]}{f_x^y(s;T) - \delta_3 - 
    C(s\wedge (T-s))^{\frac{1}{2}+ \epsilon}}(W_{\cdot})
    \cap
    \UB{[0,T]}{f_a^b(s;T) - C(s\wedge (T-s))^{\frac{1}{2}-\epsilon}}(W.)}
    \nonumber \\
    &\asymp
    \P_{x,T}^y \pth{\UB{[0,T]}{
    f_a^b(s;T) + C(s\wedge (T-s))^{\frac{1}{2}-\epsilon}
    }(W_{\cdot})
     } 
  \asymp 
    \frac{(a-x)(b-y)}{T} \, .
    \label{eqn:brz-modified}
\end{align}

\end{lemma}

\begin{proof}[Proof of Lemma~\ref{lem:brz-modified}]
For any $\cS \subset [0,T]$, define the events 
\begin{align*}
    \cA(\cS) &:= 
    \LB{\cS}{f_x^y(s;T) - \delta_3 - 
    C(s\wedge (T-s))^{\frac{1}{2}+ \epsilon}}(W_{\cdot})
    \cap
    \UB{\cS}{f_a^b(s;T) - C(s\wedge (T-s))^{\frac{1}{2}-\epsilon}}(W_{\cdot})
    \\
    \overline{\cA}(\cS) &:= 
    \UB{\cS}{
    f_a^b(s;T) + C(s\wedge (T-s))^{\frac{1}{2}-\epsilon} 
    }(W_\cdot) \, .
\end{align*}
It is trivial that $\P_{x,T}^y(\cA([0,T])) \leq \P_{x,T}^y( \bar{\cA}([0,T]) )$. For the other direction, we see from Lemma~\ref{lem:brz-lem2.1-bridge} that there exists a constant $c_1>0$ such that 
\begin{equation}
\P_{x,T}^y \pth{ \cA([1,T-1]) }
\geq c_1 \P_{x,T}^y\pth{ \bar{\cA}([1,T-1]) }
\geq 
c_1 \P_{x,T}^y\pth{ \bar{\cA}([0,T])  }
\, ,
    \label{eqn:brz-modified.1}
\end{equation}
for all $x,a,y,b,$ and $T$ satisfying the given conditions. 
Since $f_x^y(0;T)-\delta_{1} \leq x \leq a-\delta_{2}$ and $y\leq b-\delta_3$, 
there exists a constant $c_2 >0$ (depending on $\delta_{1}$, $\delta_{2}$, and $\delta_3$)
\[
\P_{x,T}^y \pth{\cA\big ([0,1] \cup [T-1,T] \big ) \given  \cA([1,T-1]) } \geq c_2 \, ,
\]
for all $x,a,y,b,$ and $T$ satisfying the given conditions. Combining this with~\eqref{eqn:brz-modified.1} yields 
\[
\P_{x,T}^y\pth{ \cA([0,T])  }
\geq c_1c_2   \P_{x,T}^y\pth{ \bar{\cA}([0,T])  } \, .
\]
This proves the first $\asymp$ statement in~\eqref{eqn:brz-modified}. The second statement then follows from~\eqref{eqn:brownian-ballot-explicit} and the inequality
$
\P_{x,T}^y ( \cA([0,T]) )
\leq 
\P_{x,T}^y \Big( \UB{[0,T]}{f_a^b(s;T)}(W_\cdot) \Big)
\leq 
\P_{x,T}^y\pth{ \bar{\cA}([0,T]) }
$.
\end{proof}

\section{Proof of the main theorem,  Theorem~\ref{thm:limiting-law}}
\label{sec-3}

Our goal is to understand $\P(R_t^* \leq m_t +y )$ as $t \to \infty$. Towards this end, we will argue that particles exceeding $m_t+y$ at time $t$ will typically follow a prescribed trajectory. In particular---a fact which will significantly simplify our analysis---it will suffice to reduce to particles that at a sufficiently large, fixed (independently of $t$) time $L$ lie in a certain \emph{window} of order roughly 
$\sqrt{L}$ below  $\sqrt 2 L$. 
Namely, recall $I_L^{\win}$ and $\cN_L^{\win}$ from~\eqref{eq:def-window}.
Theorem~\ref{thm:outside_window} below states that the effect of particles $v\in \cN_L \setminus \cN_L^\win$ (outside the window) is negligible to $\P(R_t^* > m_t+y)$.
\begin{theorem} \label{thm:outside_window}
For all $y \in \R$, we have
\begin{align}
        \lim_{L\to \infty} 
        \limsup_{t\to\infty}
        \P
        \pth{\exists v \in \cN_t :
              ~R_L^{(v)} \not \in I_L^{\win}\,,\,
            R_t^{(v)} > m_t + y
            }
    = 0 \,.
    \label{eq.rough_window.main}
\end{align}
In particular, 
\begin{equation}
\lim_{L \to \infty} \limsup_{t\to\infty} \abs{\P(  R^*_t > m_t+y) - \P \pth{ \exists v \in \cN_t : R_L^{(v)} \in I_L^{\win}\,,\, R_t^{(v)} >m_t+ y} } = 0 \, .
    \label{eq.rough_window.main2}
\end{equation}
\end{theorem}
The proof of Theorem~\ref{thm:outside_window} is outlined at the beginning of Section~\ref{sec:window}, and given in Section~\ref{subsec:rough_window}.
As mentioned before, our calculations there actually imply that the exponents $1/6$ and $2/3$ in the definition of $I_L^{\win}$~\eqref{eq:def-window} need only be smaller than $1/4$ and larger than $1/2$ respectively,
though it should be true for any exponents smaller and larger than $1/2$, respectively.

In light of Theorem~\ref{thm:outside_window} and the Markov property, our main effort in the analysis will be devoted to studying
Bessel particles $v \in \cN_{t-L}$ satisfying  $R_0^{(v)}\in I_L^{\win}$ in addition to $R_{t-L}^{(v)}> m_t+y$.
We can immediately see the advantage of the window here: the factor $W_0^{-\alpha_d}$ in the Girsanov transform~\eqref{eq.girsanov} (which will be employed repeatedly) is bounded uniformly away from $0$ over such particles.
Thus, 
we fix any
\begin{align}
    z \in [L^{1/6}, L^{2/3}]
    \label{eqn:def-z}
\end{align}
(so that $\sqrt{2}L - z \in I_L^{\win}$) and
consider the probability that a branching Bessel process started (at time $0$) from $\sqrt{2}L-z$ has maximum at time $t-L$ exceeding the target $m_t+y$.
 Theorem~\ref{thm:right_tail_asymp} gives a precise asymptotic for the probability of this event, \textit{uniformly} over $z$.
\begin{theorem}\label{thm:right_tail_asymp}
Define the quantity
\begin{equation}
    \fM_{L,z} := (\sqrt{2}L-z)^{-\alpha_d} z e^{-(z+y)\sqrt{2}} \, .
    \label{def:M_Lz}
\end{equation}
There exists an absolute constant $\gamma^* >0 $
such that for all $y \in \R$,
\begin{align}
    \lim_{L \to \infty} 
    \limsup_{t\to\infty}
    \sup_{z \in [L^{1/6}, L^{2/3}]}
    \abs{ \frac{\P_{\sqrt{2}L-z}( R_{t-L}^* > m_t + y) }
    { \gamma^*\fM_{L,z}} -1 }
    =
    0 \,.
    \label{eqn:right_tail_exact}
\end{align}
\end{theorem}
In Section~\ref{sec:mod2ndmom}, we will describe in detail the modified second moment method that drives the proof of Theorem~\ref{thm:right_tail_asymp} and forms the technical heart of 
the paper. In Section~\ref{subsec:proof-thm1}, we will show how Theorem~\ref{thm:limiting-law} follows from Theorems~\ref{thm:outside_window} and~\ref{thm:right_tail_asymp}. First, we establish some notation that will be used throughout the remainder of this paper.

\subsection{Asymptotic Notation}
\label{subsec:asymp-notation}

Throughout the paper, we will compare the asymptotic behavior of two positive functions with dependence on $t$, $L$, and $z$, where  
$z$ is defined as in~\eqref{eqn:def-z}, and we always send $L \to \infty$ after $t\to \infty$. 
It will be convenient to establish notation for asymptotic comparisons of such functions. For functions $f:=f(t,L,z)$ and $g:=g(t,L,z)$, 
we write
$f \usim g$ to denote the relation
\begin{align*}
    \lim_{L \to \infty} 
    \liminf_{t\to\infty}
    \inf_{z \in [L^{1/6}, L^{2/3}]}
    \frac{f}{g} =
    \lim_{L \to \infty}  \limsup_{t\to\infty}
    \sup_{z \in [L^{1/6}, L^{2/3}]}
    \frac{f}{g} =
    1 \,.
\end{align*}
For instance,~\eqref{eqn:right_tail_exact} is equivalent to the relation 
$
    \P_{\sqrt{2}L-z}( R_{t-L}^* > m_t + y) 
    \usim \gamma^*\fM_{L,z}
$.
We  write
$f= o_u(g)$ if 
\[
    \limsup_{L \to \infty} \limsup_{t \to \infty} \sup_{z \in [L^{1/6}, L^{2/3}]}  \frac{f}{g} = 0 \,.
\]
For instance, the probability in the left-hand side of~\eqref{eq.rough_window.main} is $o_u(1)$. We will write
$f \ulesssim g$ if there exists some constant $C> 0$ such that
    \begin{equation*}
        \limsup_{L \to \infty} \limsup_{t\to \infty} 
        \sup_{z \in [L^{1/6}, L^{2/3}]}
        \frac{f}g  < C\,. 
    \end{equation*}
Lastly, we will write $f \uasymp g$ if $f \ulesssim g$ and $g \ulesssim f$.

\subsection{Proof of  Theorem~\ref{thm:limiting-law} (assuming Theorems~\ref{thm:outside_window} and~\ref{thm:right_tail_asymp})}
\label{subsec:proof-thm1}
From Theorem~\ref{thm:outside_window}, we have
\begin{align*}
    \P(R_t^* > m_t+y) -o_u(1) 
    &=
    1- 
    \P\bigg(
        \bigcap_{v \in \cN_L^{\win}}
        \bigcap_{u \in \cN_{t-L}^v}
        \{  R^{(u)}_t \leq m_t+y \}
    \bigg)
    \nonumber \\
    &=
    1- \E \bigg [
        \prod_{v \in \cN_L^{\win}}
        \P \bigg(
        \bigcap_{u \in \cN_{t-L}^v}
        \{  R^{(u)}_t \leq m_t+y \} \given \cF_L
    \bigg)
    \bigg ].
\end{align*}
Let $g_t(y):= \P( R_t ^* \leq m_t+y)$. Then rearranging terms in the previous display gives
\begin{align*}
    g_t(y) =
    \E \bigg[
        \prod_{v \in \cN_L^{\win}}
        \bigg( 1-
        \P \bigg(
        \bigcup_{u \in \cN_{t-L}^v}
        \{  R^{(u)}_t > m_t+y \} \given \cF_L
        \bigg)
    \bigg)
    \bigg]
    + o_u(1).  
\end{align*}
Define $z_v := \sqrt{2}L - R_L^{(v)}$ and  $A_v := (R_L^{(v)})^{-\alpha_d} z_v e^{-(z_v +y) \sqrt{2}}$.
By the Markov property, 
\begin{align*}
    \P \bigg(
        \bigcup_{u \in \cN_{t-L}^v}
        \{  R^{(u)}_t > m_t+y \} \given \cF_L
        \bigg) =
    \P_{\sqrt{2}L-z_v}(R_{t-L}^* > m_t+y)\,.
\end{align*}
Thus, Theorem~\ref{thm:right_tail_asymp} 
implies that there exists a sequence $\epsilon_L\geq 0$ satisfying $\epsilon_L \to_{L \to \infty} 0$ such that for every $L$ and all $t$ sufficiently large (compared to $L$), we have
\begin{align}
    \E \bigg[
        \prod_{v \in \cN_L^{\win}}
        \!\!\!\pth{ 1-
        \pth{ 1+ \epsilon_L}
        \gamma^*
        A_v
    }
    \bigg]
    &\leq 
    g_t(y) + o_u(1)
    \leq 
    \E \bigg[
        \prod_{v \in \cN_L^{\win}}
        \!\!\!\pth{ 1-
        \pth{ 1- \epsilon_L}
        \gamma^*
        A_v
    }
    \bigg] \,.
    \label{eqn:thm1-preineq}
\end{align}
For all particles $v \in \cN_L^{\win}$,  we have $z_v \in [L^{1/6}, L^{2/3}]$, and thus $A_v \to_{L \to \infty} 0$ deterministically. Note that for all $x \in [0,1)$, we have 
$1-x \geq  e^{-x + \log (1-x^2 e^x)} \geq e^{-x(1+ \frac{xe^x}{1-x^2e^x})}$ and 
$1-x \leq e^{-x}$ for all $x \in \R$. Thus,
for all $L$ sufficiently large and $v \in \cN_L^{\win}$,
we have the bounds
    \[ 
    1-(1 - \epsilon_L)\gamma^* A_v 
    \leq 
    e^{-(1- \epsilon_L) \gamma^* A_v} 
\quad \text{and} \quad 
    1-(1 + \epsilon_L)\gamma^* A_v \geq
    e^{-(1+ \epsilon_L') \gamma^* A_v}
\]
for another nonnegative sequence $\epsilon_L'$ such that $\epsilon_L \leq \epsilon_L' \to 0$ as $L \to \infty$. 
Applying these inequalities to~\eqref{eqn:thm1-preineq} yields
\begin{align*}
    \E \brak{
        \exp \pth{- (1+ \epsilon_L')
        \gamma^* e^{- y \sqrt{2}} 
        Z_L
    }
    }
    \leq
g_t(y) + o_u(1)
    \leq 
    \E \brak{
        \exp \pth{- (1- \epsilon_L')
        \gamma^* e^{- y \sqrt{2}} 
        Z_L
    }
    } \,,
\end{align*}
for all $L$ sufficiently large and for all $t$ sufficiently large (compared to $L$),
where we recall $Z_L$ from~\eqref{def:Z_L}.
Subtract $\E[ \exp(- \gamma^* e^{-y \sqrt{2}} Z_L)]$ from each side of the above, and 
note that the differences with the left-hand side and right-hand side both go to $0$ as $L \to \infty$, since $\exp\pth{ -e^{-x}}$ is uniformly continuous in $x$. Thus, there exists a non-negative sequence $\epsilon_L'' \to_{L \to \infty} 0 $ such that for all $L$ sufficiently large and for all $t$ sufficiently large (compared to $L$)
    \begin{align}
        \abs{ g_t(y) - \E[ \exp(- \gamma^* e^{-y \sqrt{2}} Z_L)] } = \epsilon_L'' \,. \label{eq.fast.3} 
    \end{align}
Now, since $R_t^* - m_t$ is tight, there exists a subsequence $\{t_i\}_{i \in \N}$ along which $g_{t_i}(y)$ converges to a distribution function $g(y)$. Since $g_t(y)$ has no $L$ dependence while $\E[ \exp(- \gamma^* e^{-y \sqrt{2}} Z_L)]$ has no $t$ dependence, taking the limit of the right-hand side of~\eqref{eq.fast.3} first as $t_i \to \infty$,  then as $L \to \infty$, yields 
\begin{align}
    \lim_{L \to \infty} \E[ \exp(- \gamma^* e^{-y \sqrt{2}} Z_L)]
     = g(y)\,. 
     \label{eqn:thm1-convL}
\end{align}
It follows that every subsequential limit of $\P(R_t^*\leq m_t +y)$ is equal to $g(y)$, and thus, by Prokhorov's theorem,  $ \lim_{t \to \infty} g_t(y) = g(y)$ as well. Thus, we have proved that $R_t^*-m_t$ converges in distribution. 
Further, via the Laplace transform, Eq.~\eqref{eqn:thm1-convL} shows $Z_L$ converges to $Z_{\infty}$ in law and thus $g(y) = \E[\exp( -\gamma^* e^{-y\sqrt{2}}Z_{\infty})]$.
\qed

\section{Proof of Theorem~\ref{thm:outside_window}: confinement to \texorpdfstring{$I_L^\win$}{I-L-win}}
\label{sec:window}
Unless otherwise stated, we fix $d \geq 3$ in this section. See Section~\ref{sec-2d} for the $d=2$ case.

In Section~\ref{subsec:rough_window}, we prove Theorem~\ref{thm:outside_window} via the following strategy. We are interested in bounding the probability
$\P( \exists v \in \cN_t : R_L^{(v)} \not\in I_L^{\win}; R_t^{(v)} > m_t+y )$ from above. Naively applying a union bound and the many-to-one lemma (Lemma~\ref{lem:many-to-one}) yields the upper bound
\[
e^t\P (R_L \not \in I_L^{\win}, R_t > m_t+y)\,,
\]
where $R_.$ denotes a $d$-dimensional Bessel process; however, it can be checked that the above probability does not exhibit enough decay to overpower the $e^t$ term, indicating that the union bound was too rough. The fix is Lemma~\ref{lem:rough_parabola}, which states that, with high probability, no particle ever crosses the ``barrier function" $B(s)$, defined in~\eqref{def:B(s)}, on $[L, t]$. Once this upper barrier on particle trajectories is established, we can apply the union bound and many-to-one lemma to $\P \big( \exists v \in \cN_t :  R_L^{(v)} \not \in I_L^{\win}, \UB{[L, t]}{B} (R^{(v)}_\cdot), R^{(v)}_t \in (m_t+y, B(t)) \big)$, which can be estimated by integrating over the Bessel density at time $L$ and applying the Girsanov transform~\eqref{eq.girsanov} on $[L,t]$, which in turn allows us to consider Brownian motion on $[L,t]$, and thus the barrier estimates of Section~\ref{subsec:hittingP} can be applied. 
One point of concern is the~$W_0^{-\alpha_d}$ term in~\eqref{eq.girsanov} (for our application, this will become $R_L^{-\alpha_d}$), which of course blows up at $0$; however, the integral will eventually reduce to Lemma~\ref{lem:bessel_time_L}, which gives sufficient decay. Lemmas~\ref{lem:bessel_time_L} and \ref{lem:rough_parabola} are proved in Section~\ref{subsec:parabola-bounds}.
In Section~\ref{subsec:B0}, we prove Lemma~\ref{lem:parabola_B0}, which is needed for the proof of Proposition~\ref{prop:right_tail_equiv}; its proof is similar to that of  Lemma~\ref{lem:rough_parabola}, and so we give it here.

\subsection{Preliminaries: an upper barrier}
\label{subsec:parabola-bounds}
Define the barrier function
\begin{align}
    B(s) := B(s; t,L)  = \frac{m_t}{t}s + \mathcal{C}_d \log \pth{s\wedge (t-s)}_+ + \log L\, , ~\text{for } s \in [L,t]\,,
    \label{def:B(s)}
\end{align}
where $\mathcal{C}_d > 0$ is a sufficiently large constant (precisely, we fix any $\mathcal{C}_d$ satisfying the condition of Lemma~\ref{lem:rough_parabola}). 
We begin with Lemma~\ref{lem:bessel_time_L}, a technical estimate on a certain integral over the Bessel density; the lemma
will show up in our proofs of both Lemma~\ref{lem:rough_parabola} and Theorem~\ref{thm:outside_window}.
For a time-homogeneous Markov process $X_\cdot$, we  write $p^X_s(x,y)$ to denote the transition density of $X_s$ given $X_0 =x$. Further,
let 
\begin{equation}\label{eq:fo_t-def}
\mathfrak{o}_t := \tfrac{m_t}{t}- \sqrt{2} = \fc_d \tfrac{\log t}{t}\,,
\end{equation}
where we recall $\fc_d$ from~\eqref{def:m_t-c_d}, and
define the intervals
\begin{align*}
    I_L &:= [0,\mathfrak{o}_t L + L^{1/6}+ (\mathcal{C}_d+1)  \log L ] \,,\\
    J_L &:= B(L) - I_L^{\win} = [\mathfrak{o}_t  L + L^{1/6}+(\mathcal{C}_d+1) \log L, \mathfrak{o}_t L +  L^{2/3}+(\mathcal{C}_d+1) \log L ]\,,\\
    I'_L&:=[\mathfrak{o}_t L + L^{2/3}+(\mathcal{C}_d+1) \log L, B(L)] \,.
\end{align*}
Note that $[0,B(L)] \setminus J_L = I_L \cup I_L'$.

\begin{lemma}
\label{lem:bessel_time_L}
Let $R_s$ be a $d$-dimensional Bessel process, for any $d \in \Z_{\geq 1}$, and define $B(s)$ as in~\eqref{def:B(s)}. Then
\begin{align}
    e^L L^{\mathcal{C}_d \sqrt{2}}
    \int_{[0,B(L)]\setminus J_L} \frac{p_L^R(0, B(L)-w)}{\pth{B(L) - w}^{\alpha_d}}
    (1+w)e^{-w \sqrt{2}}  \d w
    &\ulesssim L^{-\frac{6\sqrt{2}+1}{6}} \, , \text{ and} 
    \label{eq.bessel.c}
    \\
    e^L L^{\mathcal{C}_d \sqrt{2}}
    \int_{J_L} \frac{p_L^R(0, B(L)-w)}{\pth{B(L) - w}^{\alpha_d}}
    (1+w)e^{-w \sqrt{2}} \d w
    &\uasymp 
    L^{-\frac{2\sqrt{2}-1}{2}} \, .
    \label{eq.bessel.window}
\end{align}

\begin{proof}
Let $R_s = \|(B_1(s), \dots, B_d(s))\|$ for i.i.d.\ standard  Brownian motions $B_1, \dots, B_d$. 
Let $\chi_d$ have the chi distribution of $d$ degrees of freedom. Then 
$R_L \overset{(\mathrm{d})}{=}  L^{1/2} \chi_d.$ 
Letting $p^{\chi_d}(y)$ denote the density of $\chi_d$, it follows that we have the following estimates uniformly over all $w \in [0, B(L)]$:
\begin{align}
    p_L^R \pth{  0, B(L)-w  }
    &= L^{-1/2} 
    p^{\chi_d} (  L^{-1/2} (B(L) - w) )
    \nonumber \\
    &\uasymp
    L^{-1/2}
    \frac{(B(L) -w)^{d-1}}{L^{\frac{d-1}{2}}}
    e^{ - \frac{(B(L)- w)^2}{2L}} 
    \nonumber \\
    &
    \uasymp 
    \frac{(B(L) -w)^{d-1}}{L^{\frac{d-1}{2}}}
    L^{-\frac{1}{2} -(1+\mathcal{C}_d) \sqrt{2}} e^{- \frac{w^2}{2L} + (\sqrt{2}+ \mathbf{f}_L)w -L} \,,
    \label{eq.bessel.density}
\end{align}
where $\mathbf{f}_L :=     ((\mathcal{C}_d+1)\log L)/L$.
Recalling $\alpha_d$ from~\eqref{def:m_t-c_d}, we then have that for any $I \subseteq [0, B(L)]$:
\begin{align}
    \Upsilon_I &:= e^L L^{\mathcal{C}_d \sqrt{2}}
    \int_I \frac{p_L^R \pth{ 0, B(L)-w }}{\pth{B(L) - w}^{\alpha_d}}
    (1+w)e^{-w \sqrt{2}} \d w
    \nonumber \\
    &\uasymp
    L^{-\frac{1+2\sqrt{2}}{2}}
    \int_I
    \big( \frac{B(L) - w}{L} \big)^{\frac{d-1}{2}} (1+w)
    e^{ \mathbf{f}_L w}
    e^{- \frac{w^2}{2L}} \d w \, .
    \label{eq.bessel.sub}
\end{align}
Note that for $w \in [\mathfrak{o}_t L +L^{2/3}+(\mathcal{C}_d \log L+1), B(L)]$, we have
    $e^{ \mathbf{f}_Lw} \ulesssim L^{C}$, where $C= \sqrt{2}(\cC_d +1)$, and
    $B(L) - w \ulesssim L$. 
Thus, for some positive constants $C$ and $C'$, we have
\begin{align}
  \Upsilon_{I_L'} &\ulesssim  L^{C'} 
    \int_{\mathfrak{o}_t L + L^{2/3}+ (\mathcal{C}_d+1) \log L}^{B(L)} (1+w) e^{-\frac{w^2}{2L}} \d w
    \ulesssim 
    L^{C''} e^{-L^{1/3}/2}\,. 
    \label{eq.bessel.4}
\end{align} 
For $w < \mathfrak{o}_t L  + L^{2/3}+ (\mathcal{C}_d+1) \log L$, we have 
\begin{align}
e^{ \mathbf{f}_Lw} \uasymp 1\,, \quad \text{and} \quad 
B(L) - w \uasymp L \,, \label{eq.bessel.2}
\end{align}
so that 
\begin{align}
    \Upsilon_{I_L} &\uasymp 
    L^{-\frac{1+2\sqrt{2}}{2}} \int_0^{\mathfrak{o}_t L + L^{1/6}+(\mathcal{C}_d+1) \log L} (1+w) e^{-\frac{w^2}{2L}} \d w \ulesssim L^{-\frac{6\sqrt{2}+1}{6}} \, . \label{eq.bessel.5}
\end{align}
Equations~\eqref{eq.bessel.4} and~\eqref{eq.bessel.5} then give equation~\eqref{eq.bessel.c}. Finally, equation~\eqref{eq.bessel.sub} and ~\eqref{eq.bessel.2} give
\begin{align*}
    \Upsilon_{J_L}
    \uasymp L^{-\frac{1+2\sqrt{2}}{2}} \int_{J_L} (1+w) e^{-\frac{w^2}{2L}} \d w \uasymp L^{-\frac{2\sqrt{2}-1}{2}}\,.
\end{align*}
This shows equation~\eqref{eq.bessel.window}.
\end{proof}
\end{lemma}

\begin{lemma} \label{lem:rough_parabola}
There exists a constant $C := C(d) >0$ such that for all $\mathcal{C}_d \geq C$, 
we
have
\begin{align}
    \P\bigg( 
        \bigcup_{s \in [L, t]}
        \bigcup_{v\in \cN_s} 
        \{R_s^{(v)} > B(s) \}
    \bigg) \ulesssim L^{-\frac{2\sqrt{2}-1}{2}}.
    \label{eqn:no_parabola_UB}
\end{align}
\end{lemma}

\begin{proof}[Proof of Lemma~\ref{lem:rough_parabola}]
We consider separately three distinct time intervals during which a particle might exceed the barrier $B(s)$: $[L, t/2], [t/2, t-t^{1/8}]$, and $[t-t^{1/8}, t]$.

For this first interval, we consider separately the case $s =L$, and then employ a union bound over $s \in (L, t/2]$. 
For brevity, write 
\[ \Xi_s(v) := \{ R_s^{(v)}> B(s)\} \quad \text{and} \quad  \Xi^*_{s}(v) :=  \{ R_L^{(v)} \leq B(L) \} \cap \{ \sup_{s' \in [s, s+1]} R_{s'}^{(v)} > B(s) \}\,.\]
Then we find
\begin{align}
    \P \bigg( \bigcup_{s \in [L, t/2]} \bigcup_{v \in \cN_{s}} \Xi_s(v)\bigg)
    &\leq 
    \P \bigg(
        \bigcup_{v \in \cN_L} \Xi_L(v) 
    \bigg)
    + 
    \P \bigg( \bigcup_{s \in (L, t/2]} \bigcup_{v \in \cN_{s}} \{ R_L^{(v)} \leq B(L) \} \cap \Xi_s(v) \bigg)
    \nonumber \\
    &\leq 
    \P \bigg(
        \bigcup_{v \in \cN_L} \Xi_L(v) 
    \bigg)
    + 
    \sum_{s-L = 0}^{\floor{ \frac{t}{2} - L}} 
    \P \bigg( \bigcup_{v \in \cN_{s+1}} \Xi^*_s(v) \bigg ) \, .
    \label{eqn:parab1-s<t/2-1}
\end{align}
We begin with the first term on the last line of~\eqref{eqn:parab1-s<t/2-1}. A trivial upper bound and the many-to-one lemma yield the following upper bound
\begin{align}
    \P \bigg (
        \bigcup_{v \in \cN_L} \Xi_L(v) 
    \bigg )
    &\leq 
    \E \bigg [
    \sum_{v \in \cN_L} \one_{\Xi_L(v)}
    \bigg ] =
    e^L \P ( R_L > B(L) ) \, ,
    \nonumber 
\end{align}
where $R_s$, as usual, is a Bessel process of dimension $d$. Analyzing the chi-distribution with $d$ degrees of freedom bounds the previous expression from above by 
\begin{align}
    e^L q(L) e^{- \frac{B(L)^2}{2L}} 
    \ulesssim q(L) L^{-(1+\mathcal{C}_d) \sqrt{2}} 
    \, , 
    \nonumber 
\end{align}
where $q(\cdot)$ is a polynomial of degree  depending only on $d$.  Let $\deg(q)$ denote the degree of $q$. Then for all $\mathcal{C}_d > (\frac{1}{2}+ \deg(q))/\sqrt{2}-1$, we have the desired bound
\begin{align}
    \P \bigg(
        \bigcup_{v \in \cN_L} \Xi_L(v)
    \bigg) \ulesssim L^{-1/2} \, ,
    \label{eqn:parab1-s=L}
\end{align}
as desired. 

We next turn our attention to the summation in the last line of~\eqref{eqn:parab1-s<t/2-1}. Fix any $s \in L +  \llbracket 0, \floor{ \frac{t}{2} - L} \rrbracket $, and employ the usual upper bound and many-to-one lemma to find
\begin{align*}
    \P \bigg (  \bigcup_{v \in \cN_{s+1}} \Xi^*_s(v) 
    \bigg )
    \leq
    e^{s+1} \P \pth{  \Xi^*_{s+1}(v)
    } \, .
\end{align*}
Expanding the right-hand side of the above by integrating over $w := B(L) - R_L$, applying the Markov property, and then employing the Girsanov transform (where we bound the $\exp(\cdot) \one_{\{\cdot\}}$ term in the right-hand side of~\eqref{eq.girsanov} by $1$) yields
\begin{align}
    &e^{s+1} 
    \int_0^{B(L)} 
    p_L^R(0, B(L)-w)
    \P_{B(L)-w} \Big ( \sup_{s' \in [s, s+1]} R_{s'-L} > B(s)
    \Big )
    \d w
    \nonumber \\
    &\leq e^{s+1}
    \int_0^{B(L)} 
    \frac{p_L^R(0, B(L)-w)}{\pth{B(L) - w}^{\alpha_d}}
    \E_{B(L)-w} \brak{ 
    W_{s+1-L}^{\alpha_d} 
    \one_{ \{ \sup_{s' \in [s-L, s+1-L]} W_{s} > B(s) \} }
    } 
    \d w \, .
    \label{eqn:parab1-s<t/2-gir}
\end{align}
Now, define $M_{s+1 - L} := \sup_{s' \in [0, s+1-L]} W_{s'}$, and recall that the density function of $M_{s+1-L}$ is twice that of $W_{s+1-L}$ on $[0, \infty)$. Then the final display of~\eqref{eqn:parab1-s<t/2-gir} is bounded above by 
\begin{align}
    &e^{s+1}
    \int_0^{B(L)} 
    \frac{p_L^R(0, B(L)-w)}{\pth{B(L) - w}^{\alpha_d}}
    \E_{B(L)-w} \brak{ 
    M_{s+1-L}^{\alpha_d} 
    \one_{ \{ M_{s+1-L} > B(s) \} }
    } \d w 
    \nonumber \\
    &=
    e^{s+1}\int_0^{B(L)} 
    \frac{p_L^R(0, B(L)-w)}{\pth{B(L) - w}^{\alpha_d}}
    \E \brak{ 
    (B(L)-w+ M_{s+1-L})^{\alpha_d} 
    \one_{ \{ M_{s+1-L} > B(s) - B(L) + w \} }
    } 
    \d w 
    \nonumber \\
    &\uasymp 
    e^s (s+1-L)^{-1/2} 
    \int_0^{B(L)} \d w~
    \frac{p_L^R(0, B(L)-w)}{\pth{B(L) - w}^{\alpha_d}}
    \int_0^{\infty}
    \d x~
    (B(s)+x)^{\alpha_d} 
    e^{ - \frac{ (B(s)-B(L) + x+w)^2 }{2(s+1-L)} } 
    \label{eqn:parab1-s<t/2-gaussint}
\end{align}
Recall that $\fc_d := \frac{d-4}{2\sqrt{2}}$, which for $d\geq 2$ is of course minimized by $c_2 = -1/\sqrt{2}$. Then expansion of the Gaussian density in the previous display gives
\begin{align}
    e^{ - \frac{ (B(s)-B(L) + x+w)^2 }{2(s+1-L)} }
    \ulesssim
    e^{-s+L} (s/L)^{-\mathcal{C}_d \sqrt{2}} 
    \Psi_s e^{-x\sqrt{2}} e^{-w \sqrt{2}} \, ,
    \label{eqn:parab1-s<t/2-gaussian}
\end{align}
where $\Psi_s := e^{(s-L) \log t /t }$. Note that
\begin{align}
    \int_0^{\infty} (B(s)+x)^{\alpha_d} e^{-x\sqrt{2}} \d x \ulesssim s^{\alpha_d} \, ,
    \label{eqn:parab1-s<t/2-intx}
\end{align}
and from Lemma~\ref{lem:bessel_time_L} (and the fact that $1+w \geq 1$), we have
\begin{align}
    \int_0^{B(L)} 
    \frac{p_L^R(0, B(L)-w)}{\pth{B(L) - w}^{\alpha_d}} e^{-w \sqrt{2} } \d w
    \ulesssim 
    e^{-L} L^{-\mathcal{C}_d \sqrt{2}-\frac{2\sqrt{2}-1}{2}} \, .
    \label{eqn:parab1-s<t/2-lem4.1}
\end{align}
Substituting~\eqref{eqn:parab1-s<t/2-gaussian}---\eqref{eqn:parab1-s<t/2-lem4.1} into~\eqref{eqn:parab1-s<t/2-gaussint}, we find 
\begin{align}
    \P \bigg ( \bigcup_{v \in \cN_{s+1}} \Xi^*_s(v) \bigg )
    \ulesssim L^{-\frac{2\sqrt{2}-1}{2}} s^{\alpha_d-\mathcal{C}_d\sqrt{2}} \Psi_s  \, .
    \label{eqn:parab1-s<t/2-cup-s}
\end{align}
Now, for $s \leq L+ t/\log t$, we have $\Psi_s \leq e$. Take $\mathcal{C}_d > \alpha_d/\sqrt{2}$; then 
\begin{align}
    \sum_{s-L = 0}^{\floor{t/\log t}} 
    \P \bigg ( \bigcup_{v \in \cN_{s+1}} \Xi^*_s(v) \bigg ) \ulesssim L^{-\frac{2\sqrt{2}-1}{2}} \, .
    \label{eqn:parab1-<t/logt}
\end{align}
For $s \in [L + \floor{t/\log t}, t/2]$, $\Psi_s \leq t^{1/2}$, but $s^{-\mathcal{C}_d \sqrt{2} } \ulesssim t^{-\mathcal{C}_d}$. It then follows from~\eqref{eqn:parab1-s<t/2-cup-s} that 
\begin{align}
    \sum_{s-L = \floor{t/\log t}}^{\floor{ \frac{t}{2} - L}} 
    \P \bigg ( \bigcup_{v \in \cN_{s+1}} \Xi^*_s(v) \bigg )
    &\ulesssim L^{-\frac{2\sqrt{2}-1}{2}} t^{\frac{1}2} (\tfrac{t}{\log t})^{\alpha_d + 1-\mathcal{C}_d}
    \ulesssim L^{-\frac{2\sqrt{2}-1}{2}} \, ,
    \label{eqn:parab1-t/logt<} 
\end{align}
where the last line follows for all $\mathcal{C}_d >\alpha_d + \frac{3}{2}$.
Thus, equations~\eqref{eqn:parab1-s<t/2-1},~\eqref{eqn:parab1-s=L},~\eqref{eqn:parab1-<t/logt}, and~\eqref{eqn:parab1-t/logt<} imply that for all $\mathcal{C}_d$ sufficiently large (depending only on $d$),
\begin{align}
    \P \bigg ( \bigcup_{s \in [L, t/2]} \bigcup_{v \in \cN_s} \Xi_s(v) \bigg )
    \ulesssim L^{-\frac{2\sqrt{2}-1}{2}} \,.
    \label{eqn:parab1-s<t/2-final}
\end{align}

We next turn our attention to $s \in [t/2,  t]$.
For a real-valued process $X_{\cdot}$ and an interval $[a,b] \subset \R_{\geq 0}$,  define the event
\[
\Theta_{[a,b]}^*(X_\cdot) := \UB{[a,b]}{B}(X_\cdot) \cap \Big \{ \sup_{s' \in [b,b+1]} X_{s'} > B(b) \Big \} \, .
\]
Instead of a union bound as in~\eqref{eqn:parab1-s<t/2-1}, we use the result of~\eqref{eqn:parab1-s<t/2-final} to obtain the following bound:
\begin{align}
    \P \bigg ( 
        \bigcup_{s \in [t/2, t]} \bigcup_{v \in \cN_s} \Xi_s(v)
    \bigg )
    &\ulesssim
    L^{-\frac{1}{2}} 
    +
    \P \bigg (
    \bigcup_{s\in [t/2, t]}
    \bigcup_{v \in \cN_s}
    \UB{[L, t/2]}{B}(R^{(v)}_\cdot)
    \cap \Xi_s(v)
    \bigg )
    \nonumber \\
    &\leq
    L^{-\frac{1}2} + 
    \sum_{s- \frac{t}{2} = 0}^{ \lfloor t/2\rfloor}
    \P \bigg ( 
        \bigcup_{v \in \cN_{s+1}} 
        \Theta_{[L,s]}^*(R^{(v)}_\cdot)
    \bigg ) \, .
    \label{eqn:parab1-s>t/2-1}
\end{align}
It suffices to consider the second term in the second line of~\eqref{eqn:parab1-s>t/2-1}; in particular, we have reduced to a union over particles $v \in \cN_s$ that stay bounded by $B(u)$ for all $u \in [L, s]$ before exceeding $B(s)$ at some time in $[s,s+1]$. This barrier event will give extra decay.\footnote{We did not need this extra decay in the range $s \in [L, t/2]$. Here, it is needed essentially because  $\frac{\d}{\d s}( B(s) - \frac{m_t}{t}(s) )$ is negative when $s > t/2$.} 
Fix any $s \in \frac{t}{2} + \llbracket 0, \floor{\frac{t}{2} }\rrbracket$. 
A union bound and the many-to-one lemma (Lemma~\ref{lem:many-to-one}) give
\begin{align}
    \P \bigg ( 
        \bigcup_{v \in \cN_{s+1}} 
        \Theta_{[L,s]}^*(R^{(v)}_\cdot)
    \bigg )
    &\leq
    e^{s+1} \P ( \Theta_{[L,s]}^*(R_\cdot) ) \, .
    \label{eqn:parab1-s>t/2-m21}
\end{align}
Now, similar to~\eqref{eqn:parab1-s<t/2-gir}, we integrate over $w_1 := B(L) - W_L$ and apply the Markov property at time $L$ to find
\begin{align}
    \P ( \Theta_{[L,s]}^*(R_\cdot) ) 
    &= 
    \int_0^{B(L)} p_L^R \big ( B(L) - w_1) 
    \P( \Theta_{[L,s]}^*(R_\cdot) \given R_L = B(L) - w_1 \big )\d w_1 \, .
    \label{eqn:parab1-s>t/2-int}
\end{align}
Next, we apply the Girsanov transform, this time bounding the $\exp(\cdot) \one_{\{\cdot\}}$ term in the right-hand side of~\eqref{eq.girsanov} by $\one_{\{W_{s-L} > 0 \} }$:
\begin{align}
    \P( \Theta_{[L,s]}^*(R_\cdot) \given R_L = B(L) - w_1 \big )
    &= 
    \pth{ B(L) - w_1 }^{-\alpha_d} \Upsilon
    \label{eqn:parab1-s>T/2-gir} \, ,
\end{align}
where
\begin{align*}
    \Upsilon := 
    \E \big [
    W_{s+1}^{\alpha_d} \one_{ \{ W_{s} > 0 \} \cap \Theta^*_{[L,s]}(W_\cdot) \} }
    \given W_L = B(L) - w_1
    \big ] \, .
\end{align*}
Let $M_{s+1-L}^1 := \sup_{s' \in [s-L, s+1-L]} W_{s'}$. Then using the Markov property to shift time by $L$ and conditioning on $W_{s-L}$ yields
\begin{align}
    \Upsilon
    &= \E_{B(L) - w_1} \Big[ 
        \E \Big [ W_{s+1-L}^{\alpha_d} \one_{ \{ M_{s+1-L}^1 > B(s) \} } \given W_{s-L} \Big ]\nonumber\\
       &\qquad \qquad\qquad\qquad \qquad \times \P_{B(L)-w_1, s-L}^{W_{s-L}} \Big ( 
        \UB{[0,s-L]}{B(\cdot+L)}(W_\cdot)
        \Big ) 
        \one_{ \{W_{s-L} >0 \} }
    \Big ] 
    \nonumber \\
    &\leq 
    \E_{B(L) - w_1} \Big[ 
        \E \Big [ W_{s+1-L}^{\alpha_d} \one_{ \{ M_{s+1-L}^1 > B(s) \} } \given W_{s-L} \Big ]\nonumber\\
        &\qquad \qquad\qquad\qquad \qquad
      \times \P_{B(L)-w_1, s-L}^{W_{s-L}} \Big ( 
        \UB{[0,s-L]}{B(\cdot+L)+1}(W_\cdot)
        \Big ) 
        \one_{ \{W_{s-L} \in (0, B(s)] \} }
    \Big ] 
    \nonumber \\
    &\uasymp 
    \frac{1+w_1}{s-L}
    \E_{B(L) - w_1} \Big [
        W_{s+1-L}^{\alpha_d} (1+ B(s) - W_{s-L}) 
        \one_{ \{ M_{s+1-L}^1 > B(s), W_{s-L} \in (0, B(s)] \} }
    \Big ]
    \nonumber \\
    &\ulesssim 
    \frac{1+w_1}{s-L}
    \E_{B(L) - w_1} \Big [
        (M_{s+1-L}^1)^{\alpha_d}
        (1+ B(s) - W_{s-L}) 
        \one_{ \{ M_{s+1-L}^1 > B(s), W_{s-L} \in (0, B(s)] \} }
    \Big ] \, ,
    \label{eqn:parab1-s>t/2-exp}
\end{align}
where in the second line we have used~\eqref{eqn:brz-modified}, which holds uniformly over $w_1 \geq 0$, and thus the above bound does as well. 
Write $M_1 := \sup_{s' \in [0,1]} W_{s'}$.
Integrating over $w_2:= B(s) - W_{s-L}$ and using the Markov property and shift-invariance gives
\begin{align}
    \Upsilon 
    &\ulesssim 
    \frac{1+w_1}{(s-L)^{3/2}}
    \int_0^{B(s)} (1+w_2) e^{- \frac{(B(s) - B(L) + w_1 - w_2)^2 }{2(s-L)}}
    \E_{B(s) - w_2} \big [ M_1^{\alpha_d} \one_{\{ M_1 > B(s) \} } \big ]
    \d w_2
    \nonumber \\
    &\uasymp 
    \frac{1+w_1}{t^{3/2}}
    \int_0^{B(s)} (1+w_2) B(s)^{\alpha_d} e^{- \frac{(B(s) - B(L) + w_1 - w_2)^2 }{2(s-L)}}
    e^{-\tfrac{w_2^2}{2}}
    \d w_2
    \, .
    \label{eqn:parab1-s>t/2-int2}
\end{align}
Since $s\geq t/2$, we have $B(s) \ulesssim t$ and so 
\begin{align}
    e^{ - \frac{ (B(s)-B(L)  - w_2 + w_1)^2 }{2(s+1-L)} }
    \ulesssim
    e^{-s+L} 
    L^{\mathcal{C}_d \sqrt{2}} 
    \Big((t-s)^{-\mathcal{C}_d \sqrt{2}} \wedge 1\Big)
    \bar{\Psi}_s 
    e^{w_2\sqrt{2}} e^{-w_1 \sqrt{2}} \, ,
\end{align}
uniformly over the relevant ranges of $w_1$ and $w_2$, where 
$
\bar{\Psi}_s := e^{- \sqrt{2} \fc_d   (\log t)\frac{s-L}{t} } \, .
$
Substituting these bounds into~\eqref{eqn:parab1-s>t/2-int2}
gives
\begin{align}
    \Upsilon \ulesssim 
    (1+w_1)
    e^{-w_1 \sqrt{2}}
    L^{\mathcal{C}_d \sqrt{2}} 
    e^{-s+L} 
    t^{- \frac{3}{2}+ \alpha_d}
    \Big((t-s)^{-\mathcal{C}_d \sqrt{2}} \wedge 1\Big)
    \bar{\Psi}_s  \, ,
    \label{eqn:parab1-s>t/2-upsilon}
\end{align}
which holds uniformly over  $w_1 \in [0,B(L)]$. 
Equations~\eqref{eqn:parab1-s>t/2-m21},~\eqref{eqn:parab1-s>T/2-gir},~\eqref{eqn:parab1-s>t/2-upsilon}, and Lemma~\ref{lem:bessel_time_L} leave us with 
\begin{align}
    \P \bigg ( 
        \bigcup_{v \in \cN_{s+1}} 
        \Theta_{[L,s]}^*(R^{(v)}_\cdot)
    \bigg )
    \ulesssim 
    L^{-\frac{2\sqrt{2}-1}{2}}
    t^{- \frac{3}{2}+ \alpha_d}
    \Big((t-s)^{-\mathcal{C}_d \sqrt{2}} \wedge 1\Big) 
    \bar{\Psi}_s  \, .
    \label{eqn:parab1-s>t/2-unionv}
\end{align}
 Now, for $s \in [t/2, t- t^{1/8}]$, we may take $\bar{\Psi}_s \ulesssim t$ (because, as noted before, $\fc_d \geq -1/\sqrt{2}$) and $t-s \geq t^{1/8}$. Then the bound in~\eqref{eqn:parab1-s>t/2-unionv} becomes 
 \begin{align}
     \P \bigg ( 
        \bigcup_{v \in \cN_{s+1}} 
        \Theta_{[L,s]}^*(R^{(v)}_\cdot)
    \bigg )
    \ulesssim 
    L^{-\frac{2\sqrt{2}-1}{2}}
    t^{\alpha_d- \frac{1}{2} - \frac{\mathcal{C}_d \sqrt{2}}{8}} \, ,
    \nonumber
 \end{align}
 from which we find 
 \begin{align}
     \sum_{s- \frac{t}{2}= 0}^{ \floor{ \frac{t}{2}- t^{1/8} }} \P \bigg ( 
        \bigcup_{v \in \cN_{s+1}} 
        \Theta_{[L,s]}^*(R^{(v)}_\cdot)
    \bigg )
    \ulesssim L^{-\frac{2\sqrt{2}-1}{2}} t^{\alpha_d+ \frac{1}{2}- \frac{\mathcal{C}_d \sqrt{2}}{8}}
    \ulesssim L^{-\frac{2\sqrt{2}-1}{2}} \, ,
    \label{eqn:parab1-s>t/2-middle}
 \end{align}
for any $\cC_d \geq 4d/\sqrt{2}$.
For $s \in \big [\lfloor t-t^{1/8}\rfloor, t \big ]$, we have $\bar{\Psi}_s \usim t^{- \sqrt{2}{\fc_d}} =  t^{\frac{3}{2}-\alpha_d}$, so that the bound in~\eqref{eqn:parab1-s>t/2-unionv} becomes 
 \begin{align}
     \P \bigg ( 
        \bigcup_{v \in \cN_{s+1}} 
        \Theta_{[L,s]}^*(R^{(v)}_\cdot)
    \bigg )
    \ulesssim 
    L^{-\frac{2\sqrt{2}-1}{2}}
    (t-s)^{- \mathcal{C}_d \sqrt{2}} \, .
    \nonumber
 \end{align}
Thus,
\begin{align}
    \sum_{s -  \frac{t}{2}= \floor{ \frac{t}{2} - t^{1/8}} + 1}^{ \floor{ t/2}}
    \P \bigg ( 
        \bigcup_{v \in \cN_{s+1}} 
        \Theta_{[L,s]}^*(R^{(v)}_\cdot)
    \bigg ) \ulesssim 
    L^{-\frac{2\sqrt{2}-1}{2}}
    \, ,
    \label{eqn:parab1-s>t/2-last}
\end{align}
for any $\mathcal{C}_d >0$.
The lemma then follows from
\eqref{eqn:parab1-s>t/2-1},~\eqref{eqn:parab1-s>t/2-middle}, and~\eqref{eqn:parab1-s>t/2-last}. 
\end{proof}

\subsection{Proof of Theorem~\ref{thm:outside_window}}
\label{subsec:rough_window}

Note that for all $t\geq 1$ and for all $y \leq t^{1/2}$, Theorem $1.1$ of~\cite{Mallein15} shows $\P(R_t^* > m_t + y) > 0$.
From this and Lemma~\ref{lem:rough_parabola}, we see that it suffices to show 
\begin{align}
    \P \pth{ \exists v \in \cN_t :  R_L^{(v)} \not \in I_L^{\win},
    \UB{[L, t]}{B} (R^{(v)}_\cdot),
    R^{(v)}_t \in (m_t+y, B(t))   
    } 
    = o_u(1)\, .
    \label{eq.window.goal}
\end{align} 
We show~\eqref{eq.window.goal} by upper bounding the left-hand side above by
\begin{align*}
    \E \bigg[
    \sum_{v \in \cN_t} \one_{
        \{ R_L^{(v)} \not \in I_L^{\win} \} \cap 
        \{ \UB{[L,t]}{B}(R^{(v)}_\cdot) \}
        \cap 
        \{ R^{(v)}_t \in (m_t+y, B(t)) \}
    }
    \bigg] \,,
\end{align*}
whence the many-to-one lemma (Lemma~\ref{lem:many-to-one}) gives the above as equal to    
\begin{align*}
    e^t \P \pth{
    R_L \not \in I_L^{\win};
    ~\UB{[L, t]}{B(\cdot+L)} (R_\cdot); 
    ~R_t \in (m_t+y, B(t))} \,,
\end{align*}
where $R.$ denotes a $d$-dimensional Bessel process.
Integrating over $w_L:= B(L) -R_L$ and applying the Markov property at time $L$ gives the above as equal to
\begin{align}
    e^t \int_0^{B(L)} p_L^R(0, B(L)-w_L) 
    \cP(w_L)\,,
    \label{eqn:window-markov}
\end{align}
where 
\begin{align}
\cP(w_L) &:= \P_{B(L)-w_L}\Big( \UB{[0,t-L]}{B(\cdot+L)}(R_.), R_{t-L} \in (m_t+y, B(t)) \Big) \d w_L \nonumber \\
    &=\E_{B(L)-w_L} \bigg[ 
    \E_{B(L)-w_L} \Big[ \one_{\UB{[0, t-L]}{B(\cdot+L)} (R_\cdot)} \given R_{t-L} \Big] \one_{\{R_{t-L} \in (m_t+y, B(t))\}}
    \bigg] \,.
    \label{eqn:window-condexp}
\end{align}
We now apply the Girsanov transform to the conditional expectation to replace the Bessel process with a Brownian motion $(W_s)_{s \in[0,t-L]}$ conditioned to end at $R_{t-L}$. Bounding the indicator function and the exponential term in~\eqref{eq.girsanov} by $1$ gives the following upper bound on the expression for $\cP(w_L)$ in~\eqref{eqn:window-condexp}:
\begin{align}
    \E_{B(L)-w_L} \bigg[ \Big(\frac{W_{t-L}}{B(L)-w_L}\Big)^{\alpha_d}
    \P_{B(L)-w_L, t-L}^{W_{t-L}} \Big(\UB{[0, t-L]}{B(\cdot+L)} (W_\cdot) \Big) \one_{\{W_{t-L} \in (m_t+y, B(t))\}}
    \bigg] \,.
    \label{eqn:window-gir}
\end{align}
Increase the barrier event above to $\UB{[0, t-L]}{B(\cdot+L)+1} (W_\cdot)$, then apply~\eqref{eqn:brz-modified}, so uniformly on $\mathds{1}_{\{W_{t-L} \in (m_t+y, B(t))\}}$, we have
\begin{align*}
    \P_{B(L)-w_L, t-L}^{W_{t-L}} \Big(\UB{[0, t-L]}{B(\cdot+L)} (W_\cdot) \Big) \ulesssim \frac{(w_L+1)(B(t)-W_{t-L}+1)}{t-L}\,.
\end{align*}
Substituting the above into~\eqref{eqn:window-gir}, noting that $W_{t-L} \uasymp t$ uniformly on $\mathds{1}_{\{W_{t-L} \in (m_t+y, B(t))\}}$,  and then
integrating over $w_t:= B(t) - W_{t-L}$ yields
\begin{align}
    \cP(w_L) &\ulesssim
    \frac{(w_L+1)t^{\alpha_d -\frac{3}{2}}}{(B(L)-w_L)^{\alpha_d}}
    \int_0^{\log L -y} (w_t+1) e^{- \frac{(B(t)-B(L) + w_L-w_t)^2}{2(t-L)}} \d w_t \nonumber \\
    &\uasymp L^{\cC_d \sqrt{2}} e^{-(t-L)} \frac{(w_L+1)e^{- \sqrt{2}w_L}} 
    {(B(L)-w_L)^{\alpha_d}}
    \int_0^{\log L -y} (w_t+1) e^{\sqrt{2}w_t} \d w_t \nonumber \\
    &\uasymp 
    (\log L) L^{(\cC_d+1) \sqrt{2}} e^{-(t-L)} \frac{(w_L+1)e^{- \sqrt{2}w_L}} 
    {(B(L)-w_L)^{\alpha_d}} \,,
    \label{eqn:window-cPbound}
\end{align}
where in moving from the first line to the second line of~\eqref{eqn:window-cPbound} we have used the fact that $w_L, w_t \ulesssim L$ to greatly simplify the exponential in the integrand of the first line. 
Substituting the result of~\eqref{eqn:window-cPbound} into~\eqref{eqn:window-markov} and computing the resulting integral via equation~\eqref{eq.bessel.c} of Lemma~\ref{lem:bessel_time_L} simply yields $(\log L) L^{-1/6}$, thereby showing~\eqref{eq.window.goal}. \qed 
\subsection{Another upper barrier} 
\label{subsec:B0}
We present in this short section a lemma, whose proof
is very similar in strategy and method to the proof of Lemma~\ref{lem:rough_parabola} above. It will be used in the proof of Proposition~\ref{prop:right_tail_equiv}. Let $\mathcal{K}_d$ be a large positive constant depending only on $d$, and define
\begin{align}
    B_0(s) := B_0(s; t,L) = \frac{m_t}{t}(s+L) + y +  \log \ell + \mathcal{K}_d \log \pth{ s\wedge (\tl-\ell-s)}_+ \,,
    \label{def:B0}
\end{align}
where $\tilde{t}$ and $\ell$ will be chosen as functions of $t$ and $L$ (see~\eqref{eqn:def-ttilde} and~\eqref{eqn:def-ell} respectively).
The following lemma shows that we can find a $\mathcal{K}_d$ such that the probability that a branching Bessel process started somewhere in $I_L^{\win}$ exceeds $B_0(s)$ at time $s$ for some $s \in [0, \tl- \ell]$ is negligible. 
\begin{lemma}
\label{lem:parabola_B0}
There exists a constant $\mathcal{K}_d$ depending only on $d$ such that 
\begin{align}
        \P_{\sqrt{2} L -z} \bigg( 
        \bigcup_{s \in [0, \tl - \ell]} 
        \bigcup_{v \in \cN_{s}} 
        \{ R_s^{(v)} > B_0(s) \} \bigg)
    \ulesssim 
    \ell^{- \sqrt{2}} \fM_{L,z}\,. 
    \label{eq.lem:parabola_b0}
\end{align}
\end{lemma}
\begin{proof}
For a real-valued process $X_\cdot$ and an interval $[a,b] \subset \R_{\geq 0}$,  define the event
\[
\bT_{[a,b]}(X_\cdot) := \UB{[a,b]}{B_0}(X_\cdot) \cap \Big \{ \sup_{s' \in [b,b+1]} X_{s'} > B_0(b) \Big \} \, .
\]
Then, similar to~\eqref{eqn:parab1-s>t/2-1}, we find 
\[
    \P_{\sqrt{2} L -z} \bigg( 
        \bigcup_{s \in [0, \tl - \ell]} 
        \bigcup_{v \in \cN_{s}} 
        \{ R_s^{(v)} > B_0(s) \} \bigg)
    \leq 
    \P_{\sqrt{2}L -z} \bigg(  \bigcup_{s \in \llbracket 0, \tl-\ell\rrbracket} 
    \bigcup_{v \in \cN_{s+1}}
    \bT_{[0,s]}\big(R^{(v)} \big)
    \bigg) \, .
\]
A union bound over $s$ shows that the last expression is at most
\[ \sum_{s=0}^{\lfloor\tilde t - \ell\rfloor}  \E_{\sqrt{2}L-z} \bigg[ \sum_{v\in \cN_{s+1}} 
\one_{\bT_{[0,s]}( R^{(v)} ) }
    \bigg ]\,,
    \]
whence we may conclude from the many-to-one lemma and a Girsanov transform (where we upper bound the term $\exp(\cdot)\one_{\{\cdot \}}$ in the right-hand of~\eqref{eq.girsanov} by $\one_{\{W_s>0\}}$, valid for $d\geq 3$)
that the left-hand side of~\eqref{eq.lem:parabola_b0} is at most
\begin{align}
    &
    \sum_{s=0}^{\lfloor\tilde t - \ell\rfloor} e^{s+1}
    \big (\sqrt{2}L -z \big )^{-\alpha_d}
    \E_{\sqrt{2}L-z}
    \brak{
        W_{s+1}^{\alpha_d} 
    \one_{
    \{ W_{s} >0 \} \cap 
    \bT_{[0,s]}(W_\cdot) }
    } \,.   
    \nonumber 
\end{align}
Let $M_{s+1}^1 := \sup_{s' \in [s, s+1]} W_{s'}$. Then from the above, we find
\begin{align}
    \P_{\sqrt{2} L -z} \bigg( 
        \bigcup_{s \in [0, \tl - \ell]} 
        \bigcup_{v \in \cN_{s}} 
        \{ R_s^{(v)} > B_0(s) \} \bigg)
    &\leq 
    \sum_{s=0}^{\lfloor\tilde t - \ell\rfloor} e^{s+1}
    \big (\sqrt{2}L -z \big )^{-\alpha_d}
    \Upsilon_s\, ,
    \label{eqn:parab2-gir}
\end{align}
where
\begin{align}
    \Upsilon_s :=
    \E_{\sqrt{2}L-z}
    \Big [
        (M^1_{s+1})^{\alpha_d} 
    \one_{
    \{ W_{s} >0 \} \cap 
    \{ 
    M_{s+1}^1 > B_0(s)
    \} \cap 
    \UB{[0,s]}{B_0}(W_\cdot)
    }
    \Big ] \,.    
\end{align}
Let's first consider $s > 0$. Integrating over $w:= B_0(s) -W_s \in [0, B_0(s)]$ gives
\begin{align}
    \Upsilon_s 
    &=
    \int_0^{B_0(s)}  p_s^W \big ( \sqrt{2}L-z, B_0(s) -w \big )
    \P_{\sqrt{2}L -z, s}^{B_0(s)-w} \Big ( 
    \UB{[0,s]}{B_0(\cdot)+1}(W_\cdot) \Big )
    \nonumber \\
    &\times
    \E \big [
    (M_{s+1}^1)^{\alpha_d} \one_{M_{s+1}^1 > B_0(s)} 
    \given W_s = B_0(s)-w
    \big ] \d w
    \nonumber \\
    &\ulesssim 
    \frac{z B_0(s)^{\alpha_d}}{s^{3/2}} 
    \int_0^{B_0(s)} 
    (1+w)
    e^{- (B_0(s)-w - \sqrt{2}L+ z)^2/(2s)} 
     e^{-w^2/2}
    \d w \, ,
    \label{eqn:parab2-int} 
\end{align}
where the above bound comes from~\eqref{eqn:brz-modified} and the bound
\begin{align*}
\E \big [
    (M_{s+1}^1)^{\alpha_d} \one_{M_{s+1}^1 > B_0(s)} 
    \given W_s = B_0(s)-w
    \big ] 
    &= \int_0^{\infty} (B_0(s) +x)^{\alpha_d} e^{-(x+w)^2/2} \d x \nonumber \\
    &\ulesssim B_0(s)^{\alpha_d} e^{-w^2/2} \, ,
\end{align*}
which holds uniformly over $w \in [0, B_0(s)]$. Let $\mathcal{E}(w) := B_0(s)-w - \sqrt{2}L+ z - \sqrt{2}s$, and consider $s \in [1, \lfloor (\tl-\ell)/2 \rfloor ]$. Then expanding $(\sqrt{2}s + \mathcal{E}(w))^2$ in the last line of~\eqref{eqn:parab2-int} gives
\begin{align}
    \frac{e^{s+1} (\sqrt{2}L - z)^{-\alpha_d} \Upsilon_s}{\ell^{-\sqrt{2}} \fM_{L,z}}
    \ulesssim 
    s^{ - \frac{3}{2} - \mathcal{K}_d\sqrt{2}}
    B_0(s)^{\alpha_d} 
    \widehat{\Psi}_s 
    \int_0^{B_0(s)} (1+w)e^{w \sqrt{2} - \frac{w^2}{2}}
    e^{- \frac{\mathcal{E}(w)^2}{2s} }
    \d w \, ,
    \label{eqn:parab2-expanded}
\end{align}
where 
$
\widehat{\Psi}_s := t^{ ( \frac{3}{2}-\alpha_d)s/t} 
$.

For $s \in [1, z^{1/2}]$, we have $\mathcal{E}(w) \uasymp z$, $B_0(s) \ulesssim L \ulesssim z^6$, and $\widehat{\Psi}_s  \usim 1$. Then~\eqref{eqn:parab2-expanded} gives 
\begin{align}
    \frac{e^{s+1} (\sqrt{2}L - z)^{-\alpha_d} \Upsilon_s}{\ell^{-\sqrt{2}} \fM_{L,z}}
    \ulesssim 
    z^{6\alpha_d} e^{-z^{3/2}/2} \, ,
    \nonumber 
\end{align}
so that 
\begin{align}
    \frac{\sum_{s=1}^{\floor{z^{1/2}}}
    e^{s+1} (\sqrt{2}L - z)^{-\alpha_d} \Upsilon_s}{\ell^{-\sqrt{2}} \fM_{L,z}}
    \ulesssim z^{6 \alpha_d + z^{1/2}} e^{-z^{3/2}/2}
    \ulesssim 1 \, .
    \label{eqn:parab2-z1/2}
\end{align}

For $s \geq z^{1/2}$, we will always take the trivial upper bound $\exp (- \mathcal{E}(w)^2/(2s) ) \leq 1$. 
For $s \in [z^{1/2}, L]$, we again have $B_0(s) \ulesssim L \ulesssim z^6$ and $\widehat{\Psi}_s  \usim 1$, so that 
\begin{align}
    \frac{e^{s+1} (\sqrt{2}L - z)^{-\alpha_d} \Upsilon_s}{\ell^{-\sqrt{2}} \fM_{L,z}}
    \ulesssim 
    \big (z^{1/2} \big )^{ - \frac{3}{2} - \mathcal{K}_d \sqrt{2}} z^{6\alpha_d} \, .
    \nonumber 
\end{align}
Since $L \ulesssim z^6$, it follows from the previous display that for all $\mathcal{K}_d \geq \frac{24 \alpha_d +21}{2 \sqrt{2}}$, we have
\begin{align}
    \frac{\sum_{s=\floor{z^{1/2}} +1}^{\floor{L}}
    e^{s+1} (\sqrt{2}L - z)^{-\alpha_d} \Upsilon_s}{\ell^{-\sqrt{2}} \fM_{L,z}}
    \ulesssim 
    z^{6+ 6\alpha_d - \frac{3}{4} - \frac{\mathcal{K}_d \sqrt{2}}{2}}
    \ulesssim 1 \, .
    \label{eqn:parab2-L}
\end{align}

For $s \in [L, t^{1/3}]$, we have $B_0(s) \uasymp s$ and $\widehat{\Psi}_s \usim 1$. Then
\begin{align*}
    \frac{e^{s+1} (\sqrt{2}L - z)^{-\alpha_d} \Upsilon_s}{\ell^{-\sqrt{2}} \fM_{L,z}}
    &\ulesssim 
    s^{- \frac{3}{2}- \mathcal{K}_d\sqrt{2} + \alpha_d } \, ,
\end{align*}
so that for all $\mathcal{K}_d \geq \frac{2\alpha_d-1}{2 \sqrt{2}}$, we have
\begin{align}
    \frac{\sum_{s=\floor{L} +1}^{\floor{t^{1/3}}}
    e^{s+1} (\sqrt{2}L - z)^{-\alpha_d} \Upsilon_s}{\ell^{-\sqrt{2}} \fM_{L,z}}
    \uasymp
    L^{- \frac{1}{2}- \mathcal{K}_d\sqrt{2} + \alpha_d }
    \ulesssim 1 \, .
    \label{eqn:parab2-t1/3}
\end{align}

For $s \in [t^{1/3}, (\tl- \ell)/2]$, we have $\widehat{\Psi}_s \ulesssim t^{3/4}$; thus, for all $\mathcal{K}_d \geq \frac{4\alpha_d + 15}{4 \sqrt{2}}$,  we have
\begin{align}
    \frac{\sum_{s=\floor{t^{1/3}} +1}^{\floor{(\tl - \ell)/2}}
    e^{s+1} (\sqrt{2}L - z)^{-\alpha_d} \Upsilon_s}{\ell^{-\sqrt{2}} \fM_{L,z}}
    \ulesssim 
    t \big (t^{1/3} \big)^{- \frac{3}{2}- \mathcal{K}_d\sqrt{2} + \alpha_d }  t^{3/4} 
    \ulesssim 1
    \, .
    \label{eqn:parab2-t/2}
\end{align}

For $s \geq (\tl-\ell)/2$, 
instead of~\eqref{eqn:parab2-expanded}, we have 
\begin{align}
    \frac{e^{s+1} (\sqrt{2}L - z)^{-\alpha_d} \Upsilon_s}{\ell^{-\sqrt{2}} \fM_{L,z}}
    \ulesssim 
    t^{\alpha_d}
    s^{- \frac{3}{2}}
    \big( (\tl-\ell-s) \vee 1 \big)^{-\mathcal{K}_d\sqrt{2}}
    \widehat{\Psi}_s  \, ,
    \label{eqn:parab2-expanded2}
\end{align}
where we have used $B_0^{\alpha_d}(s) \usim 1$ and $e^{-\frac{\mathcal{E}(w)^2}{2s}} \leq 1$. 

For $s \in [ (\tl-\ell)/2, \tl-t^{1/8} ]$, we have $s^{-3/2} \widehat{\Psi}_s \ulesssim 1$. Thus, for $\mathcal{K}_d \geq \frac{8(1+ \alpha_d)}{\sqrt{2}}$, \eqref{eqn:parab2-expanded2} gives us
\begin{align}
    \frac{\sum_{s=\floor{(\tl-\ell)/2} +1}^{\floor{\tl-t^{1/8}}}
    e^{s+1} (\sqrt{2}L - z)^{-\alpha_d} \Upsilon_s}{\ell^{-\sqrt{2}} \fM_{L,z}}
    \ulesssim 
    t^{1 + \alpha_d - \frac{\mathcal{K}_d\sqrt{2}}{8}  }  
    \ulesssim 1 \, .
    \label{eqn:parab2-1/8}
\end{align}

For $s \in [\tl-t^{1/8}, \tl -\ell]$, we have $s \sim t$, so that~\eqref{eqn:parab2-expanded2} gives us 
\begin{align}
    \frac{\sum_{s=\floor{\tl - t^{1/8}} +1}^{\floor{\tl-\ell}}
    e^{s+1} (\sqrt{2}L - z)^{-\alpha_d} \Upsilon_s}{\ell^{-\sqrt{2}} \fM_{L,z}}
    \ulesssim 
    \big( (\tl-\ell-s) \vee 1 \big)^{-\mathcal{K}_d\sqrt{2}} 
    \ulesssim 1 \, ,
    \label{eqn:parab2-ell}
\end{align}
for any $\mathcal{K}_d > 0$.

Lastly, we consider the summand in~\eqref{eqn:parab2-gir} corresponding to $s = 0$:
\begin{align}
    &e
    \big (\sqrt{2}L -z \big )^{-\alpha_d}
    \E_{\sqrt{2}L-z}
    \brak{
        (M^1_{1})^{\alpha_d} 
    \one_{
    \{ 
    M_{1}^1 > B_0(0)
    \} 
    }}
    \nonumber \\
    &\uasymp 
    \big (\sqrt{2}L -z \big )^{-\alpha_d} \int_{B_0(s)}^{\infty} x^{\alpha_d} e^{- (z+y+\log \ell + x)^2/2 } \d x
    \nonumber \\
    &\ulesssim 
    \big (\sqrt{2}L -z \big )^{-\alpha_d} e^{- (z+y)^2/2} 
    e^{ -(\log \ell)^2 /2} 
    \big (B_0(0) \big )^{\alpha_d} 
    e^{- B_0(0)^2/2} 
\ulesssim \ell^{- \sqrt{2} } \fM_{L,z} \, .
    \label{eqn:parab2-0}
\end{align}
Lemma~\ref{lem:parabola_B0} then follows from~\eqref{eqn:parab2-gir} and the bounds in
\eqref{eqn:parab2-z1/2}--
\eqref{eqn:parab2-t/2},\eqref{eqn:parab2-1/8}, and \eqref{eqn:parab2-ell}--\eqref{eqn:parab2-0}.
\end{proof}

\section{Proof of Theorem~\ref{thm:right_tail_asymp}: a modified second moment method for \texorpdfstring{$d\geq 2$}{d at least 2}}
\label{sec:mod2ndmom}

We are in the situation of a branching Bessel process started at height $\sqrt{2} L -z$ and run for time $t- L$.  For brevity, we will write
\begin{align}
\tl := t- L \,.
\label{eqn:def-ttilde}
\end{align}
In this section, we detail the proof of Theorem~\ref{thm:right_tail_asymp} for all fixed $d\geq 2$, which gives the exact asymptotics of the right-tail probability $\P_{\sqrt{2}L-z}(R_{\tl}^* > m_t+y)$. 
The model for our proof comes from the modified second moment method of~\cite{BDZ16}, which took place in the setting of a one-dimensional non-lattice random walk. 
We detailed this method in Section~\ref{subsec-BDZ-method} in the case of one-dimensional branching Brownian motion.
We will refer to the content of that section to highlight the similarities and differences between the one-dimensional method and ours.

Adopting the modified second moment method of~\cite{BDZ16} to the case of $d$-dimensional Brownian motion (for $d \geq 2$) requires several additional technical results to handle the differences between one-dimensional Brownian motion and the $d$-dimensional Bessel process--- for us, these differences are essentially captured by the Girsanov transform~\eqref{eq.girsanov}.
Our analogue of $\cA_{v,t}(x)$ is given by $G_{L,t}(v)$, defined below in~\eqref{def:G-event}; however, before going into further details, we must first establish some notation.

\subsection{Notation}
Fix a parameter $\ell := \ell(L)$ such that 
\begin{equation}
    1\leq \ell(L) \leq L^{1/6}  
    \qquad\mbox{and}\qquad 
    \lim_{L\to\infty}\ell(L)=\infty \,.
    \label{eqn:def-ell}
\end{equation}
Any choice of $\ell(L)$ satisfying~\eqref{eqn:def-ell} is valid\footnote{In fact, the condition $\ell \leq L^{1/6}$ is not imperative; however, enforcing it abbreviates several calculations.}.
Note that $\ell(L) \leq z(L)$ for all $L \geq 1$.

Next, we discuss the relevant barriers and associated notation. Recall the barrier events given in equation~\eqref{eqn:notation-barrier}. 
We now define our ``lower barrier" function:
as mentioned at the start of this subsection, this is needed to handle the $\exp(\cdot)$ term in the Girsanov transform~\eqref{eq.girsanov}.
Recall from Section \ref{subsec:hittingP}
that we write $f_{a}^b(s;T)$ to denote the function of $s \in [0,T]$ whose graph is the straight line segment in $\R^2$ connecting $(0,a)$ to $(T,b)$. 
Let\footnote{This choice of $\ell_1$ will be made clear in the proof of Claim~\ref{claim:bridge-barrier-2}. Really, we just require $\ell_1 := \ell^{\epsilon}$ for any $\epsilon >0$ such that $\ell^{1/3} - \ell_1 \sqrt{2} - \ell_1^{2/3} \gg 0$.}
\begin{align}
\ell_1 := \ell^{1/4} \,.
\label{eqn:def-ell1}
\end{align}
It will also be useful at this point to define the quantities
\begin{align}
    \x(a):= \sqrt{2}L - a, \quad \y(b) := \frac{m_t}{t}(t-\ell) +y - b\, ,
    \label{def:xz-yw}
\end{align}
for any $a, b \in \R$.
The function $Q_z(\cdot) := Q_{y, z,L,t}(\cdot)$, defined below, will be the most important ``lower barrier function":
\begin{align}
    Q_z(s) =
    \begin{cases}
        \sqrt{2}L - 2L^{2/3}
        &s \in [ 0, \ell_1 ]
        \\
        f_{\x(2L^{2/3})}^{\y(2\ell^{2/3})}(s; \tl-\ell) - ( s\wedge (\tl-\ell-s) )^{2/3}
        &s \in (\ell_1, \tl - \ell - \ell_1)
        \\
        \y(2 \ell^{2/3})
        &s \in [\tl - \ell - \ell_1, \tl -  \ell]
    \end{cases} \, .
    \label{eq.lowerbanana.Qdefn}
\end{align} 
See Figure~\ref{fig:G-Lt-event} for a depiction of $Q_z$.
Note that for all $L$ and $t$ sufficiently large, we have
$Q_z(\ell_1) > \lim_{s \downarrow \ell_1} Q_z(s)$ 
and 
$Q_z(\tl- \ell  - \ell_1)>
\lim_{s \downarrow \tl-\ell-\ell_1} Q_z(s)$.
We will also make use of the upper-barrier function $B_0$, defined in~\eqref{def:B0}.
With a view towards the $\exp(\cdot)$ term in Girsanov transform~\eqref{eq.girsanov}, note  that for all $\ell > 0$,
\begin{align}
     \exp  \Big( \frac{\alpha_d -\alpha_d^2}{2} \int_0^{\tl - \ell}\frac{1}{Q_z(s)^2} \d s \Big)
     \usim 
     \exp  \Big( \frac{\alpha_d -\alpha_d^2}{2} \int_0^{\tl - \ell}\frac{1}{B_0(s)^2} \d s \Big)
     \usim 1 \,.
    \label{eqn:trivial-exp-gir}
\end{align}
A crucial barrier event used throughout Sections \ref{sec:mod2ndmom}--\ref{sec:2ndmom} will be the event
that a process $X_{\cdot}$ is bounded above by the linear barrier $\frac{m_t}{t}(\cdot+L)+y$ (where we write $f(\cdot+r)$ to denote the function $u \mapsto f(u+r)$) and bounded below by $Q_z(\cdot)$ on a certain time interval $I \subset \R_{\geq 0}$. We will denote this event by $\B_I(X_\cdot) := \B_{I,y,z,L,t}(X_\cdot)$; note that
\begin{align}
    \B_I(X_\cdot) =
    \UB{I}{\frac{m_t}{t}(\cdot+L)+y}(X_\cdot)
    \cap 
    \LB{I}{Q_z}(X_\cdot)
    \, .
    \label{eqn:def-upper/lowerbarrier}
\end{align}
For any $s, r \geq 0$ and $v \in \cN_s$, we also define
\begin{align}
    \mathfrak{T}_r(v) &:= \mathfrak{T}_{r, t, y}(v) =   \left \{ \max_{v' \in \cN_{r}^v} R_{r+s}^{(v')} > m_t+y
    \right \} \, .
    \label{def:T-tailevent}
\end{align}
 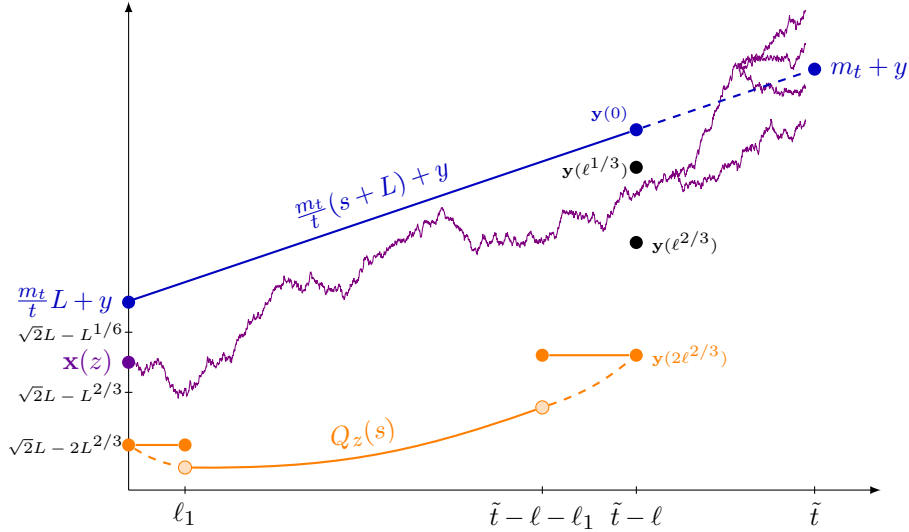
\begin{figure}
     \begin{tikzpicture}[>=latex,font=\small]
 \draw[->] (0, 0) -- (10, 0); 
  \draw[->] (0, 0) -- (0, 6.5);
  \node (fig1) at (4.55,3.8) {
  \includegraphics[width=.65\textwidth]{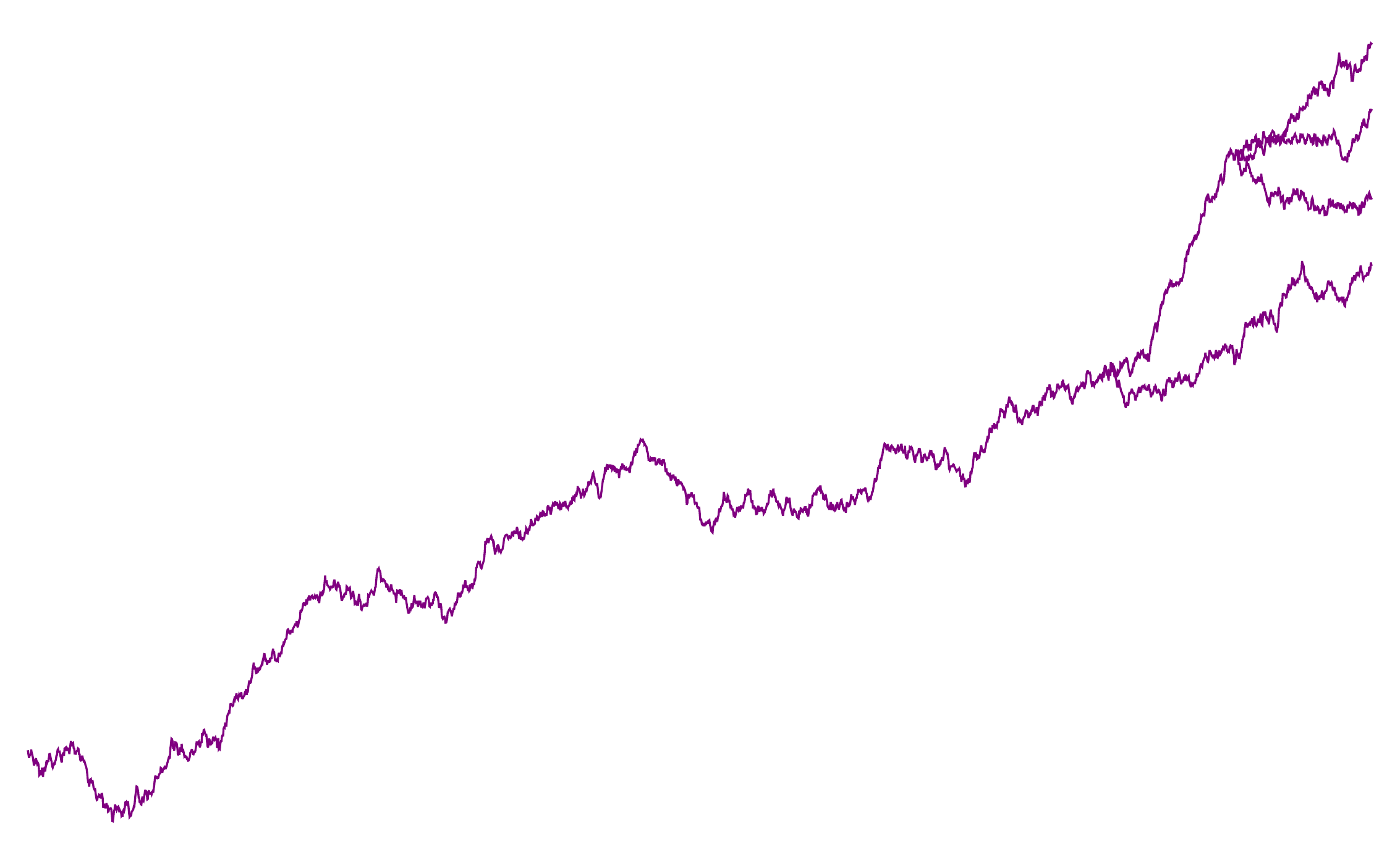}};
   \node[circle,fill=blue!75!black,label={[blue!75!black]left:$\frac{m_t}tL+y$}, inner sep=1.75pt] (a) at (0.0,2.5) {};
   \node[circle,fill=blue!75!black,label={[blue!75!black]right:$m_t+y$},inner sep=1.75pt] (b) at (9.12,5.6) {};
   \draw (-0.05,2.1) -- (0.05,2.1) node [left,font=\tiny] {$\sqrt 2 L - L^{1/6}$};
   \draw (-0.05,1.3) -- (0.05,1.3) node [left,font=\tiny] {$\sqrt 2 L - L^{2/3}$};
   \draw (-0.05,.6) -- (0.05,.6) node [left,font=\tiny] {$\sqrt 2 L - 2L^{2/3}$};
   \node[circle,fill=blue!45!purple,label={[blue!45!purple]left:$\x(z)$},inner sep=1.75pt] (x) at (0, 1.7) {};
   \node[circle, fill=blue!75!black, label={[blue!75!black, font=\tiny,label distance=-3pt]above left:$\y(0)$}, inner sep=1.75pt] (y0) at ( $(a)!0.74!(b)$) {};
   \node[circle, fill=black, label={[black, font=\tiny,label distance=-3pt]left:$\y(\ell^{1/3})$}, inner sep=1.75pt] (y1) at ( $(y0)-(0,0.5)$) {};
   \node[circle, fill=black, label={[black,font=\tiny]right:$\y(\ell^{2/3})$}, inner sep=1.75pt] (y2) at ( $(y0)-(0,1.5)$) {};
   \node[circle, fill=orange, label={[orange,font=\tiny]right:$\y(2\ell^{2/3})$}, inner sep=1.75pt] (y3) at ( $(y0)-(0,3)$) {};
   \node[circle, fill=orange, inner sep=1.75pt] (a0) at ( 0.75, 0.6 ) {};
   \node[circle, fill=orange, inner sep=1.75pt] (a00) at ( 0, 0.6 ) {};
   \node[circle, color=orange,fill=orange!25, draw, inner sep=1.75pt] (a1) at ( 0.75, 0.3 ) {};
   \node[circle, color=orange,fill=orange!25, draw, inner sep=1.75pt] (b1) at ( 5.5, 1.1 ) {};
   \node[circle, fill=orange, inner sep=1.75pt] (b2) at ( $(y3)-(1.25,0)$ ) {};
  \draw[orange, thick] (a00) -- (a0);
  \draw[orange, thick, dashed] (a00) to[bend right=10] (a1);
  \draw[orange, thick] (a1) to[bend right=10] node[midway,above,sloped] {$Q_z(s)$} (b1);
  \draw[orange, thick, dashed] (b1) to[bend right=10] (y3);
  \draw[orange, thick] (b2) -- (y3);
  \draw[blue!75!black, thick] ($(a)+(0.02,0.02)$) -- node[midway,above,sloped] {$\frac{m_t}t(s+L)+y$} (y0);
  \draw[blue!75!black, thick, dashed] (y0) -- (b);
  \draw (.75,0.05) -- (.75,-0.05) node[below] {$\ell_1$};
  \draw (5.5,0.05) -- (5.5,-0.05) node[below] {$\tilde t-\ell-\ell_1$};
  \draw (6.75,0.05) -- (6.75,-0.05) node[below] {$\tilde t - \ell$};
  \draw (9.12,0.05) -- (9.12,-0.05) node[below] {$\tilde t$};
  \end{tikzpicture}
     \caption{The event $G_{L,t}(v)$ from~\eqref{def:G-event}: for a fixed $v\in \cN_{\tll}$, $R^{(v)}_.$ is bounded above by the blue line and below by the solid orange curves that comprise $Q_z$ on $[0,\tll]$, $R^{(v)}_{\tll}$ lies in the ``window" $[\y(\ell^{2/3}), \y(\ell^{1/3})]$, and $v$ produces a descendent in $\cN_{\tl}$ that exceeds $m_t+y$.}
     \label{fig:G-Lt-event}
 \end{figure}
Now, for particles $v \in \cN_{\tl - \ell}$, define the events
\begin{align}
    F_{L,t}(v)
    &:= 
        \UB{[0, \tl - \ell]}{B_0}(R^{(v)}_\cdot)
        \cap 
        \Big \{R^{(v)}_{\tll} > \frac{t}{\sqrt{d}} \Big \}
        \cap 
        \mathfrak{T}_\ell(v)
    \, , \text{ and} \label{def:F-event}\\
    G_{L,t}(v)
    &:= 
    \B_{[0, \tl-\ell]}(R^{(v)}_\cdot)
        \cap 
        \left\{
            \y\big( R^{(v)}_{\tl - \ell} \big)  \in [ \ell^{1/3}, \ell^{2/3}]
        \right\}
        \cap
        \mathfrak{T}_\ell(v) \,.  \label{def:G-event}
\end{align}
Figure~\ref{fig:G-Lt-event} depicts the event $G_{L,t}(v)$ for some fixed $v \in \cN_{\tll}$. 
Lastly, define the simple random variables
\begin{align}
    \Gamma_{L,t}
    &:= 
    \sum_{v \in \cN_{\tl -\ell}}
    \one_{ 
        F_{L,t}(v)
    } \quad \text{and} \quad
    \bar{\Lambda}_{L,t} := 
    \sum_{v \in \cN_{\tl -\ell}}
    \one_{ 
        G_{L,t}(v)
    } \, .
    \label{def:simplefunctions}
\end{align}
Note that 
$\bar{\Lambda}_{L,t}  \leq \Gamma_{L,t}$ holds trivially for all $t, L, \ell > 0$. As we will later show (see Lemma~\ref{lem:barrier-equiv} below), in terms of first moment $\E_{\x(z)}[\cdot]$, these random variables are in fact asymptotically equivalent (in the sense of $\usim$) to each other, and are each asymptotically equivalent (in the sense of $\uasymp$) to $\fM_{L,z}$ (defined in~\eqref{def:M_Lz}).

\subsection{The modified second moment method}
The event $G_{L,t}(v)$ is the suitable analogue of $\cA_{v,t}(x)$ (see~\eqref{eqn:def-A_v,t,ell} and the discussion that follows it), which had the key properties that \textbf{(a) } leading-order asymptotics for $\E[\Xi_{t,\ell}(x)]$ were precisely computable, and \textbf{(b)} the second moment $\E[\Xi_{t,\ell}(x)^2]$ was asymptotically equivalent to the first moment squared in the sense of~\eqref{eqn:bdz-1st=2ndmom}. Towards step~\textbf{(a)}, we have the following. 

\begin{lemma} \label{lem:barrier-equiv}
For any $y\in \R$, we have 
\begin{align}                          
    \E_{\x(z)}\brak{\bar{\Lambda}_{L,t}}
    \usim \E_{\x(z)}\brak{\Gamma_{L,t}}\,,
    \label{eq.1stmom.bananaequiv}
\end{align}
and
\begin{align}
            \frac{
        \E_{\x(z)} \brak{\bar{\Lambda}_{L,t}}
            }
            {\fM_{L,z}}
    &\usim 
            \frac{2^{ \frac{1+\alpha_d}{2} }}
            {\sqrt{\pi}}
            \int_{\ell^{1/3}}^{\ell^{2/3}} w e^{w \sqrt{2}} 
            \P\big( W_{\ell}^* > \sqrt{2}\ell + w \big) \d w 
            \uasymp 1
            \,.
            \label{eq.1stmom.finalintegral}
\end{align}
\end{lemma}

Lemma~\ref{lem:barrier-equiv} is proved in Section~\ref{sec:1stmom} (for
$d \geq 3$) as follows.
First, in light of the many-to-one lemma (Lemma~\ref{lem:many-to-one}), we need precise estimates for $\P_{\x(z)}(G_{L,t}(v))$ in order to show \eqref{eq.1stmom.finalintegral}. These can be obtained because 
on $[0, \tll]$, we have good control over the path of $R^{(v)}_.$, as equation~\eqref{eqn:trivial-exp-gir} and the endpoint $R_{\tll}^{(v)} \usim \frac{m_t}{t}(\tll)$ show that the Girsanov transform to Brownian motion~\eqref{eq.girsanov} can be precisely estimated (in the sense of $\usim$), and thus we will be able to import Brownian motion estimates on $[0, \tll]$. Furthermore, the important result Corollary~\ref{cor:F_ell_1d} and the Markov property will allow us to estimate the probability of $\fT_{\ell}(v)$ by the tail of 1-dimensional branching Brownian motion. Thus, we will be able to precisely calculate $\P_{\x(z)}(G_{L,t}(v))$. 
On the other hand,
the work of Section~\ref{subsec:barrier-equiv-bridge} will allow us to equate the barrier events of~\eqref{def:F-event} and~\eqref{def:G-event} uniformly over certain Brownian bridge laws which appear after some technical estimates (see the Brownian bridge probability of the barrier event in~\eqref{eqn:bar-equiv:Phi-smallint}); thus, we will find~\eqref{eq.1stmom.bananaequiv}. See Section~\ref{sec-2d} for the two-dimensional case.

Towards Step~\textbf{(b)}, we have the following.
\begin{lemma}
\label{lem:2ndmom}
For any $y \in \R$, 
\begin{align}
    \E_{\x(z)}[ \bar{\Lambda}_{L,t} ] \usim
    \E_{ \x(z)} [ \bar{\Lambda}_{L,t}^2 ].
\end{align}
\end{lemma}
Lemma~\ref{lem:2ndmom} is proved in Section~\ref{sec:2ndmom}, also via the Girsanov transform and barrier estimates. Its validity follows from the fact that our events $G_{L,t}(v)$ are suitably de-correlated.
Lemmas~\ref{lem:barrier-equiv},~\ref{lem:2ndmom}, and~\ref{lem:parabola_B0} yield Propositions~\ref{prop:right_tail_equiv} and~\ref{prop:1stmom-asymp} below, which are proved in Sections \ref{subsec-righttailasymp-equiv} and \ref{subsec-righttail-1stmom}, respectively.

\begin{proposition}
\label{prop:right_tail_equiv}
For any $y \in \R$, 
\begin{align}
    \P_{\x(z)} \pth{ 
            R^*_{\tl} > m_t+y
    } 
    \usim \E_{\x(z)}[ \bar{\Lambda}_{L,t}] \, .
    \label{eq.right_tail_equiv}
\end{align}
\end{proposition}
\begin{proposition}
\label{prop:1stmom-asymp}
There exists a constant $\gamma^* > 0$ such that for all $y \in \R$, we have 
\begin{align}
    \E_{\x(z)}\brak{\bar{\Lambda}_{L,t}}
    \usim
    \gamma^* \fM_{L,z} \, .
    \label{eq.1stmom.exact}
\end{align}
\end{proposition}

\begin{proof}[Proof of Theorem~\ref{thm:right_tail_asymp}]
Theorem~\ref{thm:right_tail_asymp} follows immediately from Propositions~\ref{prop:right_tail_equiv} and~\ref{prop:1stmom-asymp}.
\end{proof}

\subsection{Proof of Proposition~\ref{prop:right_tail_equiv}}
\label{subsec-righttailasymp-equiv}
A trivial first-moment bound gives 
\begin{align}
    \P_{\x(z)} \pth{ 
            R^*_{\tl} > m_t+y
    }
    \leq 
    \cP_1 
    +\cP_2 
    +\E_{\x(z)} [ \Gamma_{L,t}] \,,
    \label{eqn:right-tail-equiv:1}
\end{align}
where we define 
\begin{align*}
    \cP_1
    &:= \P_{\x(z)} \bigg( \bigcup_{v \in \cN_{\tl -\ell}} \bigcup_{s \in [0, \tl - \ell]} \{ R_s^{(v)} > B_0(s) \}\bigg) \,, \text{ and} \\
    \cP_2 &:= \P_{\x(z)} \Big( \exists v \in \cN_{\tl} \,:\; R_{\tll}^{(v)} < \frac{t}{\sqrt{d}},~R_{\tl}^{(v)} > m_t +y \Big) \,.
\end{align*}
Note that $\cP_1$ is bounded in Lemma~\ref{lem:parabola_B0}.

\begin{claim}
For all $d \geq 2$ and $L>0$, we have 
\begin{align}
    \lim_{t\to\infty} 
    \sup_{z \in [L^{1/6}, L^{2/3}]}
    \P_{\x(z)} \Big( \exists v \in \cN_{\tl} \,:\; R_{\tll}^{(v)} < \frac{t}{\sqrt{d}},~R_{\tl}^{(v)} > m_t +y \Big) = 0\,.
    \label{eqn:claim-t-rootd}
\end{align}

\begin{proof}[Proof of Claim]
By a union bound and the many-to-one lemma, we have 
\begin{align*}
    \P_{\x(z)} \Big( \exists v \in \cN_{\tl} \,:\; R_{\tll}^{(v)} < \frac{t}{\sqrt{d}},~R_{\tl}^{(v)} > m_t +y \Big)
    \leq 
    e^{\tl} \P_{\x(z)}\Big(R_{\tll} < \frac{t}{\sqrt{d}}, ~R_{\tl} > m_t+y\Big) \,,
\end{align*}
where $(R_s)$ is a $d$-dimensional Bessel process. Denote the left-hand side of the above display $\cP$.
Coupling $(R_s)$ with a $d$-dimensional Wiener process $(W_s^{(1)}, \dots, W_s^{(d)})$ such that $R_s = \|(W_s^{(1)}, \dots, W_s^{(d)}) \|$, it follows that 
\begin{align*}
    \{ R_{\tll} < \frac{t}{\sqrt{d}}, ~R_{\tl} > m_t+y \}
    \subseteq 
    \bigcup_{i=1}^d \Big \{ W_{\tilde t-\ell}^{i} < \frac{t}{\sqrt d}\,, W_{\tilde t}^{i} > \frac{m_t+y}{\sqrt{d}} \Big \}\,,
\end{align*}
whence a union bound shows that
\begin{align}
    \cP 
    \leq de^{\tl}  \P_{\x(z)} \left( W_{\tilde t-\ell} < \frac{t}{\sqrt d}\,, W_{\tilde t}> \frac{m_t+y}{\sqrt{d}}\right)
    \leq 
    t\sqrt{d}e^t \P_{\frac{t}{\sqrt{d}}}\left( W_\ell > \frac{m_t+y}{\sqrt d}\right)\,,
\end{align}
where the second inequality above holds for all $t\geq \ell^{2/3}$ (the Wiener process must travel a distance of order $t$ in a time interval of length $\ell$). That $\cP \to 0$ as $t\to\infty$ then follows by Gaussian tail estimates, concluding the proof of the claim.
\end{proof}
\end{claim}

Dividing both sides of~\eqref{eqn:right-tail-equiv:1} by $\E_{\x(z)}[\bar{\Lambda}_{L,t}]$ and then applying 
the bounds given in~\eqref{eq.lem:parabola_b0}, \eqref{eq.1stmom.bananaequiv}, \eqref{eq.1stmom.finalintegral}, and~\eqref{eqn:claim-t-rootd}, we find
\begin{align}
    \lim_{L \to \infty} \limsup_{t \to \infty} \sup_{z \in \zint} \frac{\P_{\x(z)} \pth{ 
            R^*_{\tl} > m_t+y
    } }
    {
        \E_{\x(z) }[ \bar \Lambda_{L,t}] 
    } 
    \leq 1 \, .
    \label{eq.righttailequiv.1}
\end{align}
For the lower bound, we have 
\begin{align*}
    &\P_{\x(z)} \pth{ 
        R^*_{\tl} > m_t +y
    } 
    \geq 
    \P_{\x(z)}
    \bigg(
        \bigcup_{v \in \cN_{\tl}}
        G_{v,t,L}(y)
    \bigg)
    \geq 
    \frac{\E_{\x(z)}[\bar{\Lambda}_{L,t}]^2}{\E_{\x(z)}[\bar{\Lambda}_{L,t}^2]} \, ,
\end{align*}
where the Paley--Zygmund inequality was used in the last inequality. Lemma~\ref{lem:2ndmom} implies that
\begin{align}
    \lim_{L \to \infty} \liminf_{t \to \infty} \inf_{z \in \zint}
    \frac{\P_{\x(z)} \pth{ 
        R^*_{\tl} > m_t+y
    } }
    {
        \E_{\x(z) }[ \bar \Lambda_{L,t}] 
    }
    \geq 1.
    \label{eq.righttailequiv.2}
\end{align}
Equations~\eqref{eq.righttailequiv.1} and~\eqref{eq.righttailequiv.2} imply~\eqref{eq.right_tail_equiv}.
\qed

\subsection{Proof of Proposition~\ref{prop:1stmom-asymp}}
\label{subsec-righttail-1stmom}

From~\eqref{eq.1stmom.finalintegral}, it remains to show that there exists some
$\gamma >0$ such that
\begin{align*}
g(\ell):= \int_{\ell^{1/3}}^{\ell^{2/3}} w e^{w \sqrt{2}} 
            \P\big( W_{\ell}^* > \sqrt{2}\ell + w \big) \d w
            \usim \gamma.
\end{align*}
Since $\ell$ has no $z$ or $t$ dependence, this amounts to showing 
$
    \lim_{L \to \infty} g(\ell)
    = \gamma \, ,
$
for some $\gamma > 0$.
To see that the limit exists, note that~\eqref{eq.right_tail_equiv} and~\eqref{eq.1stmom.finalintegral} together imply
\begin{align}
    \frac{
        \P_{\x(z)} (R_{\tl}^* > m_t+y)
            }
            {\fM_{L,z}}
    &\usim
            \frac{2^{ \frac{1+\alpha_d}{2} }}
            {\sqrt{\pi}}
            g(\ell) \,.
            \label{eqn:1stmom-no-ell}
\end{align}
The left-hand side above has no $\ell$ dependence, and so it follows that the convergence in~\eqref{eqn:1stmom-no-ell} will hold for any sequence of $\ell:= \ell(L)$ satisfying~\eqref{eqn:def-ell}.
In particular, we find that for any two $\ell:=\ell(L)$  and $\ell':=\ell'(L)$ satisfying~\eqref{eqn:def-ell},
$\lim_{L \to \infty}
    g(\ell)/g(\ell')
    = 1 \, .$
Since \eqref{eq.1stmom.finalintegral} implies that there exist constants $C_1, C_2 >0$ such that  $0 < C_1 \leq g(\ell), g(\ell') \leq C_2$, the previous limit yields that
$\lim_{ L \to \infty}  \abs{ g(\ell) - g(\ell') } = 0$.
Because $\ell, \ell' \to \infty$ as $L \to \infty$ and they were chosen arbitrarily according to~\eqref{eqn:def-ell}, it follows by the Cauchy criterion that there exists some $\gamma >0 $ such that
$\lim_{L \to \infty} g(\ell) = \gamma$. This concludes the proof.
\qed

\begin{remark}
  The convergence of $g(\ell)$ was also proved in Lemmas~4.4 and 4.5 of \cite{ABK13} using asympotics of the solution to the F-KPP equation.
\end{remark}

\section{Proof of Lemma~\ref{lem:barrier-equiv}: a first moment analysis}
\label{sec:1stmom}

Throughout this section, we fix $d\geq 3$, so that  $\alpha_d -\alpha_d^2 \leq 0$, and therefore the $\exp(\cdot)$ factor in~\eqref{eq.girsanov} may be upper-bounded by $1$. See Section~\ref{sec-2d} for the two-dimensional case.

\subsection{Estimates on the right tail of \texorpdfstring{$R_s^*$}{R*-s}}
\label{subsec:bbm-est}

Recall the definition of $\mathfrak{T}_r(v)$ from~\eqref{def:T-tailevent}, and note that since the Bessel process is homogeneous Markov, we have for any $s, r \geq 0$ and any $v \in \cN_s$, 
\[
\P(\mathfrak{T}_r(v)) = \P_{R_s^{(v)}}(R_r^* > m_t+y) \,,
\]
conditional on $R_s^{(v)}$. In light of the events defined in~\eqref{def:F-event}--\eqref{def:G-event}, estimates on $\P_x(R_{\ell}^*>m_t+y)$ for $x > t/\sqrt{d}$ will be crucial to our analysis in Section~\ref{subsec:proof_bananaequiv}.

\begin{lemma}
\label{lem:F_w<c-ell}
For every $\epsilon< \sqrt{2}$ and for all $x \in  [\epsilon t, m_t+y]$, we have 
\begin{align}
    \P_x(R_{\ell}^* >m_t+y) \ulesssim e^{\ell - \frac{(m_t+y-x)^2}{2\ell}}\,,
    \label{eqn:tail-bd-3sqrt2}
\end{align}
where the implied constant depends on $\epsilon$ and $d$.

\begin{proof}
We use a union bound, followed by the many-to-one lemma and the Girsanov transform as usual:
\begin{align*}
    \P_x(R_{\ell}^* >m_t+y)
    &\leq 
    \E_x \bigg[ \sum_{v \in \cN_{\ell}} \one_{ \{R^{(v)}_{\ell} > m_t +y \}}  \bigg]  
    = e^{\ell} \P_x (R_{\ell} > m_t+ y) 
     \\
    &\leq 
    e^{\ell} \E_x \brak{ \pth{\frac{W_{\ell}}{x}}^{\alpha_d} 
        \one_{ \{ W_{\ell} > m_t +y \} }
    } = \frac{e^{\ell}}{\sqrt{2\pi \ell}}
    \int_{m_t+y-x}^{\infty}  \Big(\frac{w+x}{x}\Big)^{\alpha_d}  e^{-\frac{w^2}{2\ell}}\d w ,
    \end{align*}
using that, by our assumption on $x$ we have $((w+x)/x)^{\alpha_d} < C(\epsilon, \delta)$ for some $C(\epsilon, d)>0$ over the range
$w \leq m_t+y+\ell^{2/3}$, say, whereas for larger $w$, the contribution of the integral becomes negligible as first~$t\to\infty$ then $L\to\infty$. Thus, 
\begin{align*}
   \P_x(R_{\ell}^* >m_t+y)
   &\ulesssim e^{\ell}\ell^{-1/2} \int_{m_t+y-x}^{\infty} e^{-\frac{w^2}{2\ell}} \d w 
   \leq C e^{\ell- \frac{(m_t+y-x)^2}{2\ell}} 
\end{align*}
for an absolute constant $C>0$.
\end{proof}
\end{lemma}

A key result here is Corollary~\ref{cor:F_ell_1d}, which states that, uniformly over $x$ lying in an interval around $m_t+y$ of order $\ell$, 
$\P_x(R_{\ell}^*>m_t+y) \usim \P(W^*_\ell > m_t +y -x)$. This will allow us to approximate the branching Bessel process with $1$-dimensional BBM, which has the advantage of being shift-invariant and better understood.
For instance, McKean~\cite{McKean75} showed that 
$u(t,x)=\P(W^*_t > x)$ solves the F-KPP equation 
 $    u_t(t,x) = \frac{1}{2}u_{xx} - u^2+ u$ with Heaviside initial condition 
    $ u(0,x) = \one_{\{x\leq 0\}}$, a special case of the identity $1-u(t,x) = \E\left[ \prod_{v\in \cN_t} f(x+W_t^{(v)})\right]$ for initial data $u(0,x)=1-f(x)$. 
This connection was heavily exploited by Bramson~\cite{Bramson78, Bramson83} to prove convergence in distribution of $W_t^* - m_t(1)$, and will be further exploited throughout the rest of this subsection.

In what follows, we consider the natural coupling  of BBM in $\R$ and a branching $d$-dimensional Bessel process, obtained by using the same branching tree for both processes (hence the same set of particles in both processes at all times), and the same driving Brownian motion for each edge in the tree (to be used in each of the SDEs by the two processes for evaluating the location of the corresponding particle). 
Thus, for all $s >0$, each $v \in \cN_s$, is associated to a Bessel process $R^{(v)}_\cdot$ and a $1$-d Brownian motion $W^{(v)}_\cdot$ satisfying the SDE 
\[
dR^{(v)}_r = \frac{\alpha_d}{R^{(v)}_r} \d r + \d W^{(v)}_r \,.
\]
Note that $R^{(v)}_r \geq W^{(v)}_r$.

\begin{claim}\label{clm:coupling}
Consider the above coupling of $1$-dimensional BBM $\{W_s^{(v)}\}_{s \geq 0, v \in \cN_s}$ and a branching $d$-dimensional Bessel process $\{R_s^{(v)}\}_{s \geq 0, v\in \cN_s}$ started at $x$. Fix $\ell>0$, and let
\[ \cG_x = \Big\{ \min_{v\in \cN_{\ell}} \inf_{0\leq s \leq \ell}  R_s^{(v)} \geq x/4\Big\}\,.\]
Then there exists some constant $C_d>0$ such that for large enough $x$ (in terms of $\ell$),
\[ \sup_{0 \leq  s \leq \ell} \sup_{v\in \cN_s} \left|R_s^{(v)} - W_s^{(v)} \right| \one_{\cG_x} \leq C_d \ell / x \,.
\]
\end{claim}
\begin{proof}
For every particle $v \in \cN_s$, we have that $U_s := U_{s,v} = (R^{(v)}_s - W^{(v)}_s)\one_{\cG_x}$ satisfies, by definition of the coupling, the equation $\d U_s = \alpha_d (R^{(v)}_s)^{-1} \one_{\cG_x} \d s$ with initial condition~$0$, whence
$U_s \leq \int_{0}^s C_d / x \d r \leq 4 C_d s / x$ by definition of $\cG_x$, for some constant $C_d>0$. 
\end{proof}

\begin{corollary}
\label{cor:F_ell_1d}
Fix $c> 0$ and $\ell\geq 1$. Then uniformly over $x \in [m_t+y-c\ell, m_t+y + c\ell]$, we have
        \begin{align}
            \P_x(R_{\ell}^* >m_t+y)
= (1+o(1))\P(W^*_{\ell} > m_t +y-x) \,.
            \label{eqn:tailequiv}
            \end{align}
            where the $o(1)$-term goes to $0$ as $t\to\infty$.
\begin{proof}

Let $\cI := [m_t+y-c\ell, m_t+y + c\ell]$. Then since $\P_x(R_{\ell}^* >m_t+y) \geq \P_x(W^*_{\ell} > m_t +y)$, it suffices by squeezing to show that
\begin{align}
            \sup_{x \in \cI}\P_x(R^*_{\ell} > m_t +y) \leq (1+o(1)) \  \sup_{x\in I} \ \P_x(W^*_{\ell} > m_t +y)\,.
            \label{eqn:tailequiv:goal}
\end{align}
Towards this end, we first observe that the probability on the right-hand is uniformly bounded away from 0 as $t\to\infty$. Indeed,
\[    \inf_{x \in \cI} \P_x (W_{\ell}^* > m_t+y) \geq \P(W_{\ell} > c\ell) > 0
\]
independently of $t$. Next,  if
$K_{\ell,x} := x^2/(8\ell)$,
then  by Markov's inequality,
\[
    \P_x (|\cN_\ell|> e^{K_{\ell,x}})  \leq 
    e^{\ell-K_{\ell,x}} \,.
\]
Thus, recalling the definition of the event $\cG_x$ from Claim~\ref{clm:coupling}, we have that 
    \begin{align*} \P_x\left(\cG_x^c \right)
    &\leq \P_x\big(\cG_x^c \given |\cN_{\ell}| \leq e^{K_{\ell,x}} \big ) + \P_x\left(|\cN_\ell|\geq e^{K_{\ell,x}}\right) \nonumber \\ &\leq
    e^{K_{\ell,x}} \P_x \Big( \inf_{s\in [0,\ell]} R_s \leq x/4 \Big) + e^{\ell-K_{\ell,x}}
    \,,
    \end{align*}
where the last inequality used a union bound and the many-to-one lemma. 
As before, we can express $R_s$ as the norm of a $d$-dimensional Brownian motion, so
\begin{align*} 
\P_x \Big( \inf_{s\in [0,\ell]} R_s \leq x/4 \Big) \leq \P_x \Big(\inf_{s \in [0,\ell]} |W_s| \leq \frac{x}{4} \Big)\leq 
\P_x \Big(\inf_{s\in[0,\ell]} W_s \leq \frac x4\Big)  
\leq 2 e^{- \frac{9x^2}{32\ell} }\,.
\end{align*}
Combining the last two displays along with the definition of $K_{\ell,x}$, we see that
\[
    \P_x\left(\cG_x^c\right) \leq (2+e^\ell) e^{-x^2/(8\ell)}\,.
\]
For every fixed $\ell$, all $x \in \cI$ satisfy $x = \sqrt{2}t(1+o(1))$, and thus
\[  \sup_{x \in \cI} \P_x\left(\cG_x^c\right) \leq C(\ell)e^{ -(\frac14+o(1))t^2/\ell} =o(1) \,.
\]
Finally, by Claim~\ref{clm:coupling}, whenever $\cG$ holds, we have, using that $x\in\cI$,
\[ |R_s^*- \bar W_s^*| \leq \frac{C_d \ell} x \leq \frac{C_d\ell}{m_t+y-c\ell} \] 
for 1-dimensional BBM $\bar W_s^*$, coupled to $R_s^{(\cdot)}$ as in that claim. Hence, for every $x\in\cI$,
\[ \P_x(R_\ell^* > m_t+y) \leq \P_x\Big(W_\ell^* > m_t+y - \frac{C_d \ell}{m_t+y-c\ell}\Big) + \P(\cG^c)\,. 
\]
The proof is concluded by noticing that,
because the solution to the F-KPP equation is continuous, we have the following for every fixed $\ell$:
\[
\P_x \Big(m_t+y -\frac{C_d\ell}{m_t+y-c\ell} \leq W_{\ell}^* \leq  m_t+y\Big) =o(1)\,.
\]
Thus,  we have shown~\eqref{eqn:tailequiv:goal}.
\end{proof}
\end{corollary}

Given this connection between the right-tail of $d$-dimensional branching Bessel process and $1$-dimensional BBM for large initial values, it is apparent that bounds on the right-tail probability $\P(W_t^* > m_t+r)$ will be instrumental to the calculations that follow. 
Our first such bound comes from a result of Bramson~\cite{Bramson83}, stated there in terms of the F-KPP equation and given here in terms of BBM.
Before stating the bound, we define  
\[
\bar{m}_t(1)  := \sup \Big\{ x : \P(W_t^* > x ) \geq \frac{1}{2} \Big\} 
\] 
to be the median of $W_t^*$. 
Note that (cf.~\cite[Eq.~(8.2)]{Bramson83})
\begin{align}
    \bar{m}_t(1) = m_t(1) + O(1)
    \label{eqn:med-diff}
\end{align}
(recalling from~\eqref{def:m_t-c_d} that $m_t(1) := \sqrt{2}t - \frac{3}{2\sqrt{2}}\log t$).

\begin{lemma} [{\cite[Prop.~8.2]{Bramson83}}\footnote{We note that taking $y_0$ to be $-1$ in the notation of~\cite[Proposition~8.2]{Bramson83} gives Lemma~\ref{lem:bram8.2}.}]
\label{lem:bram8.2}
There exists a constant $C>0$ such that for all $x \geq \bar{m}_t(1)+1$ and for all $t$ sufficiently large,  we have
\begin{align}
    \P (W_t^* > x )
    \leq \frac{Ce^t}{\sqrt t} \int_{-1}^0 e^{-\frac{(x-y)^2}{2t}} \big(
    1- e^{-2(y+1)(x- \bar{m}_t(1))/t} \big) \d y \,.
    \label{eqn:bram8.2}
\end{align}
\end{lemma}

\begin{corollary}
\label{cor:bram8.2}
Fix any constant $K>0$ such that $\abs{\bar{m}_t(1) - m_t(1)} \leq K$.
There exists a constant $C>0$ such that for all $w \geq   K+1 -\frac{3}{2\sqrt{2}}\log t$ and for all $t$ sufficiently large,  we have
\begin{align}
    \P (W_t^* > \sqrt{2}t + w )
    \leq C t^{-3/2} \big(w + \frac{3}{2\sqrt{2}} \log t\big) e^{-w\sqrt{2}} e^{- \frac{w^2}{2t}} \,.
    \label{eqn:bram8.2-cor}
\end{align}
\end{corollary}
\begin{proof}[Proof of Corollary~\ref{cor:bram8.2}]
In the notation of Lemma~\ref{lem:bram8.2}, take $x = \sqrt{2}t+w$. 
Note that for the assumed range of $w$, we have 
$ 
    x - \bar{m}_t(1) \geq K+1 - \abs{\bar{m}_t(1)-m_t(1)} \geq 1
$,
so that the conditions of Lemma~\ref{lem:bram8.2} are satisfied.
Furthermore, 
\[
x - \bar{m}_t(1) \leq w+ \frac{3}{2\sqrt{2}}\log t +K < 2(w+ \frac{3}{2\sqrt{2}}\log t)\,.
\]
It follows from the last display, as well as the inequality
$1-x\leq e^{-x} $ for all $x \in \R$, that 
\[
    1- e^{-2(y+1)(x- \bar{m}_t(1))/t} \leq \frac{4(w+ \frac{3}{2\sqrt{2}}\log t)}{t}\,,\qquad
    \mbox{for all $y \in [-1,0]$}\,.
\]
Since $x \geq 0$ and $y \leq 0$, we also have $e^{-(x-y)^2/(2t)} \leq e^{-x^2/(2t)}$. Substituting these two bounds into the right-hand side of~\eqref{eqn:bram8.2} yields~\eqref{eqn:bram8.2-cor}.
\end{proof}

\subsection{Equivalence of barriers}
\label{subsec:barrier-equiv-bridge}
Recall the time parameters $\tl:= t-L$ \eqref{eqn:def-ttilde}, $\ell$ \eqref{eqn:def-ell}, and $\ell_1:= \ell^{1/4}$ \eqref{eqn:def-ell1}; and recall that $y \in \R$ is fixed throughout.
The main result of this subsection is Lemma~\ref{lem:bridge-barrier-equiv}, which states that certain barrier events taking place within the interval $[0,\tl-\ell]$ are equivalent under the laws of a family Brownian bridges whose start and end points all lie in certain windows away from the barriers in question--- see Figure~\ref{fig:lemma-bridge-barrier-equiv}. 
As such, we will make extensive use of the  notation set forth in the second and third paragraphs of Section \ref{subsec:hittingP} and recall the definition of the event $\B_{I}(X.)$, for an interval $I$ and process $X.$, from~\eqref{eqn:def-upper/lowerbarrier}. 
Furthermore, we recall $\x(a):= \sqrt{2}L - a$ and $\y(b) := \frac{m_t}{t}(t-\ell) +y - b$ for any $a, b \in \R$ from~\eqref{def:xz-yw}.
This notation is reminiscent of the notation used in Lemmas \ref{lem:Bram2.6}--\ref{lem:Bram6.1}, which will be the key inputs in the proof of Claim~\ref{claim:bridge-barrier-1} below.

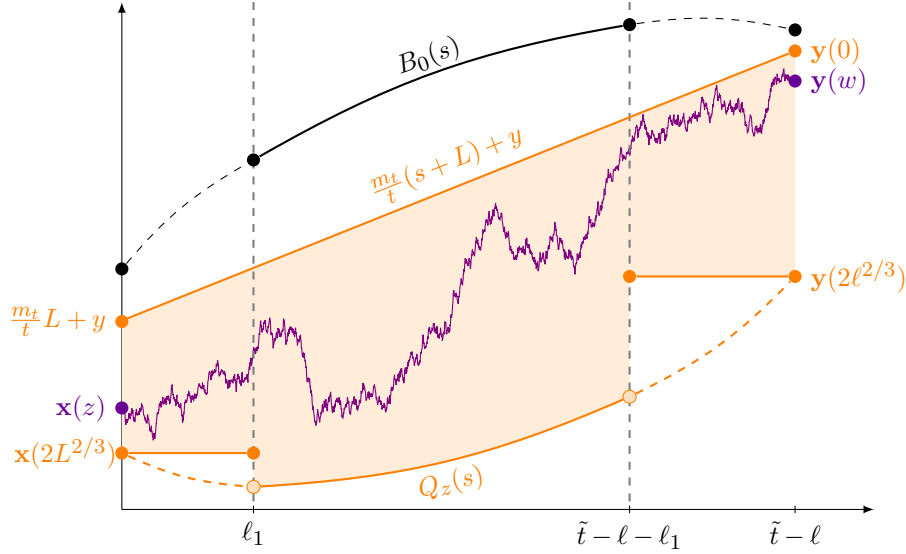
\begin{figure}
     \begin{tikzpicture}[>=latex,font=\small]
 \draw[->] (0, 0) -- (10, 0); 
  \draw[->] (0, 0) -- (0, 6.75);
  \fill[color=orange!15] (0, 2.5) -- (0, 0.75) -- (1.75, 0.75) -- (1.75, 0.3) to[bend right=10] (6.75,1.5) -- (6.75, 3.1) -- (8.95, 3.1) -- (8.95, 6.1);
  \node (fig1) at (4.45,3.3) {
  \includegraphics[width=.645\textwidth]{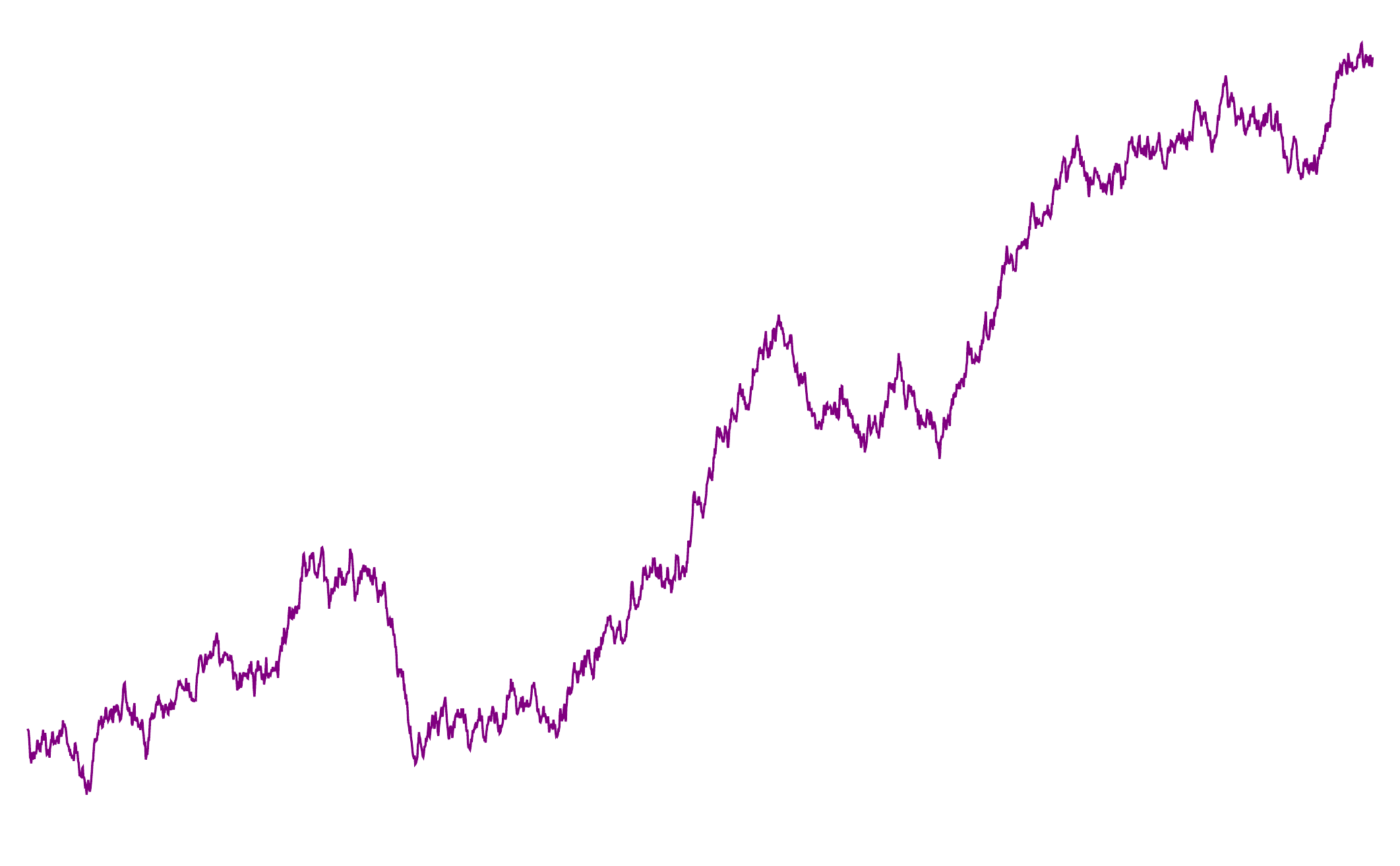}};
  \draw[dashed, thick,gray,dashed] (1.75,6.75) -- (1.75,0);
  \draw[dashed, thick,gray,dashed] (6.75,6.75) -- (6.75,0);
   \node[circle,fill=orange,label={[orange]left:$\frac{m_t}tL+y$}, inner sep=1.75pt] (a) at (0.0,2.5) {};
   \node[circle,fill=orange,label={[orange]right:$\y(0)$},inner sep=1.75pt] (y0) at (8.95,6.1) {};
   \draw[orange] (-0.05,.75) -- (0.05,.75) node [left] {$\x( 2L^{2/3})$};
   \node[circle,fill=blue!45!purple,label={[blue!45!purple]left:$\x(z)$},inner sep=1.75pt] (x) at (0, 1.35) {};
   \node[circle, fill=blue!45!purple, label={[blue!45!purple]right:$\y(w)$}, inner sep=1.75pt] (yw) at ( $(y0)-(0,0.4)$) {};
   \node[circle, fill=orange, label={[orange]right:$\y(2\ell^{2/3})$}, inner sep=1.75pt] (y3) at ( $(y0)-(0,3)$) {};
   \node[circle, fill=orange, inner sep=1.75pt] (a0) at ( 1.75, 0.75 ) {};
   \node[circle, fill=orange, inner sep=1.75pt] (a00) at ( 0, 0.75 ) {};
   \node[circle, color=orange,fill=orange!25, draw, inner sep=1.75pt] (a1) at ( 1.75, 0.3 ) {};
   \node[circle, color=orange,fill=orange!25, draw, inner sep=1.75pt] (b1) at ( 6.75, 1.5 ) {};
   \node[circle, fill=orange, inner sep=1.75pt] (b2) at ( $(y3)-(2.2,0)$ ) {};
   \draw[orange, thick] (a00) -- (a0);
   \draw[orange, thick, dashed] (a00) to[bend right=10]  (a1);
   \draw[orange, thick] (a1) to[bend right=10]  node[midway,below,sloped] {$Q_z(s)$} (b1);
   \draw[orange, thick, dashed] (b1) to[bend right=10]  (y3);
   \draw[orange, thick] (b2) -- (y3);
   \draw[orange, thick] ($(a)+(0.02,0.02)$) -- node[above,sloped,pos=0.5] {$\frac{m_t}t(s+L)+y$} (y0);
   \node[circle, fill=black,inner sep=1.75pt] (B0) at (0,3.2) {};
   \node[circle, fill=black,inner sep=1.75pt] (B1) at (1.75,4.65) {};
   \node[circle, fill=black,inner sep=1.75pt] (B2) at (6.75,6.45) {};
   \node[circle, fill=black,inner sep=1.75pt] (B3) at (8.95,6.38) {};
   \draw[dashed] (B0) to[bend left=10] (B1);
   \draw[thick] (B1)  to[bend left=10] node[above,sloped,inner sep=1pt] {$B_0(s)$} (B2);
   \draw[dashed] (B2) to[bend left=10] (B3);
  \draw (1.75,0.05) -- (1.75,-0.05) node[below] {$\ell_1$};
  \draw (6.75,0.05) -- (6.75,-0.05) node[below] {$\tilde t - \ell-\ell_1$};
  \draw (8.95,0.05) -- (8.95,-0.05) node[below] {$\tilde t-\ell$};
  \end{tikzpicture}
     \caption{The barrier events in Lemma~\ref{lem:bridge-barrier-equiv}: barriers in orange and black, and a Brownian bridge in purple. One event constrains the bridge to be below the black curve on $[\ell_1, \tll-\ell_1]$ (numerator of~\eqref{eqn:bridge-barrier-equiv}), and the other to the orange region  in $[0,\tll]$ (denominator of~\eqref{eqn:bridge-barrier-equiv}).}
     \label{fig:lemma-bridge-barrier-equiv}
 \end{figure}

\begin{lemma}
\label{lem:bridge-barrier-equiv}
Recall the barrier function $B_0(\cdot)$ from~\eqref{def:B0}. For any fixed $y \in \R$, we have
\begin{align}
    \sup_{w \in [\ell^{1/3}, \ell^{2/3}]} 
        \dfrac{
            \P_{\x(z), \tl-\ell}^{\y(w)}
            \pth{
            \UB{[\ell_1,\tl-\ell-\ell_1]}{B_0 }(W_\cdot)
            }
        }
        {
            \P_{\x(z),\tl-\ell}^{ \y(w) } 
            \pth{
                \B_{[0,\tl-\ell]}(W_\cdot)
            }
        }
        \usim 1.
        \label{eqn:bridge-barrier-equiv}
\end{align}
\end{lemma}
Note that the numerator in~\eqref{eqn:bridge-barrier-equiv} is greater than or equal to the denominator. 

\begin{proof}[Proof of Lemma~\ref{lem:bridge-barrier-equiv}]
Equation~\eqref{eqn:bridge-barrier-equiv} follows immediately from Claims \ref{claim:bridge-barrier-1} and \ref{claim:bridge-barrier-2} below.
\end{proof}

\begin{claim}
\label{claim:bridge-barrier-1}
For any fixed $y \in \R$, we have
\begin{align}
    \sup_{w \in [\ell^{1/3}, \ell^{2/3}]}
        \dfrac{
            \P_{\x(z),\tl-\ell}^{\y(w)}
            \pth{
           \UB{[\ell_1,\tl-\ell-\ell_1]}{B_0 }(W_\cdot)
            }
        }
        {
            \P_{\x(z),\tl-\ell}^{\y(w)} 
            \pth{
                \B_{[\ell_1,\tl-\ell-\ell_1]}(W_\cdot)
            }
        }
        \usim 1.
        \label{eqn:claim-bridge-barrier-1}
\end{align}
\end{claim}

\begin{claim}
\label{claim:bridge-barrier-2}
For any fixed $y \in \R$, we have
\begin{align}
    \sup_{w \in [\ell^{1/3}, \ell^{2/3}]}
        \dfrac{
            \P_{\x(z),\tl-\ell}^{\y(w)}
            \pth{
                \B_{[\ell_1,\tl-\ell-\ell_1]}(W_\cdot)
            }
        }
        {
            \P_{\x(z),\tl-\ell}^{ \y(w) } 
            \pth{
                \B_{[0,\tl-\ell]}(W_\cdot)
            }
        }
        \usim 1.
        \label{eqn:claim-bridge-barrier-2}
\end{align}
\end{claim}
Note that the numerators of the left-hand sides of~\eqref{eqn:claim-bridge-barrier-1} and~\eqref{eqn:claim-bridge-barrier-2} are greater than or equal to their respective denominators. 

\begin{proof}[Proof of Claim~\ref{claim:bridge-barrier-1}]
The idea of the proof is to apply Lemmas \ref{lem:Bram2.6}-\ref{lem:Bram6.1}.
First, note that for any $z \in \zint$ and $w \in [\ell^{1/3}, \ell^{2/3}]$, we have 
\begin{align}
    Q_z(s) \leq f_{\x(z)}^{\y(w)}(s;\tl- \ell) - (s\wedge (\tl-\ell-s))^{2/3}
    \label{eqn:claim-bridge-1:Qlowest}
\end{align}
for all $s \in [\ell_1, \tl-\ell-\ell_1]$.
This is clear because 
$
    f_{\x(z)}^{\y(w)}(s;\tl- \ell) 
    \geq 
    f_{\x(L^{2/3})}^{\y(\ell^{2/3})}(s;\tl- \ell) 
$ for all $s \in [0, \tl -\ell]$, and the latter expression is greater than or equal to $Q_z(s) +(s\wedge (\tl-\ell-s))^{2/3}$, for all $s \in [ \ell_1, \tl-\ell-\ell_1]$.\footnote{The statement is true for all $s \in [0, \tl-\ell]$, but we will only need it to hold on the smaller interval $[\ell_1, \tl-\ell-\ell_1]$.} 
Let $\mathbf{f}(s) := f_{\x(z)}^{\y(w)}(s;\tl- \ell) - (s\wedge (\tl-\ell-s))^{2/3}$.  
By Lemma~\ref{lem:Bram2.7}, we have that
\begin{align}
    \P_{\x(z),\tl-\ell}^{\y(w)} \Big (
    \LB{[\ell_1, \tl-\ell-\ell_1]}{\mathbf{f}}(W_\cdot) 
    \given 
    \UB{[\ell_1, \tl-\ell-\ell_1]}{f_{\x(z)}^{\y(w)}(\cdot;\tl- \ell)}(W_\cdot)
    \Big ) \to 1 \text{ as }  L \to \infty \, ,
    \label{eqn:claim-bridge-1:bram2.7}
\end{align}
uniformly over $\tl - \ell \geq 2\ell_1$ and over $w \in [\ell^{1/3}, \ell^{2/3}]$ and $z \in \zint$ (in fact, the left-hand side of~\eqref{eqn:claim-bridge-1:bram2.7} is constant in $w$ and $z$--- see the last sentence of Lemma~\ref{lem:Bram2.7}). 
Next, we apply Lemma~\ref{lem:Bram2.6}. In the notation of this lemma, we take
\begin{align*}
    \cL(s) &:=
    \begin{cases}
        -\infty, &s \in [0, \ell_1] \cup [\tl - \ell - \ell_1, \tl -\ell] \\
        \mathbf{f}(s), &\text{for } s \in [\ell_1, \tl-\ell-\ell_1] 
    \end{cases} \, ,
    \\
    \cU_1(s) &:= 
    \begin{cases}
        \infty, &s \in [0, \ell_1] \cup [\tl - \ell - \ell_1, \tl -\ell] \\
        f_{\x(z)}^{\y(w)}(s;\tl- \ell), &\text{for } s \in [\ell_1, \tl-\ell-\ell_1] 
    \end{cases} \, ,
    \text{ and }
    \\
    \cU_2(s) &:= 
    \begin{cases}
        \infty, &s \in [0, \ell_1] \cup [\tl - \ell - \ell_1, \tl -\ell] \\
        \frac{m_t}{t}(s+L) +y, &\text{for } s \in [\ell_1, \tl-\ell-\ell_1] 
    \end{cases} \, ,
\end{align*}
which yields
\begin{align}
    &\P_{\x(z),\tl-\ell}^{\y(w)} \pth{
    \LB{[\ell_1, \tl-\ell-\ell_1]}{\mathbf{f}}(W_\cdot) 
    \given 
    \UB{[\ell_1, \tl-\ell-\ell_1]}{f_{\x(z)}^{\y(w)}(\cdot;\tl- \ell)}(W_\cdot)
    } 
    \nonumber \\
    &
    \leq \P_{\x(z),\tl-\ell}^{\y(w)} \pth{
    \LB{[\ell_1, \tl-\ell-\ell_1]}{\mathbf{f}}(W_\cdot) 
    \given  
    \UB{[\ell_1, \tl-\ell-\ell_1]}{\frac{m_t}{t}(\cdot +L)+ y}(W_\cdot)
    } \, . 
    \label{eqn:claim-bridge-1:bram2.6-a}
\end{align}
Equation~\eqref{eqn:claim-bridge-1:Qlowest} states that $Q_z(s) \leq \mathbf{f}(s)$ for all $s \in [\ell_1, \tl - \ell - \ell_1]$, and so \eqref{eqn:claim-bridge-1:bram2.6-a} yields
\begin{align}
    &\P_{\x(z),\tl-\ell}^{\y(w)} \pth{
    \LB{[\ell_1, \tl-\ell-\ell_1]}{\mathbf{f}}(W_\cdot) 
    \given 
    \UB{[\ell_1, \tl-\ell-\ell_1]}{f_{\x(z)}^{\y(w)}(\cdot;\tl- \ell)}(W_\cdot)
    } \nonumber \\
    &\leq \P_{\x(z),\tl-\ell}^{\y(w)} \pth{
    \LB{[\ell_1, \tl-\ell-\ell_1]}{Q_z}(W_\cdot) 
    \given 
    \UB{[\ell_1, \tl-\ell-\ell_1]}{\frac{m_t}{t}(\cdot +L)+ y}(W_\cdot)
    } 
    \label{eqn:claim-bridge-1:bram2.6-b}
\end{align}
It then follows from~\eqref{eqn:claim-bridge-1:bram2.7} and~\eqref{eqn:claim-bridge-1:bram2.6-b} that 
\begin{align}
    \sup_{w \in [\ell^{1/3}, \ell^{2/3}]}
    \dfrac{ 
    \P_{\x(z),\tl-\ell}^{\y(w)} \pth{
        \UB{[\ell_1, \tl-\ell-\ell_1]}{\frac{m_t}{t}(\cdot +L)+ y}(W_\cdot)
    }
    }
    { 
    \P_{\x(z),\tl-\ell}^{\y(w)} \pth{
    \B_{[\ell_1, \tl-\ell-\ell_1]}(W_\cdot) 
    } 
    }
    \usim 1 \, .
    \label{eqn:claim-bridge-1:bram2.6+7}
\end{align}

Next, 
we are interested in applying Lemma~\ref{lem:Bram6.1}: in its notation, we take $a := \frac{m_t}{t}L+y$, $b := \y(0)$, 
$C := \mathcal{K}_d$, $\delta := 1/3$, $T := \tl- \ell$, $r := \ell_1$, and 
$\cL_{\tl-\ell}(s) := \frac{m_t}{t}(s+L)+y$.
Note that the condition in~\eqref{eqn:bram6.1-cond} holds with $r_0=0$ (since we are in the case $f_a^b(s; \tll) = \cL_{\tll}(s)$). 
Lemma~\ref{lem:Bram6.1}  then gives the following convergence uniformly over all~$z \in \zint$ and~$w\in[\ell^{1/3}, \ell^{2/3}]$:
\begin{align}
    \dfrac{ 
    \P_{\x(z), \tl-\ell}^{\y(w)} \Big(
    \UB{[\ell_1, \tl-\ell-\ell_1]}{\overline{\cL}_{\tl-\ell} }(W_\cdot) 
    \Big) 
    }
    {
    \P_{\x(z),\tl-\ell}^{\y(w)} \pth{
    \UB{[\ell_1, \tl-\ell-\ell_1]}{\frac{m_t}{t}(\cdot +L)+ y}(W_\cdot) 
    } 
    }
    \usim 1 \, .
    \label{eqn:claim-bridge-1:bramlem2}
\end{align}
Recall the definition of $B_0(s)$ from~\eqref{def:B0}, and consider $t$ and $L$ large enough such that, for all $s \in [\ell_1, \tl-\ell-\ell_1]$,
\begin{align*} 
\log \ell+ \mathcal{K}_d \log (s\wedge (\tl-\ell-s))_+ \leq \mathcal{K}_d (s\wedge (\tl-\ell-s))^{1/3} \,.  
\end{align*}
This implies that 
$
\frac{m_t}{t}(s +L)+ y \leq 
B_0(s) \leq \overline{\cL}_{\tl-\ell}(s)
$
for all $s \in [\ell_1, \tl-\ell-\ell_1]$ and for all~$t$ and~$L$ large enough, and thus from~\eqref{eqn:claim-bridge-1:bramlem2} we see
\begin{align}
    \sup_{w \in [\ell^{1/3}, \ell^{2/3}]}
    \dfrac{ 
    \P_{\x(z), \tl-\ell}^{\y(w)} \pth{
    \UB{[\ell_1, \tl-\ell-\ell_1]}{B_0 }(W_\cdot)
    } 
    }
    {
    \P_{\x(z),\tl-\ell}^{\y(w)} \pth{
    \UB{[\ell_1, \tl-\ell-\ell_1]}{\frac{m_t}{t}(\cdot +L)+ y}(W_\cdot) 
    } 
    }
    \usim 1 \,.
    \label{eqn:claim-bridge-1:bram6.1}
\end{align}
Equation~\eqref{eqn:claim-bridge-barrier-1} then follows from~\eqref{eqn:claim-bridge-1:bram2.6+7} and~\eqref{eqn:claim-bridge-1:bram6.1}.
\end{proof}

\begin{proof}[Proof of Claim~\ref{claim:bridge-barrier-2}]
From~\eqref{eqn:brz-modified}, we have 
\begin{align}
    \P_{\x(z),\tl-\ell}^{ \y(w) } 
            \pth{
                \B_{[0,\tl-\ell]}(W_\cdot)
            }
        \uasymp \frac{zw}{t} 
        \geq \frac{z \ell^{1/3}}{t}
        \, .
    \label{eqn:claim-barrier-2:good}
\end{align}
Define the hitting time $\rho :=  \inf\{ s: W_s \not \in [ Q_z(s), \frac{m_t}{t}(s+L)+y ] \}$. 
We wish to show that   
\begin{align*}
    \cP&:= \P_{\x(z),\tl-\ell}^{\y(w)} 
    \pth{
    \big \{ \rho \in [ 0,  \ell_1) \cup ( \tl - \ell - \ell_1, \tl -\ell] \big \}
    \bigcap
    \B_{[\ell_1, \tl-\ell-\ell_1]}
    (W_\cdot) 
    } 
\end{align*}
is negligible compared to the right-hand side of~\eqref{eqn:claim-barrier-2:good}. 
From a union bound, we may write
\begin{align}
    \cP &\leq 
    \P_{\x(z),\tl-\ell}^{\y(w)} 
    \pth{
    \big \{ \rho \in [ 0,  \ell_1) \big \}
    \bigcap
    \B_{[\ell_1, \tl-\ell-\ell_1]}
    (W_\cdot) 
    }     
    \nonumber \\
    &+
    \P_{\x(z),\tl-\ell}^{\y(w)} 
    \pth{
    \big \{
    \rho \in ( \tl - \ell - \ell_1, \tl -\ell]
    \big \}
    \bigcap
    \B_{[\ell_1, \tl-\ell-\ell_1]}
    (W_\cdot) 
    } .
    \label{eqn:claim-barrier-2:unionbd}
\end{align}
For an interval $I \subseteq \R$, a constant $\theta \in \R$, and a real-valued process $X$, define the event
\[
\cB^{\Bumpeq+\theta}_I(X_\cdot) =
    \UB{I}{\frac{m_t}{t}(\cdot+L+\theta)+y}(X_\cdot)
    \cap 
    \LB{I}{Q_z(\cdot+\theta)}(X_\cdot) \,.
\]    
Note that, for any $r_1, r_2 \in \R$, we have
\[
    \P_{r_1, \tl-\ell-2\ell_1}^{r_2} \pth{
    \cB^{\Bumpeq+\ell_1}_{[0, \tl-\ell-2\ell_1]}(W_\cdot)
    }
    =
    \P_{\x(z), \tl-\ell}^{\y(w)} \pth{
    \cB^{\Bumpeq}_{[\ell_1, \tl-\ell-\ell_1]}(W_\cdot)
    \given W_{\ell_1} = r_1,~W_{\tl-\ell-\ell_1} = r_2
    }\,.
\]
Then, integrating the right-hand side of~\eqref{eqn:claim-barrier-2:unionbd} over the positions $r_1$ and $r_2$ of $W$ at times $\ell_1$ and $\tll- \ell_1$, respectively, and utilizing the Markov property at times $\ell_1$ and $\tl-\ell-\ell_1$, we find
\begin{align}
    \cP  
    &\leq 
    p_{\tll}^W(\x(z),\y(w))^{-1}
    \int_{Q_z(\ell_1)}^{ \frac{m_t}{t}(L+\ell_1)+y} 
    \int_{ Q_z(\tl-\ell-\ell_1) }^{\frac{m_t}{t}(t-\ell-\ell_1)+y} 
    \!\!\d r_1
    \d r_2 
    ~p_{\ell_1}^W( \x(z), r_1)
    p_{\tl -\ell - 2\ell_1}^W(r_1, r_2)
    \nonumber 
    \\ 
    &\quad \times 
    p_{\ell_1}^W( r_2, \y(w))
    \P_{r_1, \tl-\ell-2\ell_1}^{r_2} \pth{
    \cB^{\Bumpeq+\ell_1}_{[0, \tl-\ell-2\ell_1]}(W_\cdot)
    }
    \pth{ \P_{\x(z), \ell_1}^{r_1}( \rho < \ell_1 ) 
    + \P_{r_2,\ell_1}^{\y(w)} ( \rho < \ell_1 ) }\nonumber\\
 &=: \cP^* \, ,
    \label{eqn:claim-barrier-2:bad.1}
\end{align}
From a trivial upper bound, equation~\eqref{eqn:brownian-ballot-sim}, 
and the fact that $\frac{m_t}{t}(t-\ell-\ell_1)+y - r_2 \leq 2 \ell^{2/3}$, 
we have
\begin{align}
    \P_{r_1, \tl-\ell-2\ell_1}^{r_2} \pth{
    \cB^{\Bumpeq+\ell_1}_{[0, \tl-\ell-2\ell_1]}(W_\cdot)
    }
    &\leq 
    \P_{r_1, \tl-\ell-2\ell_1}^{r_2} \pth{
    \UB{[0, \tl-\ell-2\ell_1]}{\frac{m_t}{t}(\cdot+L+\ell_1)+y }(W_\cdot)
    }
    \nonumber \\
    &\leq 
    \frac{ 2\ell^{2/3}(\frac{m_t}{t}(L+ \ell_1) +y - r_1) }{t} \, \, ,
    \label{eqn:claim-barrier-2:bad.barrier}
\end{align}
which holds uniformly over $r_1$ and $r_2$ in their respective ranges. 
Define $\rx:= r_1 - \x(z)$ and $\ry := r_2 - \y(w)$. Then
using the bound \eqref{eqn:claim-barrier-2:bad.barrier} in  \eqref{eqn:claim-barrier-2:bad.1}, changing variables from $(r_1, r_2)$ to $(\rx, \ry)$, and expanding the Gaussian densities involving $\rx$ and $\ry$ yields 
\begin{align}
    \cP^* 
    &\ulesssim
    \frac{\ell^{2/3}}{t}
    p_{\tll}^W(\x(z),\y(w))^{-1}
    \int_{z - 2L^{2/3}}^{\zeta}
    \d\rx
    ~(\zeta-\rx)
    e^{-\frac{\rx^2}{2\ell_1}} 
    \nonumber \\
    &\times \int_{w-2\ell^{2/3}}^{w - \frac{m_t}{t}\ell_1}
    \d\ry
    ~p_{\tll-2\ell_1}^W (r_1, r_2)
    e^{-\frac{\ry^2}{2\ell_1}} 
    \pth{ \P_{\x(z), \ell_1}^{r_1}( \rho < \ell_1 ) 
    + \P_{r_2,\ell_1}^{\y(w)} ( \rho < \ell_1 ) } \, ,
    \label{eqn:claim-barrier-2:integrate}
\end{align}
where we recall that $\fo_t := \frac{m_t}{t}-\sqrt{2}$, and define
\[
\zeta:= \zeta(t,L,z) = \frac{m_t}{t}(L+\ell_1)+y - \x(z) = z+ \frac{m_t}{t}\ell_1+y+\fo_t L \,.
\]
We remark that \eqref{eqn:claim-barrier-2:integrate} holds uniformly over $w$ in $[\ell^{1/3}, \ell^{2/3}]$ (in what follows, we always write ``uniformly in $w$" to mean ``uniformly over $w$ in $[\ell^{1/3}, \ell^{2/3}]$").
We proceed by examining the term
\[
p_{\tll}^W(\x(z),\y(w))^{-1} p_{\tll - 2\ell_1}^W(r_1,r_2) 
\]
in~\eqref{eqn:claim-barrier-2:integrate}. Consider the quantity
$
    c:= (\tll)/(\tll-2\ell_1)
$.
Then, uniformly in $w$,
\begin{align}
    \frac{p_{\tll - 2\ell_1}^W(r_1,r_2)}{p_{\tll}^W(\x(z),\y(w))}
    &\usim
    \exp \pth{
    \frac{c(\x(z)-\y(w))^2 - (r_1-r_2)^2}{2(\tll - 2\ell_1)}
    }\, .
    \label{eqn:claim-barrier-2-gaussratio}
\end{align}
Now, since $\x(z)$, $\rx$, and $\ry$ are all bounded in absolute value by a constant times $L^{2/3}$ for all $L$ and $t$ sufficiently large, we have 
\begin{align}
    c(\x(z)-\y(w))^2 - (r_1-r_2)^2 \lesssim
    L^{4/5} + (c-1)\y(w)^2 + 2\y(w) (\rx - \ry +(1-c) \x(z)) \, ,
    \label{eqn:claim-barrier-2-gaussexpand}
\end{align}
uniformly in $w$.
It follows from equations~\eqref{eqn:claim-barrier-2-gaussratio} and~\eqref{eqn:claim-barrier-2-gaussexpand} that 
\begin{align}
    \frac{p_{\tll - 2\ell_1}^W(r_1,r_2)}{p_{\tll}^W(\x(z),\y(w))}
    \ulesssim
    \exp \pth{
    \frac{(c-1)\y(w)^2 + 2\y(w) (\rx - \ry +(1-c) \x(z))}{2(\tll- 2\ell_1)}
    } \, ,
    \label{eqn:claim-barrier-2-gaussexpand2}
\end{align}
uniformly in $w$.
Note that $\y(w) = \sqrt{2}(\tll-2\ell_1) \big(1+f_1 \big)$, where $f_1:= f_1(t,L,w)$ satisfies 
\begin{align*}
    \sup_{w \in [\ell^{1/3}, \ell^{2/3}]} f_1 \ulesssim t^{-1} \log t \,. 
\end{align*}
This expression for $\y(w)$ and~\eqref{eqn:claim-barrier-2-gaussexpand2}, along with the relation $c-1 = 2\ell_1 (\tll-2\ell_1)^{-1}$, gives 
\begin{align}
    \sup_{w \in [\ell^{1/3}, \ell^{2/3}]} \frac{p_{\tll - 2\ell_1}^W(r_1,r_2)}{p_{\tll}^W(\x(z),\y(w))}
    \ulesssim 
    e^{2 \ell_1 + \sqrt{2}\ry - \sqrt{2}\rx } \,.
    \label{eqn:claim-barrier-2-finalgauss}
\end{align}
Substituting \eqref{eqn:claim-barrier-2-finalgauss} into \eqref{eqn:claim-barrier-2:integrate} and then completing the square 
\[
-\frac{r^2}{2\ell_1} \pm \sqrt{2}r = \ell_1-\frac{(r\mp \sqrt{2}\ell_1)^2}{2\ell_1}
\]
gives the following estimate uniformly over $w$:
\begin{align}
    \cP^* 
    &\ulesssim
    \frac{\ell^{2/3}e^{4\ell_1}}{t}
    \int_{z - 2L^{2/3}}^{\zeta}
    \d\rx
    ~(\zeta-\rx)
    e^{-\frac{(\rx+ \sqrt{2}\ell_1)^2}{2\ell_1}}
    \nonumber \\
    &\times \int_{w-2\ell^{2/3}}^{w - \frac{m_t}{t}\ell_1}
    \d\ry
    ~e^{-\frac{(\ry- \sqrt{2}\ell_1)^2}{2\ell_1}} 
    \pth{ \P_{\x(z), \ell_1}^{r_1}( \rho < \ell_1 ) 
    + \P_{r_2,\ell_1}^{\y(w)} ( \rho < \ell_1 ) } \,, 
    \label{eqn:claim-barrier-2:int-gauss-sub}
\end{align}
Note that, uniformly over $\rx \in [-\sqrt{2}\ell_1 - \ell_1^{2/3}, -\sqrt{2}\ell_1 + \ell_1^{2/3}]$, we have $\zeta - \rx \usim z$. It follows from Gaussian integration that
\begin{align}
    \int_{z - 2L^{2/3}}^{\zeta}
    (\zeta-\rx)
    e^{-\frac{(\rx+ \sqrt{2}\ell_1)^2}{2\ell_1}} 
    \d\rx
    &\usim
    z\ell_1^{1/2} \,, \text{and}
    \label{eqn:claim-barrier-2:full-int-r1}
    \\
    \int_{w-2\ell^{2/3}}^{w - \frac{m_t}{t}\ell_1}
    e^{-\frac{(\ry- \sqrt{2}\ell_1)^2}{2\ell_1}} ~\d\ry
    &\usim \ell_1^{1/2} \, .
    \label{eqn:claim-barrier-2:full-int-r2}
\end{align}
Writing the right-hand side of~\eqref{eqn:claim-barrier-2:int-gauss-sub} as a sum of two terms, applying \eqref{eqn:claim-barrier-2:full-int-r1} -- \eqref{eqn:claim-barrier-2:full-int-r2}, and then applying the trivial bound $\ell^{2/3} \ell_1^{1/2} \leq e^{\ell_1}$ for $L$ large enough yields
\begin{align}
    \cP^* \ulesssim \cP_1^* + \cP_2^* \,,
    \label{eqn:claim-barrier-2:p*}
\end{align}
uniformly in $w$, where
 \[   \cP_1^* := \frac{e^{5\ell_1}}{t} \int_{z - 2L^{2/3}}^{\zeta}
    (\zeta-\rx)
    e^{-\frac{(\rx+ \sqrt{2}\ell_1)^2}{2\ell_1}}
    \P_{\x(z), \ell_1}^{r_1}( \rho < \ell_1 )
    \d\rx
    =: \frac{e^{5\ell_1}}{t} \int_{z - 2L^{2/3}}^{\zeta} \mathfrak{I}_1 \d\rx
    \,,\]
     and
     \[
    \cP_2^* := \frac{z e^{5\ell_1}}{t}  \int_{w-2\ell^{2/3}}^{w - \frac{m_t}{t}\ell_1}
    e^{-\frac{(\ry- \sqrt{2}\ell_1)^2}{2\ell_1}} 
    \P_{r_2,\ell_1}^{\y(w)} ( \rho < \ell_1 ) \d\ry \\
    =: \frac{ze^{5\ell_1}}{t}  \int_{w-2\ell^{2/3}}^{w - \frac{m_t}{t}\ell_1} \mathfrak{I}_2 \d\ry \,.\]
The idea for bounding $\cP_1^*$ and $\cP_2^*$ will be the same. Deep enough into the tails of the Gaussian terms, the exponential decay will kill the $e^{5\ell_1}$ pre-factor, and so we may simply upper bound the hitting probability by $1$. Outside of these ``deep tails", the hitting probabilities will dominate, providing sufficient exponential decay so that the Gaussian term may  be upper bounded by $1$. 

We begin with the computations for  $\cP_1^*$.
Consider the intervals 
\[
\cI_1 := [z-2L^{2/3}, -\sqrt{2}\ell_1-L^{1/10}]\cup [\zeta- L^{1/10}, \zeta], \quad 
\cI_2 := (-\sqrt{2}\ell_1-L^{1/10}, \zeta- L^{1/10}) \, .
\]
On $\cI_1$, which has length $\abs{\cI_1} \ulesssim L^{2/3}$, we use the following bounds for all $L$ and $t$ sufficiently large:
\[
\zeta- \rx \leq 4L^{2/3}, \quad 
e^{- \frac{(\rx+\sqrt{2}\ell_1)^2}{2\ell_1}} \leq e^{-\frac{L^{1/5}}{2\ell_1}} \leq e^{-L^{1/10}}\,, \quad 
\P_{\x(z), \ell_1}^{r_1}(\rho<\ell_1) \leq 1\,.
\]
This yields
\begin{align}
    \frac{e^{5\ell_1}}{t}\int_{\cI_1} \mathfrak{I}_1 ~\d\rx \ulesssim \frac{L^{4/3}}{t} e^{5\ell_1-L^{1/10}} \ulesssim \frac{1}{t}e^{-L^{1/11}}\,,
    \label{eqn:claim-barrier-2:I1-1}
\end{align}
since $\ell_1 = \ell^{1/4} \leq L^{1/24}$\,.
On $\cI_2$, we find exponential decay of the hitting probability as follows. 
Consider the hitting times
 \[   \overline{\rho}_1 := \inf \{ s : W_s > \frac{m_t}{t}(s+L)+y \}\, \text{ and}\;
    \underline{\rho}_1 := \inf \{ s : W_s < Q_z(0)\}\,,\]
from which we obtain the union bound
\begin{align}
    \P_{\x(z),\ell_1}^{r_1}( \rho < \ell_1) 
    &\leq 
    \P_{\x(z),\ell_1}^{r_1} \pth{ \overline{\rho}_1 < \ell_1 }
    +\P_{\x(z),\ell_1}^{r_1} \pth{ \underline{\rho}_1 < \ell_1 } \,. \label{eqn:claim-barrier-2:hit1-union}
\end{align}
We  use~\eqref{eqn:brownian-ballot-explicit} to compute the hitting probabilities on the right-hand side of~\eqref{eqn:claim-barrier-2:hit1-union}:
\begin{align}
    \P_{\x(z),\ell_1}^{r_1} \pth{ \overline{\rho}_1 < \ell_1 }
    &=
    \exp \bigg(
    - \frac{2 \pth{ \frac{m_t}{t}L+y - \x(z)} \pth{\frac{m_t}{t}(L+\ell_1)+y-r_1}}{\ell_1}
    \bigg) \nonumber \\
    &\ulesssim
    \exp \bigg(
    - \frac{2z (\zeta-\rx)}{\ell_1}
    \bigg)
    \,,
    \label{eqn:claim-barrier-2:hit1-1}\\
    \P_{\x(z),\ell_1}^{r_1} \pth{ \underline{\rho}_1 < \ell_1 }
    &=
    \exp \bigg(
    - \frac{2 \pth{\x(z) - Q_z(0)}\pth{r_1-Q_z(0)}}{\ell_1}
    \bigg)
    \nonumber \\
    &\leq 
    \exp \bigg(
    - \frac{2L^{2/3}\pth{\rx+2L^{2/3}-z}}{\ell_1}\bigg) \,.
    \label{eqn:claim-barrier-2:hit1-2}
\end{align}
Note that on $\cI_2$, we have $L^{1/10} < \zeta- \rx$, from which \eqref{eqn:claim-barrier-2:hit1-union}---\eqref{eqn:claim-barrier-2:hit1-2} yield
\begin{align*}
    \P_{\x(z),\ell_1}^{r_1}( \rho < \ell_1) \usim \P_{\x(z),\ell_1}^{r_1} \pth{ \overline{\rho}_1 < \ell_1 } \ulesssim 
    e^{- 2zL^{1/10}/\ell_1} \ulesssim e^{-L^{1/10}}  \,.
\end{align*}
Also on $\cI_2$, which has length $\abs{\cI_2} \ulesssim L^{2/3}$, we have $\zeta - \rx \ulesssim z \leq L^{2/3}$. It follows from bounding the Gaussian term by $1$ and the last display that 
\begin{align}
    \frac{e^{5\ell_1}}{t}\int_{\cI_2}
    \mathfrak{I}_1 \d\rx &\ulesssim \frac{L^{4/3}}{t} e^{5\ell_1 - L^{1/10}} \ulesssim \frac{1}{t}e^{-L^{1/11}} \,,
    \label{eqn:claim-barrier-2:I1-2}
\end{align}
from which \eqref{eqn:claim-barrier-2:I1-1} and~\eqref{eqn:claim-barrier-2:I1-2} yield  
\begin{align}
    \cP_1^* \ulesssim \frac{1}{t}e^{-L^{1/11}} \,.
    \label{eqn:claim-barrier-2:p1}
\end{align}

We next turn our attention to bounding $\cP_2^*$ in a similar fashion.  $\cP_2^*$ depends on $w$, but it is easily seen from the calculations that follow  that all of our estimates are uniform over $w$; thus, we do not mention this uniformity any longer. 
This time, consider the following regions:
\[
\cI_4:= [w-2\ell^{2/3}, -2\sqrt{2}\ell_1] \cup [4\sqrt{2}\ell_1, w- (m_t/t)\ell_1]\,, \quad
\cI_5:= (-2\sqrt{2}\ell_1, 4\sqrt{2}\ell_1)\,.
\]
On $\cI_4$, which has length $\abs{\cI_4} \ulesssim \ell^{2/3}$, we have 
\[
e^{-\frac{(\ry- \sqrt{2}\ell_1)^2}{2\ell_1}} \leq e^{-9\ell_1}\,, \qquad \P_{r_2, \ell_1}^{\y(w)}(\rho < \ell_1) \leq 1\,,
\]
which yields 
\begin{align}
    \frac{ze^{5\ell_1}}{t} \int_{\cI_4} \mathfrak{J}_2 ~\d\ry \leq \frac{z}{t}e^{-4\ell_1} \,.
    \label{eqn:claim-barrier-2:I2-1}
\end{align}
On $\cI_5$, we find exponential decay of the hitting probability as in $\cI_2$. Consider the hitting times
\begin{align*}
    \overline{\rho}_2 &:= \inf \{ s: W_s > 
    \frac{m_t}{t}(s+\tl-\ell-\ell_1)+y \} \,, \text{ and}
    \\
    \underline{\rho}_2 &:= \inf \{ s: W_s < Q_z(\tll-\ell_1) \} \,.
\end{align*}
A union bound gives
\begin{align}
    \P_{r_2,\ell_1}^{\y(w)} ( \rho < \ell_1)
    &\leq 
    \P_{r_2,\ell_1} \pth{ \overline{\rho}_2 < \ell_1  }
    + 
    \P_{r_2, \ell_1} \pth{ \underline{\rho}_2 < \ell_1 } \,.
    \label{eqn:claim-barrier-2:hit2-union}
\end{align}
From~\eqref{eqn:brownian-ballot-explicit}, we find
\begin{align}
    \P_{r_2,\ell_1}^{\y(w)} \pth{ \overline{\rho}_2 < \ell_1  }
    &=
    \exp \bigg( -\frac{2\pth{w- \frac{m_t}{t}\ell_1 - \ry}w}{\ell_1}\bigg) \nonumber \\
    &\leq 
    \exp \bigg( -\frac{2\ell^{1/3}\pth{w- \frac{m_t}{t}\ell_1 - \ry}}{\ell_1}\bigg)
    \label{eqn:claim-barrier-2:hit2-1}
    \\
    \P_{r_2, \ell_1}^{\y(w)} \pth{ \underline{\rho}_2 < \ell_1 } 
    &= 
    \exp \bigg( - \frac{2 \pth{\ry- w+2\ell^{2/3}}(2\ell^{2/3}+w)}{\ell_1}\bigg) \nonumber 
    \\
    &\leq 
    \exp \bigg( - \frac{4\ell^{2/3} \pth{\ry- w+2\ell^{2/3}}}{\ell_1}\bigg) \,.
    \label{eqn:claim-barrier-2:hit2-2}
\end{align}
On $\cI_5$, $\ry \uasymp \ell_1$, so that the right-hand side of~\eqref{eqn:claim-barrier-2:hit2-1} dominates the right-hand side of~\eqref{eqn:claim-barrier-2:hit2-2}. It also follows $w- \frac{m_t}{t}\ell_1 - \ry \geq w/2$ for all $\ry \in \cI_5$, $w \in[\ell^{1/3}, \ell^{2/3}]$, and~$L$ and $t$ sufficiently large. Thus, 
\[
    \P_{r_2,\ell_1}^{\y(w)} ( \rho < \ell_1) \usim \P_{r_2,\ell_1}^{\y(w)} \pth{ \overline{\rho}_2 < \ell_1  } \ulesssim 
    e^{- \ell^{1/3}w/\ell_1} \ulesssim e^{-\ell^{1/3}}\,.
\]
From the previous display and bounding the Gaussian term in $\mathfrak{J}_2$ by $1$, we have
\begin{align}
    \frac{e^{5\ell_1}}{t} \int_{\cI_5} \mathfrak{I}_2 \d\ry \ulesssim 
    \frac{z}{t}e^{5\ell_1-\ell^{1/3}} \ulesssim
    \frac{z}{t}e^{-\ell_1}\,.
        \label{eqn:claim-barrier-2:I2-2}
\end{align}
Together, \eqref{eqn:claim-barrier-2:I2-1} and~\eqref{eqn:claim-barrier-2:I2-2} imply that
\begin{align}
    \cP_2^* \ulesssim \frac{ze^{5\ell_1}}{t} \int_{\cI_4} \mathfrak{I}_2 \d\ry \ulesssim  \frac{z}{t}e^{-\ell_1}\,.
    \label{eqn:claim-barrier-2:p2}
\end{align}

Now, from~\eqref{eqn:claim-barrier-2:bad.1}, \eqref{eqn:claim-barrier-2:p*}, \eqref{eqn:claim-barrier-2:p1}, and~\eqref{eqn:claim-barrier-2:p2}, we find that
\begin{align}
    \sup_{w \in [\ell^{1/3}, \ell^{2/3}]} \cP \ulesssim \cP_2^* \ulesssim \frac{z}{t} e^{-\ell_1} \,.
\end{align}
It follows from~\eqref{eqn:claim-barrier-2:good} that
\begin{align}
    \sup_{w \in [\ell^{1/3}, \ell^{2/3}]} \frac{\cP}{\P_{\x(z),\tl-\ell}^{ \y(w) } \pth{\B_{[0,\tl-\ell]}(W_\cdot)}} = o_u(1)\,,
\end{align}
where we recall the $o_u$ notation from Section~\ref{subsec:asymp-notation}. Equation \eqref{eqn:claim-bridge-barrier-2} follows.
\end{proof}

\subsection{Proof of Lemma~\ref{lem:barrier-equiv}}
\label{subsec:proof_bananaequiv}
We begin with the many-to-one lemma: for any branching Bessel particle $v \in \cN_{\tl -\ell}$, we have 
\begin{align}
    \E_{\x(z)}[\Gamma_{L,t}] 
    &=
    e^{\tl-\ell} \P_{\sqrt{2}L -z}
    \pth{
    F_{L,t}(v)
    }
    \,.
    \label{eqn:bar-equiv:m21}
\end{align}
For brevity, we will write 
\[
    \bar{F}_{L,t}(v) := 
    \UB{[0,\tl-\ell]}{B_0}(R_\cdot^{(v)})
    \cap
    \Big \{R^{(v)}_{\tll} > \frac{t}{\sqrt{d}} \Big \} \,,
\]
so that $F_{L,t}(v) = \bar{F}_{L,t}(v) \cap \fT_{\ell}(v)$. 
Using the tower property of conditional expectation and the Markov property of the Bessel process, we may rewrite the probability on the right-hand side of~\eqref{eqn:bar-equiv:m21} as 
\begin{align}
    \E_{\sqrt{2}L -z} \brak{
    \one_{\bar{F}_{L,t}(v)}
    \E \brak{ \one_{\fT_{\ell}(v)} \given \cF_{\tl-\ell} }
    } 
    =
    \E_{\sqrt{2}L -z} \brak{
    \one_{\bar{F}_{L,t}(v)}
    \P_{R_{\tll}^{v}}(R_{\ell}^* > m_t+y)
    }\,,
    \label{eqn:bar-equiv:markov}
\end{align}
where $R_r^*$ denotes an independent $d$-dimensional branching Bessel process. We next apply the Girsanov transform~\eqref{eq.girsanov}.
Define the event 
\[
    \widehat{F}_{L,t} := 
    \big \{W_s > 0, ~\forall s\in [0,\tll] \big\} \cap 
    \UB{[0,\tl-\ell]}{B_0}(W_\cdot) \cap
    \Big \{W_{\tll} > \frac{t}{\sqrt{d}} \Big \} \,.
\]
Then substituting \eqref{eqn:bar-equiv:markov} into \eqref{eqn:bar-equiv:m21} and then applying the Girsanov transform gives
\begin{align}
    &\E_{\x(z)}[\Gamma_{L,t}] \nonumber\\
    &=e^{\tl-\ell} 
    \E_{\x(z)} \brak{
    \pth{\frac{W_{\tl-\ell}}{\x(z)}}^{\alpha_d}
    \exp \pth{\int_0^{\tl-\ell} \frac{\alpha_d- \alpha_d^2}{W_s^2}\d s}
    \one_{\widehat{F}_{L,t}} 
    \P_{W_{\tll}}(R_{\ell}^* > m_t+y)
    }
    \, . \label{eq.lowerbanana.GirBes}
\end{align}
Since $d\geq 3$, we have $\alpha_d - \alpha_d^2 \leq 0$; thus, the $\exp(\cdot)$ term in the previous display is bounded above by $1$. We also use the upper bound $\one_{\{W_s > 0 , s \in [0,\tl -\ell]\}} \leq \one_{\{W_{\tl-\ell} > 0\}}$. 
Then, from the right-hand side of 
\eqref{eq.lowerbanana.GirBes}, we obtain the following expression by 
integrating over $w:= \y( W_{\tl-\ell} ) \in [-\log \ell, \y(\frac{t}{\sqrt{d}})]$ 
(recall the definition of $\y(\cdot)$ from~\eqref{def:xz-yw}: $w$ is the distance of $W_{\tll}$ below $\frac{m_t}{t}(t-\ell)+y$) and applying the Markov property at time $\tl-\ell$:
\begin{align}
   \E_{\x(z)}[\Gamma_{L,t}]
    &\leq
    e^{\tl - \ell} 
    \int_{- \log \ell}^{\y(\frac{t}{\sqrt{d}})}
    p_{\tl-\ell}^W(\x(z), \y(w))
    \bigg ( \frac{ \y(w)}{\x(z)} \bigg)^{\alpha_d} 
    \P_{\x(z)}^{ \y(w) } 
    \pth{
        \UB{[0,\tl-\ell]}{B_0}(W_\cdot)
    }
    \nonumber \\
    & \qquad \times \P_{\y(w)} \pth{ R_\ell^* > m_t+y
    }\d w 
   =: e^{\tl-\ell} \int_{-\log \ell}^{ \y(\frac{t}{\sqrt{d}})} \Phi_{t,L,z}(w) \d w \,.
    \label{eq.lowerbanana.upperbd1}
\end{align}
Thus, to prove \eqref{eq.1stmom.bananaequiv}, it suffices to show 
\begin{align}
    e^{\tl-\ell} \int_{-\log \ell}^{ \y(\frac{t}{\sqrt{d}})} \Phi_{t,L,z}(w) \d w
    \usim
    \E_{\x(z)} \brak{\bar{\Lambda}_{L,t}(z)}
    \,.
    \label{eq.lowerbanana.up=low}
\end{align}
We proceed with Claims~\ref{claim:muchlower_window_t-L} and~\ref{claim:restofwindow_t-L}, which show that the integral of $\Phi$ over  $w\not \in [\ell^{1/3}, \ell^{2/3}]$  in equation~\eqref{eq.lowerbanana.upperbd1} is negligible compared to the integral over  $w \in [\ell^{1/3}, \ell^{2/3}]$.

\begin{claim}\label{claim:muchlower_window_t-L}
Recall $\fM_{L,z}$ from~\eqref{def:M_Lz}.
We have
\begin{align}
    \frac{
    e^{\tl-\ell} \int_{\ell^{2/3}}^{ \y(\frac{t}{\sqrt{d}})} \Phi_{t,L,z}(w) \d w}{{\fM_{L,z}}} = o_u(1) \,.
    \label{eq.muchlower_window_t-L}
\end{align}
\end{claim}

\begin{proof}[Proof of Claim~\ref{claim:muchlower_window_t-L}]
We first handle the integral of $\Phi_{t,L,z}(w)$ over $w$ in the interval $[t^{1/3}, \y(\frac{t}{\sqrt{d}})]$. 
For this integral, we may take cruder, simpler estimates.

Expanding the Gaussian term in~\eqref{eq.lowerbanana.upperbd1} gives
\begin{align}
    p_{\tl-\ell}^W \pth{\x(z), \y(w)}
    &\usim
    \frac{1}{\sqrt{2 \pi}} 
    t^{-\alpha_d + 1} 
    e^{-\tl + \ell}  
    e^{-(z+y) \sqrt{2}} 
    e^{w \sqrt{2}} 
    e^{- \frac{w^2}{2(\tl - \ell)}} 
    e^{ \frac{(y+z + \fc_d \log t)w}{\tl - \ell}} 
    =: \Upsilon_g
    \, ,
    \label{eq.1stmom.gaussianexact}
\end{align}
where we recall the definition of $\fc_d$ from~\eqref{def:m_t-c_d}.
From the bounds $\alpha_d \geq 0$, $w \leq 2(\tll)$ for all $t$ sufficiently larger than $\ell$, and $e^{-a} \leq 1$ for all $a \geq 0$, we have
\begin{align}
    \Upsilon_g
    &\ulesssim 
    e^{2z}t^{1+ 2\fc_d} e^{- \tl + \ell} e^{w \sqrt{2}} \, .
    \label{eqn:1stmom-claim1-gauss}
\end{align}
Next, from~\eqref{eqn:tail-bd-3sqrt2}, we find
\begin{align}
    \P_{\y(w)} ( R_{\ell}^* > m_t+y)  \leq C_d e^{- \sqrt{2} \fo_t \ell } e^{-\frac{w^2}{2\ell}} e^{- \frac{m_t}{t} w} 
    \label{eqn:1stmom-claim1-t13-smalltail}
\end{align}
(recall $\fo_t := \frac{m_t}{t}- \sqrt{2}$).
Finally, substituting the bounds in~\eqref{eqn:1stmom-claim1-gauss} and~\eqref{eqn:1stmom-claim1-t13-smalltail} into \eqref{eq.lowerbanana.upperbd1}
and upper bounding the barrier probability $\P_{\x(z)}^{\y(w)}\big(\UB{[0,\tll]}{B_0}(W_\cdot)\big)$ by $1$ yields the following: 

\begin{align*}
    \frac{ 
    e^{\tl-\ell} \int_{t^{1/3}}^{ \y(\frac{t}{\sqrt{d}})} \Phi_{t,L,z}(w) \d w}{\fM_{L,z}}
    &\ulesssim 
    e^{4L^{2/3}} t^{1+2\fc_d}e^{-t+\ell} \int_{t^{1/3}}^{ \y(\frac{t}{\sqrt{d}})} e^{-\fo_t w} e^{- \frac{w^2}{2\ell}} \d w   \nonumber \\
    &\ulesssim t^{\delta} e^{-t+\ell -\frac{t^{2/3}}{2\ell}} = o_u(1)
\end{align*}
where $\delta:= \delta(d)$ denotes a positive constant depending only on $d$. 

We now  handle the integral over $w \in [\ell^{2/3}, t^{1/3}]$. 
For $w \in [-  \log \ell, t^{1/3}]$, the terms in $\Phi_{t, L,z}(w)$ simplify greatly (we consider this larger interval of $w$ than what is addressed in the claim because the estimates that follow will be used again later).
From~\eqref{eq.1stmom.gaussianexact}, we have that, uniformly over $w \in [-\log \ell, t^{1/3}]$,
        \begin{align}
            p_{\tl-\ell}^W \pth{\x(z), \y(w)}
            &\usim
            \frac{1}{\sqrt{2\pi}}
            e^{-(\tl-\ell)} (t-\ell)^{-\alpha_d+ 1} e^{-(z+y)\sqrt{2}} e^{w \sqrt{2}} \, .
            \label{eq.lowerbanana.gaussianest}
        \end{align}
We also have
\begin{align}
    \y(w) \usim t \sqrt{2} \, ,
    \label{eq.lowerbanana.girsanovest}
\end{align}
uniformly over $w \in [-\log \ell, t^{1/3}]$.
From~\eqref{eqn:brz2.1-bridge}, the following estimate holds uniformly over $w \in [-\log \ell, t^{1/3}]$:
\begin{align}
    \P_{\x(z)}^{ \y(w) } 
    \pth{
        \UB{[0,\tl-\ell]}{B_0}(W_\cdot)
    }
    &\leq 
            \P_{\x(z)}^{ \y(w) } 
    \pth{
        \UB{[1,\tl-\ell-1]}{B_0}(W_\cdot)
    } \nonumber \\
    &\uasymp
            \P_{\x(z)}^{ \y(w) } 
    \pth{
        \UB{[1,\tl-\ell-1]}{ \frac{m_t}{t}(\cdot + L) +y +  \log \ell}(W_\cdot)
    }
            \uasymp 2\frac{z(1+w+  \log \ell)}{\tl-\ell} \,.
            \label{eq.lowerbanana.barrierest}
\end{align}
Now, substituting~\eqref{eq.lowerbanana.gaussianest},\eqref{eq.lowerbanana.girsanovest}, and~\eqref{eq.lowerbanana.barrierest}
        into~\eqref{eq.lowerbanana.upperbd1} yields 
        \begin{align}
            \frac{
                e^{\tl -\ell}
                \Phi_{t, L, z}(w)
            }
            {\fM_{L,z}}
            &\uasymp (1+w+  \log \ell) e^{w \sqrt{2}} 
            \P_{\y(w)}\pth{ R_{\ell}^* > m_t+y } 
            \,,
            \label{eqn:claim-barrier-equiv:intsimp-R}
        \end{align}
    %
uniformly over $w \in [-\log \ell, t^{1/3}]$. 
Now, note that uniformly over $w \in [\ell^{2/3}, t^{1/3}]$, we have
$(1+w+\log \ell) \usim w$ and $e^{-\fo_t w} \usim 1$. It then follows from~\eqref{eqn:1stmom-claim1-t13-smalltail} and the previous display that, uniformly over $w \in [-\log \ell, t^{1/3}]$, we have 
\begin{align*}
            \frac{
                e^{\tl -\ell}
                \Phi_{t, L, z}(w)
            }
            {\fM_{L,z}}
    &\ulesssim we^{-\frac{w^2}{2\ell}} \,.
\end{align*}
Thus,
\begin{align*}
   \frac{  \int_{\ell^{2/3}}^{t^{1/3}} e^{\tl -\ell} \Phi_{t, L, z}(w) \d w}{\fM_{L,z}}  \ulesssim \int_{\ell^{2/3}}^{t^{1/3}} we^{-\frac{w^2}{2\ell}} \d w \ulesssim e^{-\ell^{1/3}/2} = o_u(1) \,.
\end{align*}
This concludes the proof of the claim.
\end{proof}

\begin{claim}\label{claim:restofwindow_t-L}
We have 
\begin{align}
    \frac{e^{\tl-\ell} \int_{-\log \ell}^{\ell^{1/3}} \Phi_{t,L,z}(w) \d w}{\fM_{L,z}} = o_u(1) \, .
    \label{eqn:claim-restofwindow}
\end{align}
On the other hand,
\begin{align}
    \frac{e^{\tl-\ell} \int_{\ell^{1/3}}^{\ell^{2/3}} \Phi_{t,L,z}(w)\d w}{\fM_{L,z}}
    \usim
    \frac{2^{\frac{1+\alpha_d}{2}}}{\sqrt{\pi}}
    \int_{\ell^{1/3}}^{\ell^{2/3}} w e^{w \sqrt{2}} 
            \P\big( W_{\ell}^* > \sqrt{2}\ell + w \big) \d w
    \uasymp
    1 \,.
    \label{eq.main_window_t-L}
\end{align}
\end{claim}

\begin{proof}[Proof of Claim~\ref{claim:restofwindow_t-L}]
Recall \eqref{eqn:claim-barrier-equiv:intsimp-R}, and note that, uniformly over $w \in [-\log \ell, \ell^{2/3}]$, Corollary~\ref{cor:F_ell_1d} applies. Then substitute~\eqref{eqn:tailequiv} into \eqref{eqn:claim-barrier-equiv:intsimp-R}  to find that, uniformly over  $w \in [-\log \ell, \ell^{2/3}]$, we have 
\begin{align*}
            \frac{
                e^{\tl -\ell}
                \Phi_{t, L, z}(w)
            }
            {\fM_{L,z}}
            &\ \uasymp \
            (1+w+  \log \ell) e^{w \sqrt{2}} 
            \P\pth{ W_{\ell}^* > \frac{m_t}{t}\ell + w } \,.
        \end{align*}
%
Now, for each fixed $\ell$, the continuity of the F-KPP equation (see Section~\ref{subsec:bbm-est}) gives 
\[\abs{\P(W_{\ell}^* > \frac{m_t}{t}\ell+w) - \P(W_{\ell}^* > \sqrt{2}\ell+w)} = o(1)\,,\]
and so the last display gives 
\begin{align}
    \frac{
                e^{\tl -\ell}
                \Phi_{t, L, z}(w)
            }
            {\fM_{L,z}}
            &\uasymp
            (1+w+  \log \ell) e^{w \sqrt{2}} 
            \P\big( W_{\ell}^* > \sqrt{2}\ell + w \big) 
            +o_u(1)
            \,,
    \label{eq.lowerbanana.intsimp1}
\end{align}
uniformly over  $w \in [-\log \ell, \ell^{2/3}]$, where the $o_u(1)$ term may be bounded in absolute value by $(\ell^{2/3} + \log \ell)^2 e^{ \ell^{2/3} \sqrt{2}} \ell \fo_t$.
Now, integrating both sides of~\eqref{eq.lowerbanana.intsimp1} over $w\in [-\log \ell, \ell^{1/3}]$ and then applying the bound on $\P(W_{\ell^*}> \sqrt{2}\ell +w)$  from~\eqref{eqn:bram8.2-cor} yields 
\begin{align*}
    \frac{\int_{- \log \ell}^{\ell^{1/3}} 
                e^{\tl -\ell}
                \Phi_{t, L, z}(w) \d w
            }
            {\fM_{L,z}}
    &\ulesssim 
    \ell^{-\frac{3}{2}} \int_{-\log \ell}^{\ell^{1/3}} (1+w+\log \ell)(w+ \frac{3}{2\sqrt{2}}\log \ell) \d w 
    \ulesssim 
    \ell^{-\frac{1}{2}} \,,
\end{align*}
from which~\eqref{eqn:claim-restofwindow} follows.

We next consider the range $w \in [\ell^{1/3}, \ell^{2/3}]$. As a consequence of Lemma~\ref{lem:bridge-barrier-equiv} and the Brownian ballot estimate~\eqref{eqn:brownian-ballot-sim}, we have the following relation uniformly over $w \in [\ell^{1/3}, \ell^{2/3}]$:
\begin{align}
    \P_{\x(z), \tl-\ell}^{\y(w)}
            \pth{
            \UB{[0,\tl-\ell]}{B_0 }(W_\cdot)
            }  \usim \frac{2zw }{\tl-\ell}\,.
    \label{eqn:barrier-equiv:exact-barrier-est}
\end{align}
Over this range of $w$, \eqref{eqn:barrier-equiv:exact-barrier-est} replaces \eqref{eq.lowerbanana.barrierest}, which is a less precise estimate, and will therefore~\eqref{eqn:barrier-equiv:exact-barrier-est} gives a more precise version of~\eqref{eqn:claim-barrier-equiv:intsimp-R}. Indeed, substituting \eqref{eqn:tailequiv},
\eqref{eq.lowerbanana.gaussianest},
\eqref{eq.lowerbanana.girsanovest}, and 
\eqref{eqn:barrier-equiv:exact-barrier-est} into the definition of $\Phi_{t, L,z}(w)$ \eqref{eq.lowerbanana.upperbd1} yields 
\begin{align}
    \frac{e^{\tl -\ell}\Phi_{t, L, z}(w)}
    {\fM_{L,z}} 
    \usim 
    \frac{2^{\frac{1+\alpha_d}{2}}}{\sqrt{\pi}} w e^{w\sqrt{2}} \P\big(W_{\ell}^* > \sqrt{2}\ell+w\big)
    \label{eqn:Phi-sim-main}
\end{align}
from which we find
\begin{align}
    \frac{e^{\tl-\ell} \int_{\ell^{1/3}}^{\ell^{2/3}} \Phi_{t,L,z}(w) \d w}{\fM_{L,z}}
    \usim
    \frac{2^{\frac{1+\alpha_d}{2}}}{\sqrt{\pi}}
    \int_{\ell^{1/3}}^{\ell^{2/3}} w e^{w \sqrt{2}} 
            \P\big( W_{\ell}^* > \sqrt{2}\ell + w \big) \d w \, .
    \label{eqn:lowerbanana-nicePhi-int}
\end{align}
This gives the left asymptotic relation in~\eqref{eq.main_window_t-L}.

It remains to show the integral on the right-hand above is asymptotically equivalent to a constant.
For the lower-bound, note that for any~$\ell \geq 1$
and~$w \in [1, \ell^{1/2} - \frac{3}{2\sqrt{2}}\log \ell]$, 
\begin{align*}
    \P\big( W_{\ell}^* > \sqrt{2}\ell + w \big)
    \geq 
    C \ell^{-3/2} \pth{w+ \frac{3}{2\sqrt{2}} \log \ell} e^{-w\sqrt{2}} \, ,
\end{align*}
where $C>0$ is a constant depending only on $d$.
This follows from writing~$\sqrt{2}\ell+w$ as ~$m_\ell(1)+w+ \frac{3}{2\sqrt{2}}\log \ell$ and applying the lower bound of the $d=1$ version of~\eqref{eq.mallein}, where in the notation of~\eqref{eq.mallein} we take $t =\ell$ and $y = w+ \frac{3}{2\sqrt{2}}\log \ell$.
It follows that, for all $\ell \geq 1$, 
\begin{align*}
    \int_{\ell^{1/3}}^{\ell^{2/3}} w e^{w \sqrt{2}} 
            \P\big( W_{\ell}^* > \sqrt{2}\ell + w \big) \d w
    &\geq C \ell^{-3/2}
    \int_{\ell^{1/3}}^{\ell^{1/2} - \frac{3}{2\sqrt{2}} \log \ell} 
    w^2 \d w > C > 0 \, ,
\end{align*}
where here and in the rest of the proof of this claim, $C$ denotes a positive constant varying from occurrence to occurrence. 
For the upper bound, we use Corollary~\ref{cor:bram8.2}, which states for all $w \in [\ell^{1/3}, \ell^{2/3}]$ and for all $\ell$ sufficiently large, we have 
\begin{align*}
    \P\big( W_{\ell}^* > \sqrt{2}\ell + w \big)
    \leq C \ell^{-\frac{3}{2}} \pth{w+ \frac{3}{2\sqrt{2}} \log \ell} e^{-w\sqrt{2}} e^{- \frac{w^2}{2\ell}} \, .
\end{align*}
It follows that for all $\ell$ sufficiently large,
\begin{align*}
\int_{\ell^{1/3}}^{\ell^{2/3}} w e^{w \sqrt{2}} 
            \P\big( W_{\ell}^* > \sqrt{2}\ell + w \big) \d w
    \leq C \ell^{- \frac{3}{2}} \int_{\ell^{1/3}}^{\ell^{2/3}} w^2 e^{- \frac{w^2}{2\ell}}
    \leq C \,.
\end{align*}
Thus, we have shown 
\begin{align}
    \int_{\ell^{1/3}}^{\ell^{2/3}} w e^{w \sqrt{2}} 
            \P\big( W_{\ell}^* > \sqrt{2}\ell + w \big) \d w
    \uasymp 1 \, .
    \label{eqn:lowerbanana-tailint-asymp}
\end{align}
Now, \eqref{eqn:lowerbanana-nicePhi-int} and~\eqref{eqn:lowerbanana-tailint-asymp} imply that 
\begin{align*}
    e^{\tl-\ell} \int_{\ell^{1/3}}^{\ell^{2/3}} \Phi_{t,L,z}(w) \d w
    \uasymp
    \fM_{L,z} \,, 
\end{align*}
which gives the right asymptotic relation in~\eqref{eq.main_window_t-L}.
\end{proof}

From Claims~\ref{claim:muchlower_window_t-L} and~\ref{claim:restofwindow_t-L}, it follows that 
\begin{align}
    & e^{\tl-\ell} \int_{- \log \ell}^{\y(\frac{t}{\sqrt{d}})} \Phi_{t,L,z}(w) \d w
    \usim
    e^{\tl-\ell} \int_{\ell^{1/3}}^{\ell^{2/3}} \Phi_{t,L,z}(w)\d w 
    \nonumber \\
    &=e^{\tl - \ell} 
    \!\int_{\ell^{1/3}}^{\ell^{2/3}}
    \!\!\!p_{\tl-\ell}^W \big (\x(z), \y(w) \big )
    \Big ( \frac{ \y(w)}{\x(z)} \Big )^{\alpha_d} 
   \P_{\x(z),\tll}^{ \y(w) } 
    \pth{
        \UB{[0,\tl-\ell]}{B_0}(W_\cdot)
    }
    \P_{\y(w)} \pth{ R_\ell^* > m_t+y
    }\d w \,.
    \label{eqn:bar-equiv:Phi-smallint}
\end{align}
A consequence of Lemma~\ref{lem:bridge-barrier-equiv} is that, uniformly over $w \in [\ell^{1/3}, \ell^{2/3}]$, we have 
\[
    \P_{\x(z),\tll}^{\y(w)} \pth{\UB{[0,\tl-\ell]}{B_0}(W_\cdot)}
    \usim \P_{\x(z),\tll}^{\y(w)} \pth{ \B_{[0,\tll]}(W_\cdot)}\,.
\]
Thus, the last three displays give 
\begin{align*}
   & e^{\tl-\ell} \int_{- \log \ell}^{\y(\frac{t}{\sqrt{d}})} 
    \Phi_{t,L,z}(w) \d w \\
    &\usim 
    e^{\tl - \ell} 
   \!\! \int_{\ell^{1/3}}^{\ell^{2/3}}
    \!\!\!p_{\tl-\ell}^W \big (\x(z), \y(w) \big )
    \bigg ( \frac{ \y(w)}{\x(z)} \bigg )^{\alpha_d} 
    \P_{\x(z),\tll}^{\y(w)} \pth{ \B_{[0,\tll]}(W_\cdot)}
    \P_{\y(w)} \pth{ R_\ell^* > m_t+y
    }\d w \,,
\end{align*}
from which, by the Markov property, we find the right-hand side equal to
\begin{align}
    e^{\tl-\ell}\E_{\x(z)} \brak{
    \bigg ( \frac{ W_{\tll}}{\x(z)} \bigg )^{\alpha_d}
    \one_{\widehat{G}_{L,t}} 
    \P_{W_{\tll}}(R_{\ell}^* > m_t+y)
    }\,,
    \label{eqn:bar-equiv:undo-gir-1}
\end{align}
where we define
\[
\widehat{G}_{L,t} := \B_{[0,\tll]}(W_\cdot) \cap \{\y(W_{\tl-\ell}) \in [\ell^{1/3}, \ell^{2/3}]\} \,.
\]
From~\eqref{eqn:trivial-exp-gir}, we have
\[
\exp \bigg( \int_{0}^{\tll} \frac{\alpha_d-\alpha_d^2}{W_s^2}\d s \bigg) \usim 1 \,.
\]
We then have from~\eqref{eqn:bar-equiv:undo-gir-1} that $e^{\tl-\ell} \int_{- \log \ell}^{\y(\frac{t}{\sqrt{d}})} \Phi_{t,L,z}(w) \d w \usim $
\begin{align*}
    e^{\tl-\ell}\E_{\x(z)} \brak{
    \bigg ( \frac{ W_{\tll} }{\x(z)} \bigg )^{\alpha_d}
    \exp \bigg( \int_{0}^{\tll} \frac{\alpha_d-\alpha_d^2}{W_s^2}\d s \bigg)
    \one_{\widehat{G}_{L,t}}
    \P_{W_{\tll}}(R_{\ell}^* > m_t+y)
    }\,.
\end{align*}
We may now apply the Girsanov transform \eqref{eq.girsanov}, as in the ``reverse" direction of~\eqref{eq.lowerbanana.GirBes}, to the above display. On the event $\B_{[0,\tll]}(W_\cdot)$, we have
$\one_{\{W_s > 0,~\forall s\in[0,\tll] \}} = 1$, and so~\eqref{eq.girsanov} and the above display gives
\begin{align}
    e^{\tl-\ell} \int_{- \log \ell}^{\y(\frac{t}{\sqrt{d}})} 
    \Phi_{t,L,z}(w) \d w  
    &\usim
    e^{\tl-\ell}\E_{\x(z)} \brak{
    \one_{\bar{G}_{L,t}(v)}
    \P_{R^{(v)}_{\tll}}(R_{\ell}^* > m_t+y)
    } \, ,
\end{align}
for any branching Bessel particle $v \in \cN_{\tll}$, where
\[
\bar{G}_{L,t}(v) := \B_{[0,\tll]}(R^{(v)}_\cdot) \cap  \big\{ \y(R^{(v)}_{\tl-\ell}) \in [\ell^{1/3}, \ell^{2/3}] \big\}\,.
\]
We then have, from the Markov property and the many-to-one lemma (similar to~\eqref{eqn:bar-equiv:markov} and~\eqref{eqn:bar-equiv:m21}, respectively)
\begin{align*}
    e^{\tl-\ell} \int_{- \log \ell}^{\y(\frac{t}{\sqrt{d}})} 
    \Phi_{t,L,z}(w) \d w  
    &\usim \E_{\x(z)}\brak{\bar{\Lambda}_{L,t}}\,,
\end{align*}
which is exactly \eqref{eq.lowerbanana.up=low}. 
Thus, we have shown \eqref{eq.1stmom.bananaequiv}. 
Equation~\eqref{eq.1stmom.finalintegral} then follows from the last display,
the first line of~\eqref{eqn:bar-equiv:Phi-smallint}, and~\eqref{eq.main_window_t-L}.
\qed

\section{Proof of Lemma~\ref{lem:2ndmom}: a second moment analysis}
\label{sec:2ndmom}
In this section, we prove Lemma~\ref{lem:2ndmom} for  all  $d \geq 2$.

 \begin{proof}[Proof of Lemma~\ref{lem:2ndmom}]
Recall \eqref{def:xz-yw}. From the many-to-two lemma \eqref{eqn:many-to-two}, we have
\begin{align}
    \E_{\x(z)}[ \bar{\Lambda}_{L,t}^2 ] = 
    \E_{\x(z)}[ \bar{\Lambda}_{L,t} ] +
    2\int_{0}^{\tl-\ell} e^{\tl-\ell+s} 
    \P_{\x(z)} \big(    G_{L,t}(v) \cap G_{L,t}(v')  \big) \d s \, , 
\end{align}
where, for each $s$ on the right-hand side, $v$ and $v'$ denote  two particles in $\cN_{\tll}$ that split  at time $j_s:= \tl-\ell-s$.
In light of Lemma~\ref{lem:barrier-equiv}, it suffices to show that
\begin{align}
    \frac{\int_{0}^{\tl-\ell} e^{\tl-\ell+s} 
    \P_{\x(z)} \big(   G_{L,t}(v) \cap G_{L,t}(v') \big ) \d s}
    {
        \fM_{L,z}
    } 
    = o_{u}(1) \, .
    \label{eq.2ndmom.goal}
\end{align}
Fix $s \in (0,\tl-\ell)$ and fix two particles $v$ and $v' \in \cN_{\tll}$ that split from a single particle at time $j_s$. Then 
\begin{align}
    \P_{\x(z)}
    \pth{ G_{L,t}(v) \cap G_{L,t}(v') }
    &=  
    \E_{\x(z)}
    \bigg[
        \prod_{w \in \{v,v'\}}
        \one_{ 
        \B_{[0,\tl- \ell]}(R^{(u)}) \cap 
        \{ \y(R^{(w)}_{\tll}) \in [\ell^{1/3}, \ell^{2/3}] \} \cap 
        \mathfrak{T}_{\ell}(u))   
        }
    \bigg]
    \nonumber \\
    &= 
    \E_{\x(z)}
    \bigg[
    \one_{ 
    \B_{[0,j_s]}(R^{(v)}_\cdot)
    }
    \E 
    \Big[
        \one_{
        \bB_{[j_s, \tll]}(R_.^{(v)}) \cap
        \mathfrak{T}_{\ell}(v)
        }
        \given R^{(v)}_{j_s}
    \Big]^2
    \bigg]\,,
    \label{eqn:2mom:markov}
\end{align}
where
\[
    \bB_{[j_s, \tll]}(X_.)
    :=
    \B_{[j_s,\tl- \ell]}(X_\cdot) \cap \big\{ \y(X_{\tll}) \in [\ell^{1/3}, \ell^{2/3}] \big\}\,,
\]
for a real-valued process $X_. :[0,\infty) \to \R$.
We now apply the Girsanov transform to~\eqref{eqn:2mom:markov} to replace the Bessel process $(R_r^{(v)})_{r \in[0,j_s]}$ with a Brownian motion $(W_r)_{r\in[0,j_s]}$. Note that the indicator function in 
\eqref{eq.girsanov} is equal to $1$ on the event 
$\B_{[0,j_s]}(R^{(v)}_\cdot)$; further, the exponential term in 
\eqref{eq.girsanov} is bounded above by a constant depending only on $d$ for all $s>0$ and $\ell, t >0$.
Thus, we find from~\eqref{eqn:2mom:markov} that 
\begin{align}
\P_{\x(z)}
    \pth{ G_{L,t}(v) \cap G_{L,t}(v') } 
     \ulesssim \E_{\x(z)} \brak{
        \bigg( \frac{W_{j_s}}{\x(z)} \bigg)^{\alpha_d}
        \one_{ \B_{[0,j_s]}(W_\cdot) }
        \E \brak{\mathfrak{g}(W_{j_s})}^2
        }\,,
    \label{eqn:2mom:gir1}
\end{align}
uniformly over $s$, 
where
\[
    \mathfrak{g}(W_{j_s}) := 
\E \brak{ \one_{ \left \{
        \bB_{[j_s,\tl- \ell]}(R^{(v)}_\cdot) \cap 
        \mathfrak{T}_{\ell}(v)
        \right \} }
        \given R^{(v)}_{j_s} = W_{j_s}
    }\,.
\]
Applying the Markov property at time $\tll$ allows us to express $\mathfrak{g}(W_{j_s})$ as 
\begin{align*}
\E\brak{
        \one_{
        \bB_{[j_s,\tl- \ell]}(R^{(v)}_\cdot)
        }
        \P_{R^{(v)}_{\tll}}(R_\ell^*>m_t+y)
        \given R^{(v)}_{j_s} = W_{j_s}
    }\,.
\end{align*}
As before, we apply the Girsanov transform to replace the Bessel process~$(R^{(v)}_r)_{r\in[j_s, \tll]}$ with a Brownian motion $(\tilde{W}_r)_{r \in[j_s, \tll]}$ satisfying $\tilde{W}_{j_s} = W_{j_s}$. By the Markov property, the process $(\tilde{W}_r)_{r \in[j_s, \tll]}$ with $\tilde{W}_{j_s} = W_{j_s}$ is stochastically equivalent to~$(W_r)_{r \in[j_s, \tll]}$, and so the Girsanov transform \eqref{eq.girsanov} yields
\begin{align*}
    \mathfrak{g}(W_{j_s}) \usim
    \E \brak{
        \bigg( \frac{W_{\tl - \ell}}{W_{j_s}} \bigg)^{\alpha_d} 
        \one_{
        \bB_{[j_s, \tl- \ell]}(W_\cdot) }
        \P_{W_{\tll}}(R_{\ell}^*>m_t+y)
        \given W_{j_s} }\,,
\end{align*}
uniformly over $s \geq 0$ and $W_{j_s}$, where, as before, we have used that on $\bB_{[j_s, \tll]}(W_.)$, the indicator function and the exponential factor in~\eqref{eq.girsanov} are both $\usim 1$ (see~\eqref{eqn:trivial-exp-gir}).
From the last display and~\eqref{eqn:2mom:gir1}, we have 
\begin{align}
    \P_{\x(z)}
    \pth{ G_{L,t}(v) \cap G_{L,t}(v') } 
    \ulesssim \
    \E_{\x(z)}[\Upsilon_1(s)] \,,
    \label{eqn:2mom:gir}
\end{align}
where 
\begin{align*}
    \Upsilon_1(s) &:=
        \big( \x(z)W_{j_s} \big)^{-\alpha_d}
        \one_{ \B_{[0,j_s]}(W_\cdot) }
        \E \brak{
        W_{\tll}^{\alpha_d}
        \one_{ \bB_{[j_s, \tl- \ell]}(W_\cdot) }
        \P_{W_{\tll}}(R_{\ell}^*>m_t+y)
        \given  W_{j_s} }^2 \,.
\end{align*}
Define $\y_s(w) := \frac{m_t}t(j_s+L)+y-w$ for $w \in \R$. Recall
\eqref{eq.lowerbanana.Qdefn}. Then
expanding the right-hand side of~\eqref{eqn:2mom:gir} by integrating over the position of $w_s := \y_s(W_{j_s})$ gives 
\begin{align}
    \E_{\x(z)} \brak{ \Upsilon_1(s) } =
    \int_{0}^{\y_s(Q_z(j_s))}
    E_s^{(1)}(w_s)
    E_s^{(2)}(w_s)^2 \d w_s \,,
    \label{eq.2ndmom.PQintegral}
\end{align}
where 
\begin{align}
    E_s^{(1)}(w_s) &:=
    p_{j_s}^W \big(\x(z), \y_s(w_s) \big)
    \big( \x(z) \y_s(w_s) \big)^{-\alpha_d}
    \P_{\x(z),j_s}^{\y_s(w_s)}
    \big( \B_{[0,j_s]}(W_\cdot) \big)
    \label{eq.2ndmom.Pdef}
\end{align}
and
\begin{align}
    E_s^{(2)}(w_s) &:= 
    \E \brak{
        W_{\tll}^{\alpha_d}
        \one_{ \bB_{[j_s, \tl- \ell]}(W_\cdot) }
        \P_{W_{\tll}}(R_{\ell}^*>m_t+y)
        \given  W_{j_s} = \y_s(w_s)}^2 \,.
    \label{eq.2ndmom.Qdef}
\end{align}
We will estimate $\E_{\x(z)}[\Upsilon_1(s)]$ separately on three different intervals of $s$: $(0,\ell^{2/3}]$, $[\ell^{2/3},\tl-\ell-\ell^{2/3}]$, and~$[\tl-\ell-\ell^{2/3}, \tll)$. Before specifying any particular range of $s$, we provide some general estimates.

Note that for $w_s$ in the range of integration in~\eqref{eq.2ndmom.Pdef}, we have 
\begin{align}
    \y_s(w_s) \uasymp j_s+L \,, \label{eqn:2mom:yws}
\end{align}
uniformly over $s$ and $w_s$ (in what follows, we write ``uniformly over $w_s$" to imply ``uniformly over $w_s \in [0, \y_s\big(Q_z(j_s)\big)]$).
We can bound $E_s^{(1)}(w_s)$ from above by replacing the barrier event with the strictly larger event~$\UB{[0,j_s]}{\frac{m_t}{t}(\cdot+L)+y}(W_\cdot)$. Uniformly over $s$ and  $w_s$, we have from~\eqref{eqn:brownian-ballot-explicit} that
\begin{align}
    \P_{\x(z),j_s}^{\y_s(w_s)}
    \big( \B_{[0,j_s]}(W_\cdot) \big) \ulesssim \frac{w_s(z+y)}{j_s}\,.
    \label{eqn:2mom:E1-bridge}
\end{align}
We may then expand the Gaussian density in the right-hand side of~\eqref{eq.2ndmom.Pdef} and apply the estimates \eqref{eqn:2mom:yws} and~\eqref{eqn:2mom:E1-bridge} to find
\begin{align}
    E_s^{(1)}(w_s) \ulesssim 
    \fM_{L,z}
    (j_s+L)^{-\alpha_d} j_s^{-\frac{3}{2}}
    e^{-\frac{1}{2} \pth{\frac{m_t}{t}}^2 j_s}
    w_s e^{w_s \sqrt{2}}
    e^{- \frac{(w_s - (z+y))^2}{2j_s}}
    \, ,
    \label{eq.2ndmom.Eexpmiddle}
\end{align} 
uniformly over $s$ and $w_s$.
Towards an upper bound on $E_s^{(2)}(w_s)$, we begin with the inequality
\[
\one_{\bB_{[j_s,\tll]}(W_.)} \leq
\one_{B^{\frac{m_t}{t}(\cdot+L)+y}_{[j_s, \tl-\ell]}(W_\cdot) \cap \{ \y(W_{\tll})\in [\ell^{1/3}, \ell^{2/3}] \}} \,.
\]
Integrating over the position of $w_{\ell} := \y(W_{\tl- \ell})$ and applying the Markov property at time $j_s$ yields that $E_s^{(2)}(w_s) $ is at most
\begin{align}
  \label{eqn:2mom:Es2-int}
    \int_{\ell^{1/3}}^{\ell^{2/3}}  
    \!\!\!p_s^W\big(\y_s(w_s), \y(w_\ell) \big)
    \y(w_{\ell})^{\alpha_d}
    \P_{ \y_s(w_s),s }^{\y(w_{\ell})} 
    \Big( \UB{[0, s]}{\frac{m_t}{t}(\cdot+L+j_s)+y}(W_\cdot) \Big)
    \P_{\y(w_{\ell})}(R_{\ell}^*>m_t+y) \d w_{\ell} \,.
   \end{align}
In what follows, we write ``uniformly over $w_{\ell}$" to mean
``uniformly over $w_{\ell} \in [\ell^{1/3}, \ell^{2/3}]$." 
From~\eqref{eqn:brownian-ballot-explicit}, we have
\begin{align}
    \P_{ \y_s(w_s),s }^{\y(w_{\ell})} 
    \Big( \UB{[0, s]}{\frac{m_t}{t}(\cdot+L+j_s)+y}(W_\cdot) \Big) \ulesssim \frac{w_s w_{\ell}}{s} \,,
    \label{eqn:2mom:Es2-bridge}
\end{align}
uniformly over $s$, $w_{\ell}$, and $w_s$.
From Corollary~\ref{cor:F_ell_1d}, we have
\[
\P_{\y(w_{\ell})}(R_{\ell}^*>m_t+y) \usim \P(W_{\ell}^* > \frac{m_t}{t}\ell+ w_{\ell})\,,
\]
uniformly over $w_{\ell}$, and 
thus
\begin{align}
\P_{\y(w_{\ell})}(R_{\ell}^*>m_t+y) \usim \P(W_{\ell}^* > \sqrt{2}\ell+ w_{\ell})\,, \label{eqn:2mom:Es2-tail}
\end{align}
uniformly over $w_{\ell}$. 
By plugging~\eqref{eqn:brownian-ballot-explicit} and~\eqref{eqn:2mom:Es2-tail} in~\eqref{eqn:2mom:Es2-int}, along with the estimate $\y(w_{\ell}) \uasymp t$, which of course holds uniformly over~$w_{\ell}$, we can infer that
\begin{align}
    E_s^{(2)}(w_s)
    &\ulesssim 
    \frac{t^{\alpha_d} w_s }{s}
    \int_{\ell^{1/3}}^{\ell^{2/3}}  
    w_{\ell}
    p_s^W\big(\y_s(w_s), \y(w_\ell) \big) 
    \P(W^*_\ell > \sqrt{2} \ell +w_{\ell}) \d w_{\ell}
     \, ,
    \label{eq.2ndmom.Qintegral}
\end{align}
uniformly over $s$ and $w_s$. 
Note $w_{\ell} \uasymp w_{\ell} + \frac{3}{2\sqrt{2}} \log \ell$ uniformly over $w_{\ell}$, and thus expanding the Gaussian density in equation~\eqref{eq.2ndmom.Qintegral}  and substituting equation~\eqref{eqn:bram8.2-cor} yield
\begin{align}
    E_s^{(2)}(w_s)
    &\ulesssim
    \frac{e^{- \frac{1}{2}(\frac{m_t}{t})^2s}}{\ell^{\frac{3}{2}} s^{3/2} }
    t^{\alpha_d}
    w_s
    e^{-w_s \sqrt{2}}
    \int_{\ell^{1/3}}^{\ell^{2/3}}  
    w_{\ell}^2 
    e^{ - \frac{w_{\ell}^2}{2\ell}}
    e^{- \frac{(w_s- w_{\ell})^2}{2s}}
    \d w_{\ell} \, 
    \label{eq.2ndmom.Qexp} 
    \\
    &
    \ulesssim 
    \frac{e^{- \frac{1}{2}(\frac{m_t}{t})^2s}}{s^{3/2}}
    t^{\alpha_d}
    w_s
    e^{-w_s \sqrt{2}} \, ,
    \label{eq.2ndmom.Qexp_nointegral} 
\end{align}
uniformly over $s$ and $w_s$, where in the last line we have applied the trivial upper bound $e^{- \frac{(w_s-w_{\ell})^2}{2s}} \leq 1$ and then approximated (in the sense of $\uasymp$) the resulting Gaussian integral as $\ell^{3/2}$.
Now, bound the right-hand side of~\eqref{eq.2ndmom.Eexpmiddle} from above by taking the trivial upper bound $e^{- \frac{(w_s - (z+y))^2}{2j_s}} \leq 1$, and then substitute this bound  and~\eqref{eq.2ndmom.Qexp_nointegral} into~\eqref{eq.2ndmom.PQintegral}. This yields
\begin{align*}
    \E_{\x(z)}[\Upsilon_1(s)]
    &\ulesssim 
    \fM_{L,z}e^{-\frac{1}{2}\big(\frac{m_t}{t}\big)^2(\tl-\ell+s)}
    \Big( \frac{t^2}{j_s+L} \Big)^{\alpha_d} 
    (s^2j_s)^{-3/2}
    \int_0^{\y_s(Q_z(j_s))} w_s^3 e^{-w_s\sqrt{2}} \d w_s\,.
\end{align*}
Since
\begin{align}
    e^{-\frac{1}{2}\big(\frac{m_t}{t}\big)^2(\tl-\ell+s)} \usim e^{-(\tll+s)} t^{\frac{3}{2}-\alpha_d} e^{-\frac{\alpha_d s}{t}\log t} e^{\frac{3s}{2t}\log t}\,,
    \label{eqn:2mom:exponential}
\end{align}
the second-to-last display becomes
\begin{align}
\E_{\x(z)}[\Upsilon_1(s)]
    &\ulesssim 
    \fM_{L,z} e^{- (\tl - \ell + s)} 
    \frac{t^{\alpha_d}}{(j_s+L)^{\alpha_d} t^{\alpha_d s/t}}
    \frac{t^{3/2} t^{3s/(2t)}}{(j_s)^{3/2} s^{3} } 
    \int_0^{\y_s(Q_z(j_s))} 
     w_s^3 e^{-\sqrt{2} w_s} \d w_s\,  
    \nonumber \\
    &\uasymp 
    \fM_{L,z} e^{- (\tl - \ell + s)}  
    \frac{t^{3/2} t^{3s/(2t)}}{(j_s)^{3/2} s^{3} } 
    \frac{t^{\alpha_d}}{(j_s+L)^{\alpha_d} t^{\alpha_d s/t}} \, ,
    \label{eq.2ndmom.smiddle}
\end{align}
uniformly over $s$.
Note that uniformly over $s \geq 1$, we have 
\begin{align}
\frac{ t^{3s/(2t)}}{ s^{3} }   \ulesssim \frac{1}{s^{3/2}} \quad \text{and} \quad 
    \frac{t^{\alpha_d}}{(j_s+L)^{\alpha_d} t^{\alpha_d s/t}}
    \ulesssim 1 \, .
    \label{eqn:2mom:alphad-exp}
\end{align}
Applying the bounds in the last  display to~\eqref{eq.2ndmom.smiddle} yields the following, uniformly over $s\geq 1$:
\begin{align}
    \E_{\x(z)}[\Upsilon_1(s)]
    &\ulesssim 
    \fM_{L,z} e^{- (\tl - \ell + s)}  
    \frac{ t^{3/2} }{(sj_s)^{3/2} } 
    =: \cP_2(s)
    \, .
    \label{eqn:2mom:P2}
\end{align}
This estimate will be used for the second interval of $s$, that is, $[\ell^{2/3}, \tl -\ell-\ell^{2/3}]$. 

For $s \in [\tl- \ell-\ell^{2/3}, \tl - \ell)$, we may employ a simpler bound on $E_s^{(1)}(w_s)$ by upper bounding the barrier probability in~\eqref{eqn:2mom:E1-bridge} by $1$. Then, instead of~\eqref{eq.2ndmom.Eexpmiddle}, we have the following
\begin{align}
    E_s^{(1)}(w_s) 
    &\ulesssim z^{-1} \fM_{L,z} (j_s+ L)^{-\alpha_d} j_s^{-1/2} e^{- \frac{1}{2} \pth{\frac{m_t}{t}}^2 j_s } e^{w_s \sqrt{2}}  e^{- \frac{(w_s - (z+y))^2}{2j_s}} 
    \label{eq.2ndmom.Pexp_bigs}
\end{align}
uniformly over $s \in [\tl- \ell-\ell^{2/3}, \tl - \ell)$ and $w_s$.
Substituting~\eqref{eq.2ndmom.Qexp_nointegral} and~\eqref{eq.2ndmom.Pexp_bigs} into~\eqref{eq.2ndmom.PQintegral} yields
\begin{align*}
    &\E_{\x(z)}[\Upsilon_1(s)]\\
   & \ulesssim 
    \fM_{L,z}  z^{-1} e^{-\frac{1}{2}\big(\frac{m_t}{t}\big)^2(\tll +s)} \Big( \frac{t^2}{j_s+L} \Big)^{\alpha_d} s^{-3} 
        j_s^{-1/2} \int_0^{\y_s(Q_z(j_s))} 
        w_s^2 e^{ - \sqrt{2} w_s } e^{ - \frac{(w_s - (z+y))^2}{2j_s} } \d w_s 
    \,.
\end{align*}
Since $w_s^2 e^{-\sqrt{2}w_s}$ is uniformly bounded over $w_s$ and $s$, the integral over $w_s$ (including the $j_s^{-1/2}$ pre-factor) is also bounded over all $s$ by Gaussian integration. We may then apply the estimates in~\eqref{eqn:2mom:exponential} and \eqref{eqn:2mom:alphad-exp} to find
\begin{align}
    \E_{\x(z)}[\Upsilon_1(s)] &\ulesssim 
    \fM_{L,z} e^{- \tl - \ell + s} z^{-1}
    =: \cP_3(s)
    \, ,
    \label{eqn:2mom:P3}
\end{align}
uniformly over $s \in [\tl-\ell-\ell^{2/3} , \tl - \ell)$.

For $s \in (0, \ell^{2/3}]$, we employ a simpler bound on $E_s^{(2)}(w_s)$ by upper bounding the barrier probability in~\eqref{eqn:2mom:Es2-bridge} by $1$. Then, instead of~\eqref{eq.2ndmom.Qexp_nointegral}, we have
\begin{align}
    E_s^{(2)}(w_s) 
    &\ulesssim
    \ell^{-\frac{3}{2}}
    e^{- \frac{1}{2}(\frac{m_t}{t})^2s}
    t^{\alpha_d}
    e^{-w_s \sqrt{2}}
    \Big( 
    s^{-\frac{1}{2}}
    \int_{\ell^{1/3}}^{\ell^{2/3}}  
    w_{\ell} 
    e^{ - \frac{w_{\ell}^2}{2\ell}}
    e^{- \frac{(w_s- w_{\ell})^2}{2s}}
    \d w_{\ell}
    \Big)
    \nonumber \\
    &\ulesssim 
    \ell^{-\frac{5}{6}}
    e^{- \frac{1}{2}(\frac{m_t}{t})^2s}
    t^{\alpha_d}
    e^{-w_s \sqrt{2}} \, ,
    \label{eq.2ndmom.Qexp_ssmall}
\end{align}
uniformly over $s \in (0, \ell^{2/3}]$ and $w$.
Substituting~\eqref{eq.2ndmom.Eexpmiddle} and~\eqref{eq.2ndmom.Qexp_ssmall} into~\eqref{eq.2ndmom.PQintegral} and applying \eqref{eqn:2mom:exponential} and \eqref{eqn:2mom:alphad-exp} as before yields
\begin{align}
    \E_{\x(z)}[\Upsilon_1]
    &\ulesssim 
     \fM_{L,z} e^{- \tl - \ell + s}  
    \ell^{-\frac{5}{6}}
    =: \cP_1(s)
    \, ,
    \label{eqn:2mom:P1}
\end{align}
uniformly over $s \in (0, \ell^{2/3}]$.
Now, from~\eqref{eqn:2mom:gir}, we find 
\begin{align}
    &\frac{\int_{0}^{\tll} e^{\tll+s} \P_{\x(z)}\big( G_{L,t}(v) \cap G_{L,t}(v') \big) \d s}{\fM_{L,z}}
    \nonumber \\
    &\ulesssim 
    \frac{\int_{0}^{\ell^{2/3}} e^{\tll+s} \cP_1(s) \d s + \int_{\ell^{2/3}}^{\tl-\ell-\ell^{2/3}} e^{\tll+s} \cP_2(s) \d s+
    + \int_{\tl-\ell-\ell^{2/3}}^{\tll} e^{\tll+s} \cP_3(s) \d s}{\fM_{L,z}}
    \nonumber\\
    &\ulesssim
    \ell^{2/3}\ell^{-5/6} + \int_{\ell^{2/3}}^{\tl-\ell-\ell^{2/3}} \frac{t^{3/2}}{(sj_s)^{3/2}} \d s + \ell z^{-1} 
  \ulesssim \ell^{-1/6} + \ell^{-1/3} + \ell z^{-1}
   \ulesssim \ell^{-1/6}\,.
\end{align}
This gives~\eqref{eq.2ndmom.goal}, which concludes the proof.  
\end{proof}

\section{The two-dimensional case}
\label{sec-2d}
Throughout Sections~\ref{sec:window} and \ref{sec:1stmom}, we assumed $d \geq 3$. Then, $\alpha_d -\alpha_d^2 \leq 0$, and thus whenever we applied the Girsanov transform~\eqref{eq.girsanov},  
the $\exp(\cdot)$ factor could be bounded above by $1$. This was  the only situation in which we used the $d\geq 3$ assumption. In the $d=2$ case, this exponential factor in~\eqref{eq.girsanov} is handled by Proposition~\ref{prop:2d-lowerbarrier} below, which shows  
that the probability that there exists a particle in $\cN_t$ that 
reaches height $m_t(2)+y$ after falling, at some  $s \in [L, t]$, below a line $\cL(\cdot)$ of fixed, small slope  is $o_u(1)$ (recall \eqref{def:m_t-c_d}   and  Section~\ref{subsec:asymp-notation} for notation); then, the exponential Girsanov term becomes
$\exp( (\alpha_d-\alpha_d^2)\int_L^t \cL(u)^{-2}/2 \d u) \usim 1$.

For each $t>0$ and $v \in \cN_t$, couple $R_t^{(v)}$ with a two-dimensional Brownian motion $\mathcal{W}_t^{(v)}:=\big(W_t^{(v,1)}, W_t^{(v,2)}\big)$ such that $R_t^{(v)} = \|\mathcal{W}_t^{(v)}\|$. We will prove Proposition~\ref{prop:2d-lowerbarrier} by studying the behavior of $W_t^{(v,1)}$ and $W_t^{(v,2)}$.
Indeed, the following result (Lemma~\ref{lem:2d-lb}, depicted in Figure~\ref{fig:lemma-2d-lower-bound}) provides a crucial bound on the probability that a single coordinate is too low on a discrete set of times.

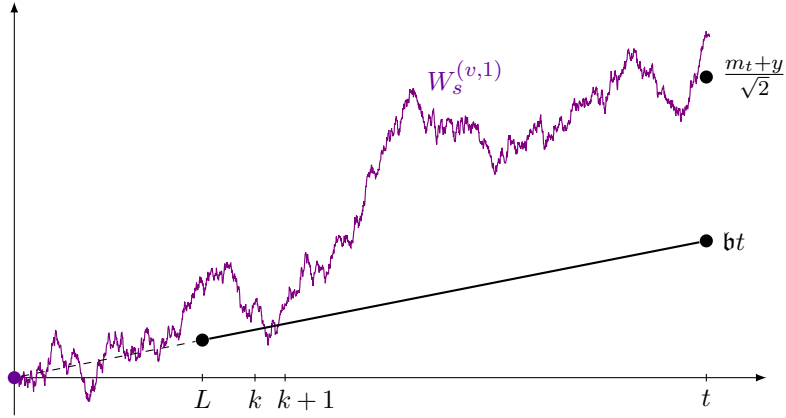
\begin{figure}
     \begin{tikzpicture}[>=latex,font=\small]
 \draw[->] (0, 0) -- (10, 0); 
  \draw[->] (0, -.5) -- (0, 5);
  \node (fig1) at (4.65,1.95) {
  \includegraphics[width=.67\textwidth]{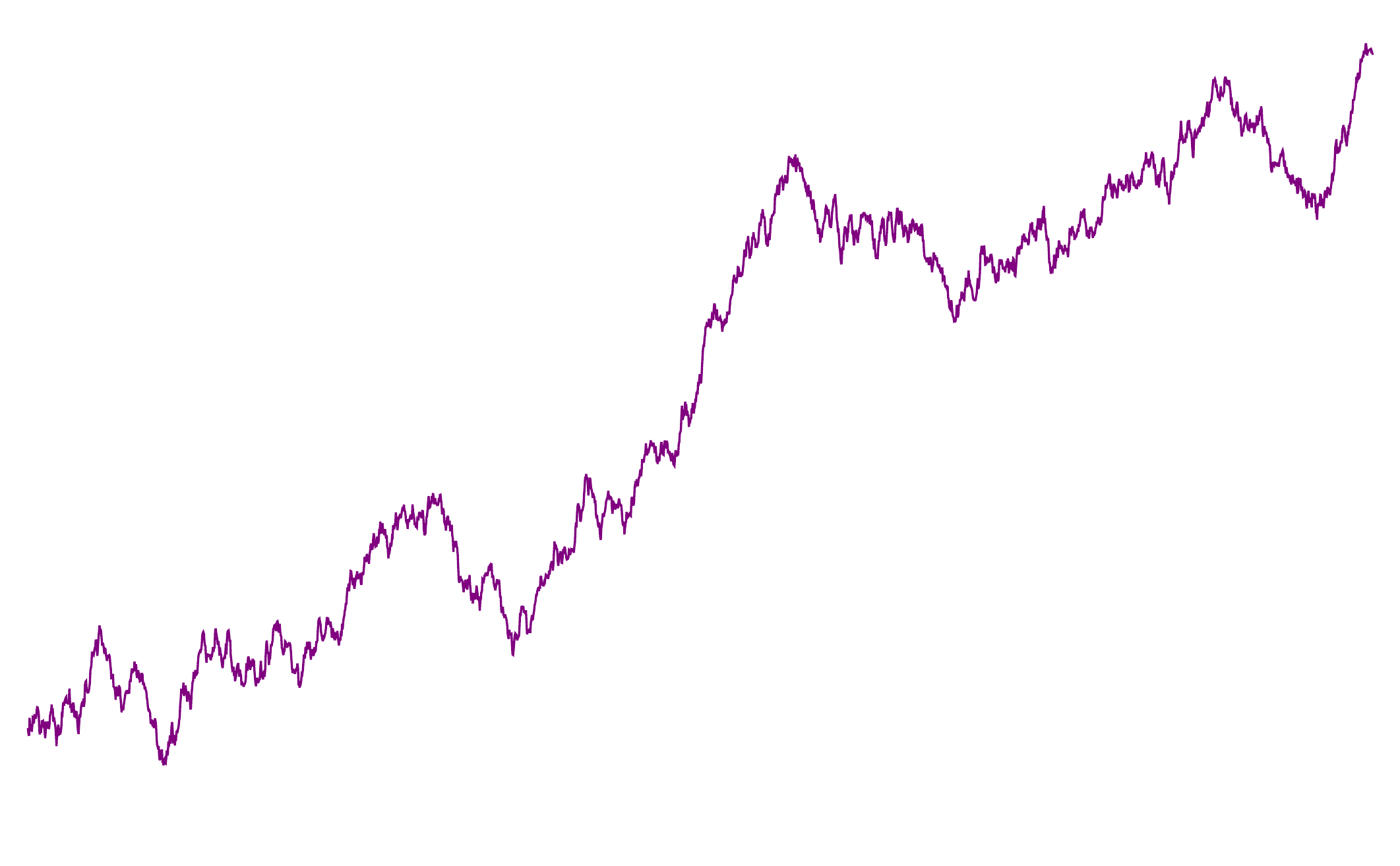}};
   \node[circle,fill=blue!45!purple,inner sep=1.75pt] (o) at (0, 0) {};
   \node[circle, fill=black, inner sep=1.75pt] (a1) at ( 2.5, 0.5 ) {};
  \node[circle, fill=black, inner sep=1.75pt,label={right:$\mathfrak{b}t$}] (a2) at ( 9.2, 1.82 ) {};
  \node[circle, fill=black, inner sep=1.75pt,label={right:$\frac{m_t+y}{\sqrt 2}$}] (y) at ( 9.2, 4 ) {};
  \node[blue!45!purple] at (6,4) {$W_s^{(v,1)}$};
  \draw[dashed](o)--(a1);
  \draw[thick](a1)--(a2);
  \draw (2.5,0.05) -- (2.5,-0.05) node[below] {$L$};
  \draw (3.2,0.05) -- (3.2,-0.05) node[below] {$k$};
  \draw (3.6,0.05) -- (3.6,-0.05) node[below,xshift=8pt] {$k+1$};
  \draw (9.2,0.05) -- (9.2,-0.05) node[below] {$t$};
  \end{tikzpicture}
     \caption{An event in the union in Lemma~\ref{lem:2d-lb}: a one-dimensional Brownian motion $W_s^{(v,1)}$ (purple) dips below the line $\cL(s) = \ep s$ (black) at some point on $[k, k+1]$, for $k \in [L,t]$, while still ending above $(m_t+y)/\sqrt{2}$ at time $t$.}
     \label{fig:lemma-2d-lower-bound}
 \end{figure}
 
\begin{lemma}\label{lem:2d-lb}
Define the interval $I_{[L,t-1]} := \llbracket L, t-1\rrbracket+t-1-\lfloor t-1\rfloor$  (this is the set $\{t-1 - \Z\} \cap [L,t-1]$). There exists a constant $\ep \in (0,1)$ such that for any $y \in \R$,
\begin{align}
    \P \bigg(
     \bigcup_{v \in \cN_t}\bigcup_{k\in I_{[L,t-1]}} \big \{W_k^{(v,1)} \in [0,\ep k], ~W_t^{(v,1)} > \frac{m_t+y}{\sqrt{2}}, ~R_t^{(v)} > m_t+y\big\}  
    \bigg) = o_u(1) \,.
    \label{eqn:2d-lb}
\end{align}
\end{lemma}

\begin{proposition}
\label{prop:2d-lowerbarrier}
Let $\ep$ be the constant from Lemma~\ref{lem:2d-lb}. For all $y \in \R$, 
\begin{align}
    \P \bigg(
    \bigcup_{s \in [L,t]} \bigcup_{v \in \cN_t} \big \{R_s^{(v)} \leq  \tfrac12 \ep s,~R_t^{(v)} > m_t+y\big\}  
    \bigg) = o_u(1) \,.
    \label{eqn:2d-lowerbarrier}
\end{align}
\end{proposition}

Proposition~\ref{prop:2d-lowerbarrier} implies that throughout Section~\ref{sec:window}, it suffices to consider (branching) Bessel particles that stay above $\cL(s) := \ep s$ on $[L,t]$, and that in Section~\ref{sec:1stmom}, it suffices to consider particles that stay above $\cL(s+L)$ on $[0, \tl]$. Thus, the $\exp(\cdot)$ factor in~\eqref{eq.girsanov} is $\usim 1$, and we can carry out the work in these sections essentially in the same way as for $d\geq 3$. We conclude by proving first the proposition, then the lemma.

\begin{proof}[Proof of Proposition~\ref{prop:2d-lowerbarrier}]
Set $\bar{\ep} := \ep/2$. 
Note that for any $0 < s \leq t$ and for any $v \in \cN_t$, 
\begin{align*}
    \P\bigg(\{ R_s^{(v)} \leq \bep s, ~R_t^{(v)} > m_t+y \} \bigg) \leq \P\Big(\bigcup_{i\in\{1,2\}}\mathcal{H}_i \Big)
    \leq 2\P(\mathcal{H}_1)
    \end{align*}
where
\[ \mathcal{H}_i = \Big\{|W_s^{(v,i)}|\leq \bep s, ~|W_t^{(v,i)}| > \frac{m_t+y}{\sqrt{2}}, ~R_t^{(v)}>m_t+y  \Big\}\,,\]
and in turn,
\begin{align}
    \P(\mathcal{H}_1) \leq 4 \P \bigg(
    \bigcup_{s \in [L,t]} \bigcup_{v\in \cN_t} \Big\{W_s^{(v,1)} \in [0,\bep], ~W_t^{(v,1)} > \frac{m_t+y}{\sqrt{2}}, ~R_t^{(v)}>m_t+y  \Big\}
    \bigg)
    \label{eqn:2d-lb:union}
\end{align}
where the inequality is by a union-bound and the reflection principle bound the above.
For each $s\in I_{[L,t-1]}$ we may increase the event $\{W_s^{(v,1)}\in[0,\bep s]\}$ to
$\{W_s^{(v,1)}\in[0,\ep s]\}$ and use Lemma~\ref{lem:2d-lb} followed by a union-bound to conclude that the right-hand of~\eqref{eqn:2d-lb:union} is at most
\begin{align}
    o_u(1)+ 
    4\sum_{k \in I_{[L,t-1]}} \P \bigg( \bigcup_{v \in \cN_t} \Big\{ \min_{s \in [k,k+1]} W_s^{(v,1)} \leq \bep (k+1) \Big\} 
    \cap \bigcap_{j\in\{k,k+1\}} \big\{ W_j^{(v,1)} > j \big\} \bigg)\,. 
    \label{eqn:2d-lb:discrete}
\end{align}
For each $k$, the event in the probability in the right-hand of~\eqref{eqn:2d-lb:discrete} 
only considers the process on the time interval $[0,k+1]$, whence the union over the particles $v\in \cN_t$ is equivalent to a union over $v\in \cN_{k+1}$. As usual, a union bound along with the many-to-one lemma (Lemma~\ref{lem:many-to-one}) imply that the expression in~\eqref{eqn:2d-lb:discrete} bounded above by 
\begin{align*}
    &o_u(1) + 4 \sum_{k \in I_{[L,t-1]}}e^{k+1} \P \bigg( \Big\{ \min_{s \in [k,k+1]} W_s^{(v,1)} \leq \bep(k+1) \Big\} 
    \cap \bigcap_{j\in\{k,k+1\}} \big\{ W_j^{(v,1)} > j \big\} \bigg)  \\
    &\leq 
    o_u(1) + 4 \sum_{k \in I_{[L,t-1]}}e^{k+1} \P_{k,1}^{k+1} \bigg( \min_{s \in [0, 1]} W_s^{(v,1)} \leq \bep(k+1) \Big)\,.
\end{align*}
From the Brownian ballot theorem \eqref{eqn:brownian-ballot-explicit}, the last line of the last display is equal to
\begin{align*}
    o_u(1) + 8 \sum_{k \in I_{[L,t-1]}}e^{k+1} e^{ -2 \big( (1-\bep)k(k+1) -2\bep (1-\bep)(k+1)  \big) } = o_u(1)\,.
\end{align*}
Thus, we have~\eqref{eqn:2d-lowerbarrier}.
\end{proof}

\begin{proof}[Proof of Lemma~\ref{lem:2d-lb}]
In what follows, we write  $\fm_t:= m_t+y$ for brevity; oftentimes, terms involving $y$ are bounded above by a constant.
For $v \in \cN_t$ and $k \in I_{[L,t-1]}$, define the event
\[
\mathbf{Dip}_k(v) := \big \{W_k^{(v,1)} \in [0,\ep k], ~W_t^{(v,1)} > \tfrac{\fm_t}{\sqrt{2}}, ~R_t^{(v)} > \fm_t\big\}\,,
\]
in terms of which we aim to show that $\P(\bigcup_{k\in I_{[L,t-1]}}\bigcup_{v\in \cN_t}\mathbf{Dip}_k(v))=o_u(1)$.
Via union bound, it suffices to show, for $i \in \{1,2\}$, 
\begin{align}
    \P \bigg( \bigcup_{v \in \cN_t}\bigcup_{k\in I_{[L,t-1]}}
    \mathbf{Dip}_k^i(v) \bigg)=o_u(1)\,,
    \label{eqn:2dlb:goal}
\end{align}
where 
\begin{align*}
    \mathbf{Dip}_k^1(v) &:= \{ W_k^{(v,1)} \in [0,\ep k], ~ W_t^{(v,1)} \in  \big(\tfrac{\fm_t}{\sqrt{2}}, \fm_t-1 \big], ~R_t^{(v)} > \fm_t \}\,,  \quad \text{and}
    \nonumber \\
    \mathbf{Dip}_k^2(v) &:= \{ W_k^{(v,1)} \in [0,\ep k], ~ W_t^{(v,1)} > \fm_t-1 \}
    \,.
\end{align*}
We begin by showing \eqref{eqn:2dlb:goal} for $i =1$.
As usual, a union bound and the many-to-one lemma (Lemma~
\ref{lem:many-to-one}) yield 
\begin{align}
    \P \bigg( \bigcup_{v \in \cN_t}\bigcup_{k\in I_{[L,t-1]}}
    \mathbf{Dip}_k^1(v) \bigg)
    &\leq 
    e^t \sum_{k \in I_{[L,t-1]}}  \P \big( \mathbf{Dip}_k^1(v) \big)\,.
\end{align}
Let us now study $\P(\mathbf{Dip}_k^1(v))$.
Integrating over $W_k^{(v,1)}$ and $w_t := W_t^{(v,1)} - \frac{\fm_t}{\sqrt{2}}$, applying the Markov property at times $k$ and $t$, and applying the symmetry relation $\P(|W_t^{(v,2)}| > x) = 2 \P(W_t^{(v,2)}>x)$ for any $x \geq 0$ gives  
\begin{align}
    \P\big(  \mathbf{Dip}_k^1(v)  \big)&= 2e^t \int_0^{\ep k} \d w_k 
    ~p_k^W(0, w_k)  
    \int_0^{(1- \frac{1}{\sqrt{2}})\fm_t-1} \d w_t 
    ~p_{t-k}^W \big(w_k, \frac{\fm_t}{\sqrt{2}}+w_t \big)
    \nonumber \\
    &\qquad \qquad \times \P\pth{W_t^{(v,2)} > \sqrt{\fm_t^2 - \big(\tfrac{\fm_t}{\sqrt{2}}+w_t\big)^2}}\,.
    \label{eqn:2dlb:main-2}
\end{align}
Expanding the last two terms in the integrand of~\eqref{eqn:2dlb:main-2} gives 
\begin{align}
    &p_{t-k}^W \bigg(w_k, \frac{\fm_t}{\sqrt{2}}+w_t \bigg)
    \uasymp (2\pi(t-k))^{-\frac{1}2} \exp \bigg( - \frac{\left( t - \frac{1}{2}\log t + (w_t - w_k) \right)^2}{2(t-k)} \bigg) 
    \nonumber \\
    &\uasymp (2\pi(t-k))^{-\frac{1}2} 
    \exp \pth{\frac{t\log t - t^2}{2(t-k)} - \frac{w_t^2}{2(t-k)}-w_t +\frac{\log t-2(k-w_k)}{2(t-k)} w_t} \,,
    \label{eqn:2dlb:t-term}
\end{align}
uniformly over $w_t$ and $w_k$,
where in the last line we have used
\[
\exp \pth{ - \frac{\frac{1}{4}(\log t)^2+ w_k^2 + w_k \log t}{2(t-k)}} \leq 1\,, 
\]
as well as 
\begin{align}
    \P\pth{W_t^{(v,2)} > \sqrt{\fm_t^2 - \big(\frac{\fm_t}{\sqrt{2}}+w_t\big)^2}}
    &\ulesssim 
    \Big( \Big(1- \frac{1}{\sqrt{2}}\Big)m_t-w_t \Big)^{-\frac{1}{2}} t^{-\frac{1}{2}} e^{-\frac{t}{2}} e^{w_t + \frac{w_t^2}{2t}} \,.
    \label{eqn:2dlb:w2-term}
\end{align}
Since
\begin{align}
\exp\Big(\! -\frac{t^2}{2(t-k)} - \frac{t}{2} - \frac{w_t^2}{2(t-k)} + \frac{w_t^2}{2t} \Big) \leq \exp \Big(\! -t - \frac{kt}{2(t-k)} \Big)\,,
    \label{eqn:2dlb:t-exp}
\end{align}
we have the  crude bound
\begin{align}
     \P\pth{W_t^{(v,2)} > \sqrt{\fm_t^2 - \big(\frac{\fm_t}{\sqrt{2}}+w_t\big)^2}} p_{t-k}^W \bigg(w_k, \frac{\fm_t}{\sqrt{2}}+w_t \bigg)
    &\ulesssim
    e^{-t+\frac{3 t\log t - 2kt}{4(t-k)}}
    \label{eqn:2dlb:crude-prod}
\end{align}
(we have taken $w_t \leq t/2$ and bounded several terms with negative exponent by $1$).
Then for $k \geq 6\log t$, we have 
$\P\big(  \mathbf{Dip}_k^1(v)  \big)\ulesssim t^{-5/4} e^{-t}$ uniformly over such $k$, so that 
\begin{align}
    e^t \sum_{\substack{k \in I_{[L,t-1]} \cap [6 \log t, t-1] }} \P\big(  \mathbf{Dip}_k^1(v)  \big) \ulesssim t^{-1/4}\,. 
    \label{eqn:2dlb-main2-logt-sum}
\end{align}
We now turn our attention to the sum over $k \leq 6 \log t$ of $\P\big(  \mathbf{Dip}_k^1(v)  \big)$. 
Over this range of $k$, we have $\frac{t}{t-k} = 1 + O(\frac{\log t}{t})$, and so from~\eqref{eqn:2dlb:t-term}---\eqref{eqn:2dlb:t-exp}, we find that $\P\big(  \mathbf{Dip}_k^1(v)  \big)$ is asymptotically bounded above (in the sense of $\ulesssim$, uniformly over $k \leq 6 \log t$) by 
\begin{align}
    &t^{-1} e^{-t-\frac{k}{2}}
    \int_0^{(1- \tfrac{1}{\sqrt{2}})\fm_t-1}
    \frac{e^{ \frac{\log t-2k}{2(t-k)} w_t }}
    {\sqrt{ \big( 1-\tfrac{1}{\sqrt{2}} \big) \fm_t -w_t  \big) }}
    \int_0^{\ep k} p_k^W(0,w_k) e^{\frac{w_k w_t}{t-k}} \d w_k \d w_t
    \nonumber \\
    &= t^{-1} e^{-t-\frac{k}{2}}
    \int_0^{(1- \tfrac{1}{\sqrt{2}})\fm_t-1}
    \frac{  e^{ \frac{\log t-2k}{2(t-k)} w_t + \frac{k w_t^2}{2(t-k)^2}}  }
    {\sqrt{ \big( 1-\tfrac{1}{\sqrt{2}} \big) \fm_t -w_t  \big) }}
    \bigg( \frac{1}{\sqrt{2 \pi k}}
    \int_0^{\ep k} 
    e^{\frac{  (w_k-\frac{k w_t}{t-k})^2  }
    {t-k}} \d w_k  
    \bigg) \d w_t 
    \nonumber \\
    &\ulesssim 
    t^{-1} \exp \Big(  -t-\frac{k}{2} + \frac{t(\log t)}{4(t-k)} + \frac{kt^2}{8(t-k)^2}  \Big) 
    \int_0^{(1- \tfrac{1}{\sqrt{2}})\fm_t-1}
    \frac{1}{\sqrt{ \big( 1-\tfrac{1}{\sqrt{2}} \big) \fm_t -w_t  \big) }} \d w_t
    \nonumber \\
    &\ulesssim t^{-\frac{1}{4}}e^{-t - \frac{3}{8}k}\,,
    \label{eqn:2dlb-main2-smallk}
\end{align}
where in moving from the first to the second line, we have expanded the Gaussian density~$p_k^W$ and completed the square in the exponent of $e$; and in moving from the  second to the third line, we have bounded the parenthetical expression by $1$ and bounded~$w_t$ by $t/2$.
Thus,
\begin{align}
    e^t \sum_{k \in I_{[L,t-1]}\cap[L, 6 \log t]} \P\big( \mathbf{Dip}_k^1(v) \big) \ulesssim t^{-1/4}e^{-\frac{3}{8}L} \,. 
    \label{eqn:2dlb-main2-smallk-sum}
\end{align}
Combining \eqref{eqn:2dlb-main2-logt-sum} and~\eqref{eqn:2dlb-main2-smallk-sum} yields \eqref{eqn:2dlb:goal} for $i =1$.

It remains to consider $i=2$. In this case, when $k \lesssim \log t$, we will make use of a barrier event to add decay.
Recall the barrier $B(s) := B_d(s)$ from~\eqref{def:B(s)}\,. We will make use of
\[
B_4(s) = \sqrt{2}s + \cC_4 \log \big(s \wedge (t-s) \big)_+ + \log L
\]
(the same argument would work using $B_3(s)$, but it will be convenient that $m_t(4)$ has no $\log t$ term since $\fc_4=0$).
Lemma~\ref{lem:rough_parabola} holds (in particular) for a four-dimensional Bessel process, and thus by coupling $W_s^{(v,1)}$ as a coordinate of a four-dimensional Brownian motion, \eqref{eqn:no_parabola_UB} gives
\begin{align}
    \P \bigg( \bigcup_{s \in [L, t]} \bigcup_{v \in \cN_s} \{W_s^{(v,1)} > B_4(s) \}  \} \bigg) = o_u(1)\,.
    \label{eqn:2dlb:parabola-o1}
\end{align}
Hence,
\begin{align}
    \P \bigg( \bigcup_{v \in \cN_t}\bigcup_{k\in I_{[L,t-1]}}
    \mathbf{Dip}_k^2(v) \bigg)
    &\leq  
    \mathbf{I} + \mathbf{II} + o_u(1) \,,
    \label{eqn:2dlb:parabola-unionbd}
\end{align}
where
\begin{align*}
    \mathbf{I} &:= \P \bigg( \bigcup_{v \in \cN_t}\bigcup_{k\in I_{[L,t-1]} \cap [5 \log t,t-1]  } 
    \mathbf{Dip}_k^2(v) \bigg) \,, \quad \text{and}
    \\
    \mathbf{II} &:= \P \bigg( \bigcup_{v \in \cN_t}\bigcup_{k\in I_{[L,t-1]} \cap [ L, 5 \log t] }
    \mathbf{Dip}_k^2(v) \cap \UB{[L,t]}{B_4} \big(  W_.^{(v,1)} \big) \bigg)\,.
\end{align*}
We begin by showing $\mathbf{I} = o_u(1)$.
Union bounds and the many-to-one lemma (Lemma~\ref{lem:many-to-one}) give
\begin{align}
    \mathbf{I} \leq 
     \sum_{\substack{k \in I_{[L,t-1]} \cap [5\log t, t-1]}}
    e^t \P \big( \mathbf{Dip}_k^2(v) \big)\,.
    \label{eqn:2dlb:parabola1-m21}
\end{align}
Let us now examine $\P \big( \mathbf{Dip}_k^2(v) \big)$. By integrating over $W_k$ and applying the Markov property at time $k$, we find
\begin{align}
    \P \big( \mathbf{Dip}_k^2(v) \big) &=
    \int_0^{\ep k}   
    p_k^W(0, w_k) \P_{w_k}(W_{t-k}  > \fm_t-1) \d w_k \,.
    \label{eqn:2dlb:main-1}
\end{align}
Consider the bound
\begin{align*}
    \P_{w_k}(W_{t-k} > \fm_t-1) 
    &\leq 
    \exp \bigg( - \frac{t^2}{t-k} - \frac{t\sqrt{2}}{t-k} \Big(- \frac{1}{\sqrt{2}}\log t - w_k + y-1 \Big) \bigg)
    \\
    &\uasymp t \exp \Big( -t +\sqrt{2}w_k - \frac{k}{t-k} \big (t- \log t - w_k \sqrt{2}  \big) \Big)\,,
\end{align*}
where the last line holds uniformly over $w_k \geq 0$.
Take $\ep \leq \frac{1}{2\sqrt{2}}$. Since $w_k \leq \ep k$, the above yields the following estimate, uniformly over $w_k \geq 0$:
\begin{align}
    \P_{w_k}(W_{t-k} > \fm_t-1)  \ulesssim t e^{-t - \frac{k}{3} + w_k \sqrt{2}}\    \label{eqn:2dlb:main-1-tail}\,.
\end{align}      
Expanding the Gaussian density in~\eqref{eqn:2dlb:main-1}, substituting the bound in~\eqref{eqn:2dlb:main-1-tail}, and completing the square in the exponent of $e$ 
\begin{align*}
    \P \big( \mathbf{Dip}_k^2(v) \big) \ulesssim
    t e^{-t - \frac{k}{3}} \Big( e^k k^{-\frac{1}{2}} \int_0^{\ep k} e^{- \frac{(w_k -\sqrt{2}k)^2}{2k}} \d w_k \Big) \leq t e^{-t - \big ( \frac{1}{3} + \frac{b^2}{2} - \sqrt{2} b \big )k } 
    \leq t e^{-t - \frac{k}{4}}\,,
\end{align*}
where the last inequality holds for all $\ep$ sufficiently small.
Thus,  from~\eqref{eqn:2dlb:parabola1-m21}, we have
\begin{align}
    \mathbf{I} \lesssim \sum_{ k \in I_{[L,t-1]} \cap [5\log t, t-1] } t e^{- \frac{k}{4} }  \ulesssim t^{- \frac{1}{4}} \,.
    \label{eqn:2dlb:main-1:log}
\end{align}

We conclude by showing $\mathbf{II} = o_u(1)$. Union bounds and Lemma~\ref{lem:many-to-one} (many-to-one) give
\begin{align}
    \mathbf{II} \leq 
    e^t \sum_{\substack{k \in I_{[L,t-1]} \cap [L, 5\log t]}}
    \P \Big (\mathbf{Dip}_k^2(v) \cap \UB{[L,t]}{B_4} \big(  W_.^{(v,1)} \big) \Big)\,.
    \label{eqn:2dlb:parabola2-m21}
\end{align}
Integrating over $w_k := B_4(k) - W_k^{(v,1)}$ and $w_t := B_4(t) - W_t^{(v,1)}$ and applying the Markov property at times $k$ and $t$ gives the following expression for $\P \Big (\mathbf{Dip}_k^2(v) \cap \UB{[L,t]}{B_4} \big(  W_.^{(v,2)} \big) \Big)$:
\begin{align}
    &\int_{B_4(k) - \ep k}^{B_4(k)} \d w_k ~p_k^W(0, B_4(k) - w_k) 
    \int_0^{B_4(t) - m_t+1} \d w_t ~p_{t-k}^W( B_4(k) -w_k, B_4(t) - w_t) \nonumber \\
    &\qquad \qquad \times \P_{B_4(k) -w_k, t-k}^{B_4(t) -w_t} \Big( \UB{[0,t-k]}{B_4(\cdot+k)} \big( W_.^{(v,1)} \big) \Big) \,.
    \label{eqn:2dlb:parabola2-integral}
\end{align}
Note that $w_k \uasymp k$ uniformly and $w_t \in [0, \frac{1}{\sqrt{2}}\log t + \log L - y +1]$. We use these bounds throughout the estimates that follow.
Replacing $B_4(\cdot+k)$ with $B_4(\cdot+k)+1$ in the barrier event above and then applying Lemma~\ref{lem:brz-modified} yields 
\begin{align}
    \P_{B_4(k) -w_k, t-k}^{B_4(t) -w_t} \Big( \UB{[0,t-k]}{B_4(\cdot+k)} \big( W_.^{(v,1)} \big) \Big) 
    \ulesssim kt^{-1}(1+w_t)\,,
    \label{eqn:2dlb:parabola2-barrier}
\end{align}
uniformly over $w_k$ and $w_t$. The Gaussian density $p_{k}^W$ in~\eqref{eqn:2dlb:parabola2-integral} is asymptotically equivalent (in the sense of $\uasymp$, uniformly over $w_k$) to 
\begin{align}
    k^{- \frac{1}{4} -\cC_4\sqrt{2} } 
    L^{- \sqrt{2}}
    \exp \bigg(\!\! - k - \frac{w_k^2}{2k} + \Big( \sqrt{2} + \frac{\cC_4 \log k + \log L}{k} \Big) w_k \bigg)\,, 
    \label{eqn:2dlb:gaussk}
\end{align}
while the density function $p_{t-k}^W$ is asymptotically bounded above (in the sense of $\ulesssim$, uniformly over $w_k$ and $w_t$) by 
\begin{align}
    k^{\cC_4 \sqrt{2}} (t-k)^{-\frac{1}{2}} 
    \exp \Big( -(t-k) - \sqrt{2}w_k + \sqrt{2}w_t \Big)\,. 
    \label{eqn:2dlb:gausst}
\end{align}
Substituting the bounds given by \eqref{eqn:2dlb:parabola2-barrier}---\eqref{eqn:2dlb:gausst} into \eqref{eqn:2dlb:parabola2-integral} yields
\begin{align*}
    &\P \Big (\mathbf{Dip}_k^2(v) \cap \UB{[L,t]}{B_4} \big(  W_.^{(v,2)} \big) \Big) \nonumber \\
    &\ulesssim 
    L^{-\sqrt{2}} k^{\frac{1}{2}} t^{-\frac{3}{2}} e^{-t}
    \int_{B_4(k) - \ep k}^{B_4(k)}  
    e^{- \frac{w_k^2}{2k} + \frac{\cC_4 \log k + \log L}{k} w_k } \d w_k
    \int_0^{B_4(t) - m_t+1} 
    (1+w_t) e^{ \sqrt{2} w_t}
    \d w_t 
    \nonumber \\
    &\ulesssim 
    k^{\frac{1}{2}}(\log t) t^{-\frac{1}{2}} e^{-t}
    \int_{B_4(k) - \ep k}^{B_4(k)}  
    e^{- \frac{(w_k - \cC_4 \log k - \log L)^2}{2k} } \d w_k 
   \ulesssim k^{\frac{1}{2}}(\log t) t^{-\frac{1}{2}} e^{-t} \,. 
\end{align*}
Thus, from~\eqref{eqn:2dlb:parabola2-m21}, we see that $\mathbf{II} \ulesssim k^{\frac{1}{2}}(\log t)^2 t^{-1/2} = o_u(1)$. This result along with \eqref{eqn:2dlb:main-1:log} and~\eqref{eqn:2dlb:parabola-unionbd} establish \eqref{eqn:2dlb:goal} for $i =2$, thereby concluding the proof of the lemma. 
\end{proof}

\appendix

\section{List of symbols}

{\renewcommand{\arraystretch}{1.2}
\begin{longtable}[t]{lp{0.73\textwidth}l}
\toprule
\multicolumn{3}{l}{Time parameters for the branching process} \\
\midrule
     $\cN_s$ & the set of particles alive at time $s$ & \S\ref{sec-subsecdef} \\
     \multicolumn{2}{l}{$\cN_r^v := \{ u \in \cN_{s+r} \,:\; u\mbox{ is a descendant of }v\}\,,  \mbox{ for } v\in\cN_s$}& 
     \S\ref{subsec-BDZ-method}
     \\
 $\tl := t- L$ &  &\eqref{eqn:def-ttilde}  \\
 $\ell \in [1,L^{1/6}]$ & any function $\ell(L)$ in this interval going to $\infty$ as $L\to\infty$ & \eqref{eqn:def-ell} \\
 $\ell_1 := \ell^{1/4}$ &  &\eqref{eqn:def-ell1} \\
\midrule
\multicolumn{3}{l}{Window notation}\\
\midrule
     \multicolumn{2}{l}{$I_L^{\win} := [\sqrt{2} L - L^{2/3}, \sqrt{2}L - L^{1/6}]$} & \eqref{eq:def-window} \\
     \multicolumn{2}{l}{$\cN_L^{\win} := \{ v \in \cN_L \,:\; R_L^{(v)} \in I_L^{\win} \}$} & \eqref{eq:def-window} \\
     \multicolumn{2}{l}{$z \in [L^{1/6}, L^{2/3}]$ any function $z(L)$ in this interval such that $\sqrt{2}L-z \in I_L^{\win}$}
      & \eqref{eqn:def-z} \\
\midrule
\multicolumn{3}{l}{Quantities depending on $t, L$, $z$, $y$, and $d$} \\
\midrule
     \multicolumn{2}{l}{$\alpha_d := (d-1)/2$}  &\eqref{def:m_t-c_d} \\
     \multicolumn{2}{l}{$\fc_d := \frac{d-4}{2\sqrt{2}}$}  & \eqref{def:m_t-c_d} \\
      \multicolumn{2}{l}{$m_t := \sqrt{2}t + \fc_d \log t$} 
      &\eqref{def:m_t-c_d} \\
     \multicolumn{2}{l}{$\fM_{L,z} :=  (\sqrt{2}L-z)^{-\alpha_d} ze^{-(z+y)\sqrt{2}}$}  &\eqref{def:M_Lz} \\
     \multicolumn{2}{l}{$\fo_t := \frac{m_t}{t}- \sqrt{2} = \fc_d \frac{\log t}{t}$}  &\eqref{eq:fo_t-def}  \\
     \multicolumn{2}{l}{$\x(a) := \sqrt{2}L-a$} & \eqref{def:xz-yw} \\ 
     \multicolumn{2}{l}{$\y(b) := \frac{m_t}{t}(t-\ell)+y-b$} & \eqref{def:xz-yw} \\
\midrule
\multicolumn{3}{l}{Processes} \\
\midrule
$\P_x(\cdot)$ & law of a process started from $x$ at time $0$ & \S\ref{subsec:hittingP} \\
  $\P_{x,T}^y(\cdot)$ & law of a process started from $x$ at time $0$ and ending at $y$ at time $T$ &
    \S\ref{subsec:hittingP} \\
 $p^X_s(x,y)$ &transition density of a Markov process $X_t$ at time $s$ given $X_0 =x$  &
    \S\ref{subsec:parabola-bounds} \\
\midrule
\multicolumn{3}{l}{Barrier functions and events} \\
\midrule
 $\UB{I}{f}(X_\cdot)$ & $\{X_s \leq f(s), ~\forall s \in I \}$&
   \eqref{eqn:notation-barrier}\\
 $\LB{I}{f}(X_\cdot)$ & $\{X_s \geq f(s), ~\forall s \in I\}$&
   \eqref{eqn:notation-barrier}\\
$f_a^b(s;T)$ & the linear function on $[0,T]$ with $f_a^b(0;T)=a$, $f_a^b(T;T) =b$ & \eqref{eq:def-f_a^b}\\
      $B(\cdot)$ 
    & barrier function defined on $[0,t]$ &\eqref{def:B(s)} \\
     $B_0(\cdot)$ & barrier function defined on $[0, \tll]$ &
     \eqref{def:B0} \\
    \multirow{2}{*}{$Q_z(\cdot)$} & barrier function that is  constant
    on $s \in [0, \ell_1]$, constant on $s \in [\tl -\ell-\ell_1, \tll]$, and far below $\frac{m_t}{t}(s+L)$ on $s\in(\ell_1,\tilde t-\ell-\ell_1)$ (see Fig.~\ref{fig:G-Lt-event})
& \multirow{2}{*}{\eqref{eq.lowerbanana.Qdefn}}\\
    $\B_I(X_\cdot)$ 
    & event: $X_s$ is bounded above by $\frac{m_t}{t}(s+L)+y$ and below by $Q_z(s)$ for $s \in I$
   & \eqref{eqn:def-upper/lowerbarrier} \\
  $\fT_r(v)$ & 
  event: a (branching Bessel) descendant of $v$ in $\cN_r^v$ exceeds $m_t+y$
  & \eqref{def:T-tailevent} \\
  $F_{L,t}(v)$ & event: $\UB{[0,\tll]}{B_0}(R^{(v)}_.)$, $R^{(v)}_{\tll}$ is lower-bounded by $t/\sqrt{d}$, and $\fT_{\ell}(v)$  & \eqref{def:F-event} \\
$G_{L,t}(v)$ & event: $\B_{[0,\tll]}(R^{(v)}_.)$, $R^{(v)}_{\tll}$ lies in a small interval below $\frac{m_t}{t}(\tll)+y$, and $\fT_{\ell}(v)$ & \eqref{def:G-event} \\
  $\Gamma_{L,t}$ & the number of $v \in \cN_{\tll}$ satisfying $F_{L,t}(v)$ & \eqref{def:simplefunctions}  \\
  $\bar{\Lambda}_{L,t}$ & the number of $v \in \cN_{\tll}$ satisfying $G_{L,t}(v)$ & \eqref{def:simplefunctions}  \\
\bottomrule
\end{longtable}
}

%



\bibliographystyle{imsart-number}
\bibliography{bbm}
\end{document}